\documentclass[11pt, a4paper]{amsart}

\usepackage[ansinew]{inputenc}
\usepackage[all]{xy}
\SelectTips{cm}{}
\usepackage[pagebackref,%colorlinks,linkcolor=black,citecolor=black%
  ,pdfborder={0 0 0}
  ,urlcolor=black,a4paper,hypertexnames=false]{hyperref}
\hypersetup{pdfauthor={ilcapovilla}
  ,pdftitle=additivity}

\usepackage{amsfonts}
\usepackage{amssymb, amsmath,amsthm, mathtools}
\usepackage{tabularx, hyperref}
\usepackage{enumerate}

\usepackage{tikz-cd}
\usepackage{comment}

\usepackage[textsize=tiny]{todonotes}
\setuptodonotes{fancyline, color=blue!30}

\makeatletter
\newtheorem*{rep@theorem}{\rep@title}
\newcommand{\newreptheorem}[2]{%
\newenvironment{rep#1}[1]{%
 \def\rep@title{#2 \ref{##1}}%
 \begin{rep@theorem}}%
 {\end{rep@theorem}}}
\makeatother

\newtheorem{intro_thm}{Theorem}
\newtheorem{intro_cor}[intro_thm]{Corollary}

\newtheorem{lemma}{Lemma}[section]
\newtheorem{thm}[lemma]{Theorem}
\newreptheorem{thm}{Theorem}  %for repeating the theorem
\newtheorem{prop}[lemma]{Proposition}
\newreptheorem{prop}{Proposition}  %for repeating the proposition
\newtheorem{cor}[lemma]{Corollary}
\newreptheorem{cor}{Corollary}  %for repeating the corollary

\theoremstyle{definition}
\newtheorem{defn}[lemma]{Definition}

\newreptheorem{quest}{Question}  %for repeating the question

\newtheorem{rem}[lemma]{Remark}

\newtheorem{setup}[lemma]{Setup}
\newtheorem*{rem*}{Remark}
\theoremstyle{definition}

\numberwithin{equation}{section}

\renewcommand{\epsilon}{\varepsilon}
\renewcommand{\tilde}{\widetilde}
\renewcommand{\hat}{\widehat}
\renewcommand{\S}{\mathcal{S}}
\renewcommand{\phi}{\varphi}

\newcommand{\N}{\mathbb{N}}

\newcommand{\R}{\mathbb{R}}

\newcommand{\alt}{\textup{alt}}

\newcommand{\acts}{\curvearrowright}

\newcommand{\id}{\textup{id}}
\newcommand{\Aut}{\textup{Aut}}
\newcommand{\K}{\mathcal{K}}
\newcommand{\Elle}{\mathcal{L}}
\newcommand{\A}{\mathcal{A}}

\newcommand{\Cone}{\mathrm{Cone}}

\newcommand{\vertsubseteq}{\rotatebox[origin=c]{90}{$\subseteq$}}

\newcommand{\twprod}{\mathbin{%
		\ooalign{\raise1.15ex\hbox{$\scriptstyle\sim$}\cr\hidewidth$\times$\hidewidth\cr}%
}}
\newcommand{\twprodboth}{\mathbin{%
		\ooalign{\raise1.15ex\hbox{$\scriptstyle(\sim)$}\cr\hidewidth$\times$\hidewidth\cr}%
}}

\DeclareMathAlphabet{\mathcal}{OMS}{cmsy}{m}{n}

\long\def\forget#1{}

\makeatletter
\@namedef{subjclassname@2020}{%
  \textup{2020} Mathematics Subject Classification}
\makeatother

\begin{document}

\title[(super)additivity of simplicial volume]{On the (super)additivity of simplicial volume}

\author[]{Pietro Capovilla}
\address{Scuola Normale Superiore, Pisa Italy}
\email{pietro.capovilla@sns.it}

\thanks{}

\keywords{bounded cohomology, multicomplexes, aspherical manifolds}
\subjclass[1000]{55N10, 55U10, 57N65}
%55N10          Singular theory
%55U10			Simplicial sets
%57N65			Algebraic topology of manifolds
\begin{abstract}
	We show that the simplicial volume is superadditive with respect to gluings along certain submanifolds of the boundary.
	Our criterion applies to boundary connected sums and 1-handle attachments.
	Moreover, we generalize a well-known additivity result in the case of aspherical manifolds.
	Our arguments are based on new results about relative bounded cohomology and pairs of multicomplexes, which are of independent interest.
\end{abstract}

\maketitle
%%%%%%%%%%%%%%%%%%%%%%%%%%%%%%%%%%%%%%%%%%%%%%%%%%%%%%%%%%

\section{Introduction}

The simplicial volume is a homotopy invariant of manifolds introduced by Gromov in his pioneering paper \cite{Gro82}. If $M$ is an oriented, compact manifold with (possibly empty) boundary, then the simplicial volume $\|M,\partial M\|$ of $M$ is the $\ell^1$-seminorm of the real fundamental class of $M$ (see Section \ref{sec: simplicial volume} for the precise definition).
Unless otherwise stated, every manifold in this paper is assumed to be compact and oriented.

Gromov Additivity Theorem can be used to compute the simplicial volume of manifolds that are obtained by gluing smaller manifolds. 
More precisely, the case of connected sums or gluings along $\pi_1$-injective boundary components with amenable fundamental group was considered in \cite{BBF+14, Kue15}, and the case of gluings along amenable portions of the boundary is considered in \cite{Kue15}.
We refer the reader to the references in \cite{BBF+14} and \cite{Kue15} for applications of Gromov Additivity Theorem in several areas of geometric and low-dimensional topology.
Without aiming to be exhaustive, let us point out that recently it has been exploited to study rigidity properties of generalized graph manifolds \cite{FLS15, CSS19}, to compute the spectrum of simplicial volumes of compact manifolds \cite{HL21} and the simplicial volume of mapping tori of 3-manifolds \cite{BN20}.
The additivity of simplicial volume for gluings along amenable boundaries has also been adapted in \cite{LLM22} to the case of boundedly acyclic boundaries.

Gromov Additivity Theorem, first stated in \cite{Gro82}, has been addressed independently in \cite{BBF+14} and in \cite{Kue15} with two radically different approaches. 
In the former, the authors deduce the additivity of simplicial volume by using techniques from homological algebra developed by Ivanov in \cite{Iva87}. 
In the latter, the proof of the main theorem relies on the theory of multicomplexes, following the original approach of Gromov.
The foundations of the theory of multicomplexes have been recently laid down by the work of Frigerio and Moraschini \cite{FM23}.
The goal of this paper is to exploit Frigerio and Moraschini's framework in order to provide a complete and self-contained proof of Gromov's Additivity Theorems (see Theorem \ref{thm: superadditivity} and Theorem \ref{thm: additivity full boundary components} below). 
A similar approach was developed by Kuessner in \cite{Kue15}, where, however, some problems arise especially in the case when the manifolds involved are not aspherical and the gluings are performed along proper submanifolds of the boundary.
We refer the reader to Remark \ref{rem: invariant chains in Kuessner} and to Subsection \ref{subsec: role of higher homotopy} for a discussion on the relationship between our and Kuessner's framework.
Also observe that \cite{BBF+14} does not deal at all with gluings along proper submanifolds of boundary components. 
\begin{intro_thm}
	\label{thm: superadditivity}
	Let $M_1,\dots, M_k$ be triangulated $n$-manifolds with $\pi_1$-injective boundary, $n \in \N_{\geq 2}$. Let $(S_1^+,S_1^-),\dots,(S_h^+,S_h^-)$ be a pairing of some oriented compact connected pairwise-disjoint $(n-1)$-submanifolds of $\partial M_1\sqcup\dots \sqcup \partial M_k$.
	Let $M$ be the manifold obtained by gluing $M_1,\dots, M_k$ along the pairing.
	Assume that the following conditions hold for every $i\in \{1,\dots,h\}$ and every $j \in \{1,\dots, k\}$:
	\begin{itemize}
		\item[$(1)$] $S_i^\pm$ has an amenable fundamental group;
		\item[$(2)$] $\partial S_i^\pm$ is path-connected, $\pi_1$-injective in the corresponding $\partial M_j$ and the map $\pi_1(\partial S_i^{\pm} \hookrightarrow S_i^{\pm},x)$ is an isomorphism for every $x \in \partial S_i^{\pm}$.
	\end{itemize}
	Then there exists a constant $C_n \in \N_{\geq 1}$ (depending only on the dimension) such that
	\[
	(2 C_n)^3 (n+2)^3\cdot\|M,\partial M\| \geq \|M_1,\partial M_1\| + \dots + \|M_k,\partial M_k\|.
	\]
\end{intro_thm}
\begin{rem*}
	The constant $C_n$ denotes the number of $n$-simplices in a suitable triangulation of an $n$-dimensional prism $\Delta^{n-1}\times [0,1]$ (Definition \ref{defn: constants C_n}). We refer the reader to Remark \ref{rem: recursive formula for C_n} for an explicit formula for $C_n$.
\end{rem*}
The various hypotheses of Theorem \ref{thm: superadditivity} are verified for example by boundary connected sums and in the case of 1-handle attachments.
\begin{intro_cor}
	\label{cor: boundary connected sums}
	Let $M_1$ and $M_2$ be triangulated $n$-manifolds with $\pi_1$-injective boundaries, $n \in \N_{\geq 4}$. Let $M$ denote the boundary connected sum of $M_1$ and $M_2$.
	Then 
	\[
	(2 C_n)^3 (n+2)^3\cdot \|M,\partial M \| \geq \|M_1,\partial M_1\| + \|M_2,\partial M_2\|.
	\]
\end{intro_cor}
\begin{intro_cor}
	\label{cor: 1-handle attachment}
	Let $N$ be a triangulated $n$-dimensional manifold with $\pi_1$-injective boundary, $n \in \N_{\geq 4}$. Let $M$ be the manifold obtained from $N$ by attaching a 1-handle $\mathbb{D}^1\times \mathbb{D}^{n-1}$ along some homeomorphism $\mathbb{S}^0\times \mathbb{D}^{n-1}\rightarrow \partial N$.
	Then
	\[
	(2 C_n)^3 (n+2)^3\cdot \|M, \partial M \| \geq \|N,\partial N\|.
	\]
\end{intro_cor}
Theorem \ref{thm: superadditivity} is deduced in Section \ref{sec: proof of superadditivity} from a more general gluing result, which is more technical (see Proposition \ref{prop: superadditivity}).
Moreover, it is not clear whether additivity could hold in the hypotheses of Theorem \ref{thm: superadditivity}.
However, for boundary connected sums with $n=3$ we do not expect more than superadditivity: the boundary connected sum of two solid tori is homeomorphic to a genus-2 handlebody. In this case, the glued manifold has positive simplicial volume \cite{Buc15}, while the solid tori have vanishing simplicial volume. 

When gluings are performed along entire boundary components and the manifolds involved are aspherical, we have the following additivity result, stated initially in \cite{Kue15} without any assumption on higher homotopy.
\begin{intro_thm}
	\label{thm: additivity full boundary components}
	Let $M_1,\dots, M_k$ be triangulated aspherical manifolds with $\pi_1$-injective aspherical boundary.
	Let $M$ be the manifold obtained by gluing $M_1,\dots, M_k$ along (some of) their boundary components. 
	If the glued boundary components have amenable fundamental group, then 
	\[
	\|M,\partial M\| = \|M_1,\partial M_1\| + \dots + \|M_k,\partial M_k\|.
	\]
\end{intro_thm}
\begin{rem*}
	The version of Gromov Additivity Theorem in \cite{BBF+14} does not cover the case in which there are non-amenable boundary components that remain unglued.
	Therefore, in the context of aspherical manifolds, Theorem \ref{thm: additivity full boundary components} is a stronger result.
	A similar situation arises when considering gluings along boundedly acyclic boundary components \cite[Remark 8.2]{LLM22}.
	However, extending our additivity result to the case of boundedly acyclic boundary components does not appear to be straightforward.
\end{rem*}
The aspherical case, albeit not general, is highly significant, for example, due to Gromov's conjecture, which states that, for closed oriented aspherical manifolds, the vanishing of the simplicial volume implies the vanishing of the Euler characteristic \cite[pag. 232]{Gro93}. The conjecture is still open and some approaches to the problem have been recently explored in \cite{LMR22}. For example, by studying the additivity properties of both the simplicial volume and the Euler characteristic, the authors consider an equivalent version of Gromov's conjecture for compact aspherical manifolds with $\pi_1$-injective aspherical boundary \cite[Question 2.13]{LMR22}.\\

The proof of our additivity results rely on the duality between simplicial volume and \emph{real} bounded cohomology.
Following Gromov's approach in the absolute case \cite{Gro82}, the first and fundamental step in this direction is to relate the relative bounded cohomology of a pair of spaces to the relative bounded cohomology of some pair of multicomplexes.

First of all, let us briefly recall Gromov's construction in the absolute case. 
Given a CW-complex $X$, we can construct the so-called \emph{singular multicomplex} $\K(X)$, whose simplices correspond to singular simplices of $X$ that are injective on the vertices, up to affine symmetries. 
There is a natural map $S\colon |\K(X)|\rightarrow X$, which is a homotopy equivalence \cite[Corollary 2.2]{FM23}.
The size of $\K(X)$ can be reduced without affecting its homotopy type: we can define inductively on the skeleta of $\K(X)$ a submulticomplex $\Elle(X)$ such that every simplex in $\K(X)$ is homotopic to a unique simplex of $\Elle(X)$ relative to the vertices.
The multicomplex $\Elle(X)$ is called the \emph{minimal} multicomplex associated to $X$.
Moreover, the restriction $S\colon |\Elle(X)|\rightarrow X$ of $S$ to $|\Elle(X)|$ is again a homotopy equivalence \cite[Corollary 3.27]{FM23}.
Finally we define $\A(X)$, the \emph{aspherical} multicomplex of $X$, to be the multicomplex obtained from $\Elle(X)$ by identifying simplices sharing the same 1-skeleton, and we denote by $\pi\colon \Elle(X)\rightarrow \A(X)$ the obvious quotient map. 
It turns out that the bounded cohomology of $X$ is isometrically isomorphic to the simplicial bounded cohomology of $\A(X)$ \cite[Corollary 4.24]{FM23}.

Some difficulties arise when dealing with pairs of spaces rather than a single one. 
In fact, given a pair of CW-complexes $(X,A)$, there is an obvious simplicial inclusion $\K(A)\subseteq \K(X)$. 
However, defining the minimal and aspherical multicomplexes for the pair $(\K(X),\K(A))$ becomes more complicated.
For example, the multicomplex $\Elle(A)$ is a submulticomplex of $\Elle(X)$ exactly when every connected component of $A$ is $\pi_n$-injective in $X$, for all $n \in \N_{\geq 1}$ (Proposition \ref{prop: when j_Elle is injective}).
Therefore, there is in general no obvious inclusion of $\Elle(A)$ into $\Elle(X)$.
This issue has been addressed by Kuessner in \cite{Kue15}, where, however, some additional conditions on higher homotopy are implicitly assumed (see Subsection \ref{subsec: role of higher homotopy} for a detailed discussion on this topic).

We say that a CW-pair $(X,A)$ is \emph{good} if $\pi_1(A\hookrightarrow X,x)$ is injective and $\pi_n(A\hookrightarrow X, x)$ is an isomorphism, for every $x \in A$ and $n\in \N_{\geq 2}$ (Definition \ref{defn: good pair}).
We do not assume $X$ or $A$ to be path-connected.
If the pair $(X,A)$ is good, the multicomplex $\A(A)$ can be realized as a submulticomplex of $\A(X)$.
By using the techniques developed in \cite{Fri22, FM23}, we show that the relative bounded cohomology of a good pair is \emph{isometrically} isomorphic to the simplicial bounded cohomology of a pair of aspherical multicomplexes.
\begin{intro_thm}
	\label{thm: relative mapping theorem, isometric version}
	Let $(X,A)$ be a good pair.
	Then, for every $n \in \N$, there is a canonical isometric isomorphism
	\[
	\Phi^n\colon H^n_b(\A(X),\A(A))\rightarrow H^n_b(X,A).
	\]
\end{intro_thm}
Our proof of Theorem \ref{thm: additivity full boundary components} is based on the following result on relative bounded cohomology.
\begin{intro_thm}
	\label{thm: relative bounded cohomology and amenable connected components}
	Let $(X,Y)$ be a good pair of topological spaces and let $A\subseteq Y$ be the union of some connected components of $Y$. Let $B = Y\setminus A$ be the union of the remaining connected components and let $j\colon (X,B)\rightarrow (X,Y)$ denote the inclusion. If every connected component of $A$ has an amenable fundamental group, then the map
	\[
	H^n_b(j)\colon H^n_b(X,Y)\rightarrow H^n_b(X,B)
	\]
	is an isometric isomorphism for every $n \in \N$.
\end{intro_thm}
Theorem \ref{thm: relative bounded cohomology and amenable connected components} is a natural generalization for good pairs of \cite[Theorem 2]{BBF+14} and \cite[Theorem 2.3]{KK15}, where the same result is achieved when $B=\emptyset$. 
This lead us to formulate a version of Gromov Equivalence Theorem (Theorem \ref{thm: equivalence theorem}) which plays an important role in our proof of Theorem \ref{thm: additivity full boundary components}.\\

By removing the assumptions on higher homotopy, we are only able to get a bi-Lipschitz control over the norms. 
In fact, there is a well-defined simplicial map $j_\Elle \colon \Elle(A)\rightarrow \Elle(X)$, which is neither injective nor surjective in general (Proposition \ref{prop: when j_Elle is injective}). 
We denote by $\A_X(A)$ the image of $\Elle(A)$ inside $\A(X)$ under the composition $\pi\circ j_\Elle\colon \Elle(A)\rightarrow \A(X)$.
If $A$ is $\pi_1$-injective in $X$, then $\A_X(A)$ is simplicially isomorphic to $\A(A)$ (Proposition \ref{prop: j_A is a simplicial embedding}).
\begin{intro_thm}
	\label{thm: relative mapping theorem, biLipschitz version}
	Let $(X,A)$ be a CW-pair such that the kernel of the morphism $\pi_1(A\hookrightarrow X, x)$ is amenable for every $x \in A$.
	Then, for every $n \in \N$, there is a constant $C_n\in \N_{\geq 1}$ (depending only on the degree) and a canonical bi-Lipschitz isomorphism of vector spaces
	\[
	\Psi^n\colon H^n_b(\A(X),\A_X(A))\rightarrow H^n_b(X,A),
	\]
	such that, for every $\alpha \in H^n_b(\A(X),\A_X(A))$,
	\[
	(n+2)^{-1}(2C_n)^{-2} \cdot \|\alpha \|_\infty \leq \| \Psi^n(\alpha)\|_\infty \leq 2C_n(n+2)\cdot \|\alpha\|_\infty.
	\]
\end{intro_thm}
The isomorphism $\Psi^n$ is constructed by using the machinery of mapping cones developed by Park in \cite{Par03}. We know that mapping cones do not give in general control over the norms \cite{FP12}. Therefore we are not able to show that $\Psi^n$ is isometric.
Moreover, the bi-Lipschitz constants in Theorem \ref{thm: relative mapping theorem, biLipschitz version} are far from being optimal (see Remark \ref{rem: not optimality of constant}).
The different control over the norms for good and non-good pairs translates in the different estimates for the simplicial volume in Theorem \ref{thm: superadditivity} and Theorem \ref{thm: additivity full boundary components}.
\begin{rem*}
	The fact that the bounded cohomology of a CW-complex is isometrically isomorphic to the bounded cohomology of the aspherical multicomplex $\A(X)$ ensures that the absolute bounded cohomology of a topological space depends only on its fundamental group. This suggests that the bounded cohomology (and consequently, via duality, the simplicial volume) is not overly sensitive to higher homotopy. However, all extensions of Gromov's results to the relative case face the fact that higher homotopy cannot truly be ignored if we want to keep control over the norms.
	In fact the notion of good pairs was already present in Frigerio and Pagliantini' s work \cite{FP12}, as well as in Blank's paper \cite{Bla16}, which exploit Ivanov's techniques.
	It is quite striking that very different techniques seem to require the same additional hypothesis on pair of spaces in order to deal with the relative case.
\end{rem*}
\begin{rem*}
	The notion of good pair is very well suited when gluing aspherical spaces. In fact, given a finite CW-complex $Z$, expressed as a union of two aspherical subcomplexes $X$ and $Y$, we have that $Z$ is aspherical if the pairs $(X,X\cap Y)$ and $(Y,X\cap Y)$ are good \cite[Theorem 3.1]{Edm20}. 
	Building on this fact, we deduce Theorem \ref{thm: additivity full boundary components} from a more general result regarding good pairs (Proposition \ref{prop: additivity full boundary components}).
\end{rem*}

\subsection*{Acknowledgments}
I am grateful to my Ph.D. supervisor Roberto Frigerio for his guidance in writing this paper. I would also like to thank Giuseppe Bargagnati, Federica Bertolotti and Francesco Milizia for useful discussions.

\subsection*{Plan of the paper}
We begin Section \ref{sec: preliminaries} by recalling basic notions on bounded cohomology and simplicial volume.
Then, we introduce pairs of multicomplexes and their homotopy theory, with a particular attention to the singular multicomplex. 
Finally, we introduce an important group action on aspherical multicomplexes.
In Section \ref{sec: bounded cohomology of pairs of multicomplexes} we set some straightforward generalizations of some results in \cite{FM23} in the relative setting.
In Section \ref{sec: pair of multicomplexes for pairs of spaces} we discuss the functorial properties of Gromov's construction of the multicomplexes $\Elle(X)$ and $\A(X)$.
In Section \ref{sec: proof of mapping theorem, isometric} we prove Theorem \ref{thm: relative mapping theorem, isometric version} and we discuss the relationship between our work and Kuessner's work on relative bounded cohomology via multicomplexes \cite{Kue15}.
The proof of Theorem \ref{thm: relative bounded cohomology and amenable connected components} is presented in Section \ref{sec: proof of theorem on amenable connected components}.
In Section \ref{sec: mapping cones} we introduce the technique of mapping cones, which is used in Section \ref{sec: proof of mapping theorem, bilipschitz} to prove Theorem \ref{thm: relative mapping theorem, biLipschitz version}.
In Section \ref{sec: regularity of simplicial actions} we introduce a regularity condition of simplicial actions on multicomplexes.
Section \ref{sec: proof of superadditivity} contains the proof of Theorem \ref{thm: superadditivity}.
In Section \ref{sec: gromov equivalence theorem} we establish a version of Gromov Equivalence Theorem for good pairs, which is used in Section \ref{sec: additivity} to prove Theorem \ref{thm: additivity full boundary components}.

\section{Preliminaries}
\label{sec: preliminaries}

\subsection{Relative bounded cohomology of spaces}
Let $X$ be a topological space. We denote by $(C_\bullet(X),\partial_\bullet)$ and $(C^\bullet(X),\delta^\bullet)$ the \emph{real} singular chain complex and singular cochain complex respectively.
For every $n\in \N_{\geq 0}$ we endow the space $C_n(X)$ with the $\ell^1$-norm 
\[
\bigl \lVert \sum a_i s_i \bigr\rVert_1 = \sum |a_i|.
\]
This norm induces a seminorm, still denoted by $\|\cdot \|_1$, on the singular homology with real coefficients $H_n(X)$.
On the other hand, we endow the space of cochains with the $\ell^\infty$-seminorm.
More precisely, for every $f\in C^n(X)$, we set
\[
\lVert f \rVert_\infty =
\sup \bigl\{ | f(\sigma) |, \; \sigma \mbox{ is a singular $n$-simplex}
\bigr\} \in [0,\infty].
\]
A singular cochain is called \emph{bounded} if $\|f\|_\infty<\infty$.
Since the differential takes bounded cochains to bounded cochains, we can consider the subcomplex of bounded cochains $C^\bullet_b(X)\subseteq C^\bullet(X)$.
The \emph{real bounded cohomology} $H^\bullet_b(X)$ of $X$ is the cohomology of the complex  $C^\bullet_b(X)$. Moreover, the $\ell^\infty$-norm on $C^\bullet_b(X)$ descends to a quotient seminorm on $H^\bullet_b(X)$.

Let $A$ be a subspace of $X$. We denote by $C^\bullet_b(X,A)\subseteq C^\bullet(X)$ the subspace of cochains that vanish on simplices supported on $A$ and we set $C^\bullet_b(X,A)=C^\bullet(X,A)\cap C^\bullet_b(X)$. We denote by $H^\bullet(X,A)$ (resp. $H^\bullet_b(X,A)$) the cohomology of the complex $C^\bullet(X,A)$ (resp. $C^\bullet_b(X,A)$).
We endow $C^\bullet_b(X,A)\subseteq C^\bullet_b(X)$ with the subspace norm, which, in turns, induces the quotient seminorm on $H^\bullet_b(X,A)$.
The well known long exact sequence of the pair for ordinary cohomology also holds for bounded cohomology: in fact the short exact sequence of complexes
\[
0\rightarrow C^\bullet_b(X,A)\rightarrow C^\bullet_b(X)\rightarrow C^\bullet_b(A)\rightarrow 0
\]
induces the long exact sequence
\[
\dots \rightarrow H^{n-1}_b(A)\rightarrow H^n_b(X,A)\rightarrow H^n_b(X)\rightarrow H^n_b(A)\rightarrow \dots
\]
\subsection{Simplicial volume}
\label{sec: simplicial volume}
We recall the definition of simplicial volume \cite{Gro82}.
\begin{defn}
	\label{defn: simplicial volume}
	Let $M$ be an oriented compact connected $n$-manifold (possibly with boundary). The	\emph{simplicial volume} of $M$ is defined as
	\[
	\lVert M, \partial M\rVert = \lVert [M,\partial M] \rVert_1\;,
	\]
	where $[M,\partial M] \in H_n(M,\partial M)\cong \R $ denotes the real (relative) fundamental class of $M$.
\end{defn}
The following classical result encodes the connection between bounded cohomology and simplicial volume.
\begin{prop}[{\cite[Proposition 7.10]{Fri17}}]
	\label{prop: duality principle}
	Let $M$ be an oriented compact connected $n$-manifold (possibly with boundary). Then
	\[
	\lVert M, \partial M\rVert = \max \bigl\{ \langle \phi_b, [M,\partial M] \rangle\;\vert \; \phi_b \in H^n_b(M,\partial M), \; \lVert \phi_b\rVert \leq 1 \bigr\},
	\]
	where $\left\langle \cdot , \cdot \right\rangle$ denotes the Kronecker product between (bounded) cohomology and homology.
\end{prop}

\subsection{Pairs of multicomplexes}
The original definition of multicomplexes, given by Gromov in \cite{Gro82}, is the following: a multicomplex is ``a set $K$ divided into the union of closed affine simplices $\Delta_i$, $i \in I$, such that the intersection of any two simplices $\Delta_i \cap \Delta_j$ is a (simplicial) subcomplex in $\Delta_i$ as well as in $\Delta_j$''.
In other words, a \emph{multicomplex} can be defined as a $\Delta$-complex that is both regular and unordered (see \cite[page 533]{Hat}) i.e. a $\Delta$-complex where the simplices have unordered and distinct vertices, and it is possible for multiple simplices to share the same set of vertices.
Simplicial complexes are in particular multicomplexes.
We refer to \cite{FM23} for a more abstract-algebraic definition of multicomplexes.

The \emph{geometric realization $|K|$} of a multicomplex $K$ is constructed in the very same way as the geometric realization of $\Delta$-complexes. As a topological space, it inherits the weak topology from its decomposition into simplices.
For every $n \in \N$, we denote by $K^n$ the $n$-skeleton of $K$.

A \emph{pair of multicomplexes} is a pair $(K,L)$, where $L\subseteq K$ is a submulticomplex of $K$. A simplicial map between pairs of multicomplexes $f \colon (K,L)\rightarrow (K',L')$ is given by a simplicial map $f\colon K\rightarrow K'$ such that $f(L)\subseteq L'$. We denote by $\Aut(K,L)$ the group of simplicial automorphisms of the pair $(K,L)$.
A simplicial map $f\colon K \rightarrow K'$ between multicomplexes is \emph{non-degenerate} if it maps every $n$-simplex of $K$ to a $n$-simplex of $K'$.

Let $I=[0,1]$ and let $\sigma_n$ denote the standard simplex $\Delta^n$, endowed with its natural structure of simplicial complex (hence multicomplex).
We define a structure of multicomplex on $\sigma_n\times I$ in the following way.
We proceed by induction, for $n=0$ we just take $\sigma_0\times I$ as the first barycentric subdivision of $I$.
Assume that $\sigma_{n-1}\times I$ has been triangulated in such a way that $\sigma_{n-1}\times \{0\}$ and $\sigma_{n-1}\times \{1\}$ appear as simplices in the multicomplex structure of $\sigma_{n-1}\times I$. In order to triangulate $\sigma_n\times I$, we consider the following triangulation of the geometric boundary of $\sigma_n\times I$.
We triangulate $\partial(\sigma_n\times I)$ by merging $\sigma_n\times \{0\}$, $\sigma_n\times \{1\}$ and the triangulations of $\tau\times I\cong \sigma_{n-1}\times I$, where $\tau$ varies among all the facets of $\sigma_n$. In conclusion, the triangulation of $\sigma_n\times I$ is obtained by coning the triangulation of $\partial (\sigma_n\times I)$ over an internal point.
The number of simplices needed to triangulate $\sigma_{n-1}\times I$ plays an important role in this work.
\begin{defn}
	\label{defn: constants C_n}
	Let $n \in \N_{\geq 1}$ and $I = [0,1]$. Let $\sigma_{n}$ denote the standard simplex $\Delta^n$.
	The constant $C_n$ denotes the number of $n$-simplices in $\sigma_{n-1}\times I$, endowed with its natural structure of multicomplex.
\end{defn}
\begin{rem}
	\label{rem: recursive formula for C_n}
	From the construction above it is easy to deduce the following recursive formula for $C_n$:
	\[
	\begin{cases}
		C_1 = 2, \\
		C_{n}=2 + n \cdot C_{n-1},
	\end{cases}
	\]
	from which we obtain the closed formula:
	\[
	C_n=2\cdot \sum_{k=1}^{n}\frac{n!}{k!}.
	\]
\end{rem}
The multicomplex $K\times I$ is the multicomplex obtained by gluing a copy of the multicomplex $\sigma\times I$ for every simplex $\sigma$ of $K$, according to the boundary maps \cite[Definition 3.12]{FM23}. 
If $f,g\colon (K,L)\rightarrow (K',L')$ are simplicial maps between pairs of multicomplexes, we say that \emph{$f$ is simplicially homotopic to $g$ (as a map of pairs)} if there exists a simplicial map $F\colon K\times I \rightarrow K'$ such that $F\circ i_0 = f$, $F\circ i_1 = g$ and $F(L\times I)\subseteq L'$, where $i_j\colon K\rightarrow K\times \{j\}\subseteq K\times I$ is the obvious inclusion. 

Given a group $G$, a simplicial group action $G\acts (K,L)$ is given by a group homomorphism $G\rightarrow \Aut(K,L)$. The notion of quotient multicomplex $K/G$ is not well defined in general \cite[Remark 1.15]{FM23}. However, if a simplicial action $G\acts (K,L)$ is \emph{$0$-trivial}, i.e. if $G$ acts trivially on the 0-skeleton of $K$, then we have a well defined pair of multicomplexes $(K/G,L/G)$ \cite[Proposition 1.14]{FM23}. 

\subsection{Bounded cohomology of multicomplexes}
Let $K$ be a multicomplex. An \emph{algebraic $n$-simplex} of $K$ is a pair
\[
\sigma = (\Delta,(v_0,\dots,v_n)),
\]
where $\Delta$ is a $k$-simplex of $K$, and $\{v_0,\dots,v_n\}$ coincides with the set of vertices of $\Delta$. Since $\Delta$ has exactly $k+1$ vertices, we have that $k\leq n$. There is no requirement for the elements of the ordered $(n+1)$-tuple $(v_0,\dots,v_n)$ to be pairwise distinct. For every $i\in \{0,\dots,n\}$, the $i$-th face of $\sigma$ is the algebraic $(n-1)$-simplex defined by
\[
\partial_n^i\sigma = (\Delta',(v_0,\dots, \hat{v}_i,\dots, v_n)),
\]
where $\Delta'=\Delta$, if $\{v_0,\dots, v_n\} = \{v_0,\dots, \hat{v}_i,\dots, v_n\}$, and $\Delta'$ is the unique face of $\Delta$ with vertices $\{v_0,\dots, \hat{v}_i,\dots, v_n\}$, otherwise.
We denote by $C_n(K)$ the real vector space generated by the set of algebraic $n$-simplices.
Of course, $C_\bullet(K)$ is a chain complex with the following boundary operator
\[
\partial_n\colon C_n(K)\rightarrow C_{n-1}(K), \qquad \partial_n = \sum_{i=0}^{n} (-1)^i\partial_n^i.
\] 
Let $C^\bullet(K)$ (resp. $C^\bullet_b(K)$) be the complex of (bounded) real simplicial cochains on $K$, and let $H^\bullet(K)$ (resp. $H^\bullet_b(K)$) be the corresponding cohomology. As in the singular case, we may endow $C^\bullet_b(K)$ with the $\ell^\infty$-norm, which descends to a seminorm on $H^\bullet_b(K)$.

A cochain $f \in C^n(K)$ is called \emph{alternating} if, for every algebraic simplex $(\Delta, (v_0,\dots, v_n))$ and every permutation $\tau$ of $\{0,\dots, n\}$, we have that 
\[
f(\Delta, (v_0,\dots, v_n)) = \text{sign}(\tau) \cdot f(\Delta, (v_{\tau(0)},\dots, v_{\tau(n)})).
\]
The differential preserves alternating cochains, and therefore we can consider the complex $C^\bullet_b(K)_\alt = C^\bullet_b(K)\cap C^\bullet(K)_\alt$ of alternating bounded cochains.
Since the inclusion $C^\bullet_b(K)_\alt\subseteq C^\bullet_b(K)$ induces an isometric isomorphism in cohomology \cite[Theorem 1.9]{FM23}, every cohomology class in $H^\bullet_b(K)$ may be represented by an alternating cocycle.

If $L$ is a subcomplex of $K$, then the inclusion $C_\bullet(L)\subseteq C_\bullet(K)$ induces the following short exact sequence of complexes
\[
0\rightarrow C^\bullet_b(K,L)\rightarrow C^\bullet_b(K)\rightarrow C^\bullet_b(L)\rightarrow 0.
\]
We denote by $H^\bullet_b(K,L)$ the cohomology of the complex $C^\bullet_b(K,L)$. As in the singular case, $C^\bullet_b(K,L)$ is endowed with the subspace norm from $C^\bullet_b(K)$, which, in turns, induces the quotient seminorm on $H^\bullet_b(K,L)$.
As in the absolute case, every cohomology class in $H^\bullet_b(K,L)$ may be represented by an alternating cocycle.

\begin{rem}
	\label{rem: algebraic homotopy is bounded}
	Is is observed in \cite{FM23} that simplicially homotopic maps are algebraically homotopic via a homotopy which is bounded in every degree. 
	More specifically, for every pair $f_0,f_1\colon L\rightarrow K$ of simplicially homotopic maps between multicomplexes, the induced maps $f_i^\bullet\colon C^\bullet_b(K)\rightarrow C^\bullet_b(L)$, $i \in \{0,1\}$, on bounded cochains are algebraically homotopic via a homotopy \[T^\bullet \colon C^\bullet_b(K)\rightarrow C^{\bullet-1}_b(L)\] such that $\|T^n\|\leq C_n$, for every $n \in \N_{\geq 1}$ \cite[Remark 3.15]{FM23}. 
	Here $C_n$ denotes the number of $n$-simplices in the natural multicomplex structure of the $n$-dimensional prism $\Delta^{n-1}\times I$ (see Definition \ref{defn: constants C_n}).
	In fact, such a homotopy can be constructed by sending each simplex $\sigma$ of $L$ to the sum (with signs) of the simplices triangulating $\sigma\times I$.
\end{rem}

\subsection{The singular multicomplex}
\label{subsec: the singular multicomplex}
In the theory of multicomplexes, the singular multicomplex plays the same role played by the singular simplicial set in the theory of simplicial sets \cite{Mil57,May92}. 
Given a topological space $X$, the \emph{singular multicomplex} $\K(X)$ of $X$ is the multicomplex having as simplices the singular simplices of $X$ with distinct vertices, up to affine symmetries \cite{Gro82} (hence, its simplices do not come with a preferred ordering of the vertices).
The geometric realization of $\K(X)$ is a CW-complex whose 0-skeleton is in bijection with the space $X$ itself. Moreover, there is an obvious map $S_X\colon |\K(X)|\rightarrow X$, which, under the assumption that $X$ is a CW-complex, is a homotopy equivalence \cite[Corollary 2.2]{FM23}.

If $(X,A)$ is a pair of topological spaces, we have an obvious simplicial inclusion $\K(A)\subseteq \K(X)$. Moreover the map $S_X$ induces a well defined map of pairs $S_X\colon(|\K(X)|,|\K(A)|)\rightarrow (X,A)$.
The following result is a straightforward generalization of \cite[Corollary 2.2]{FM23} to pairs of topological spaces. Recall that two continuous maps of pairs $f,g\colon (X,A)\rightarrow (Y,B)$ are \emph{homotopic as maps of pairs} if there exists a continuous homotopy $F\colon X \times [0,1]\rightarrow Y$ such that $F(A\times [0,1])\subseteq B$. A map of pairs is a \emph{homotopy equivalence of pairs} if it has a homotopy inverse as a map of pairs.
\begin{prop}
	\label{prop: S_X:K(X) to X induces homotopy equivalence of pairs}
	Let $(X,A)$ be a CW-pair. Then the natural projection
	$S_X\colon(|\K(X)|,|\K(A)|)\rightarrow (X,A)$
	is a homotopy equivalence of pairs.
\end{prop}
\begin{proof}
	Since ${S_X}|_{A} = S_A$, we have that the horizontal arrows of the following diagram are homotopy equivalences
	\[
	% https://tikzcd.yichuanshen.de/#N4Igdg9gJgpgziAXAbVABwnAlgFyxMJZABgBpiBdUkANwEMAbAVxiRAB8AdTgaQAoAGgEp2IAL6l0mXPkIoyARiq1GLNl158AgiPGSQGbHgJEF5ZfWatEIAXqlHZp0kuqW1NreOUwoAc3giUAAzACcIAFskMhAcCCQAJjdVaxAAZQB9OwkQ8KjEM1j4xABmZKs2TK8ckDDIpEK46OocOiwGNgALCAgAa3tavKQyosSWto6bbr7vMSA
	\begin{tikzcd}
		{|\K(X)|} \arrow[r, "S_X"]                 & X                 \\
		{|\K(A)|} \arrow[r, "S_A"] \arrow[u, hook] & A. \arrow[u, hook]
	\end{tikzcd}
	\]
	Since $(|\K(X)|,|\K(A)|)$ and $(X,A)$ are CW-pairs, then both vertical arrows are cofibrations and we can deduce the statement from \cite[7.4.2]{Bro06}.
\end{proof}
\begin{rem}
	\label{rem: on the importance of pairs}
	We would like to emphasize the role of homotopy equivalence \emph{of pairs} in our context.
	Let $f\colon (X,A)\rightarrow (Y,B)$ be a map of pairs. Requiring $f$ to be a homotopy equivalence of pairs is of course more restrictive than requiring both $f$ and $f_{|A}$ to be homotopy equivalences.
	However, if $f$ is a homotopy equivalence of pairs, it is evident that $f$ induces an \emph{isometric} isomorphism $H^n_b(f)\colon H^n_b(Y,B)\rightarrow H^n_b(X,A)$ for every $n \in \N$.
	On the contrary, if both $f$ and $f_{|A}$ are homotopy equivalences, it follows from the Five Lemma that $H^n_b(f)$ is an isomorphism, but it is not clear how to show it is isometric. 
\end{rem}

\subsection{Complete, minimal and aspherical multicomplexes}
In order to get a better understanding of the homotopy of multicomplexes, one considers \emph{complete multicomplexes}, which correspond to Kan complexes in the theory of simplicial sets \cite{Mil57}. A multicomplex $K$ is complete if every continuous map $f\colon \Delta^n\rightarrow |K|$, whose restriction to the boundary $f|_{\partial \Delta^n}\colon \partial \Delta^n\rightarrow |K|$ is a simplicial embedding, is homotopic in $|K|$ relative to $\partial \Delta^n$ to a simplicial embedding $f'\colon \Delta^n\rightarrow K$ (here $\Delta^n$ is equipped with its natural structure of multicomplex).
For example, for every CW-complex $X$, the singular multicomplex $\K(X)$ is complete \cite[Theorem 3.7]{FM23}.

Let $K$ be a multicomplex. We say that two $n$-simplices $\Delta_1$ and $\Delta_2$ of $K$ are \emph{compatible} if they share the same boundary (in the topological realization of $K$).
For every simplex $\Delta$ of $K$, we denote by $\pi_K(\Delta)$ the set of simplices of $K$ that are compatible with $\Delta$.

The submulticomplex of $K$ generated by the union of two compatible $n$-simplices is naturally homeomorphic to a $n$-sphere, and therefore, for every pair $\Delta_1$ and $\Delta_2$ of compatible $n$-simplices of $K$, we are able to identify a pointed continuous map 
\[
\dot{S}^n(\Delta_1,\Delta_2)\colon (S^n, s_0)\rightarrow (|K|,x_0),
\]
where $s_0$ is a chosen point of $S^n$ which is mapped to a vertex $x_0$ of $\Delta_1$ (or equivalently of $\Delta_2$). We refer the reader to \cite[Section 3.2]{FM23} for more details about this construction.
Two compatible $n$-simplices $\Delta_1$ and $\Delta_2$ of $K$ are \emph{homotopic in $K$} if the corresponding characteristic maps $i_j\colon\Delta^n\rightarrow |K|$, $j\in\{1,2\}$, are homotopic in $|K|$ relative to $\partial \Delta^n$.
A multicomplex is called \emph{minimal} if it does not contain any pair of homotopic simplices.

The advantage of dealing with complete and minimal multicomplexes is that we can describe their homotopy groups in a simplicial way.
\begin{prop}[{\cite[Theorem 3.9 and Proposition 3.22]{FM23}}]
	\label{prop: homotopy groups of complete multicomplexes}
	Let $K$ be a complete multicomplex, and let $\Delta_0$ be an $n$-simplex of $K$. Suppose that $x_0$ is a fixed vertex of $\Delta_0$. Then the map
	\[
	\Theta \colon \pi_K(\Delta_0)\rightarrow \pi_n(|K|,x_0), \qquad \Theta(\Delta) = [\dot{S}^n(\Delta_0, \Delta)],
	\]
	is surjective, and $\Theta(\Delta)=\Theta(\Delta')$ if and only if $\Delta$ and $\Delta'$ are homotopic.
	
	In particular, if $K$ is complete and minimal, the map $\Theta$ is bijective.
\end{prop}

To every complete and minimal multicomplex $K$ there is associated a complete, minimal and \emph{aspherical} multicomplex $A$ which is defined as follows. The 1-skeleton of $A$ coincides with the 1-skeleton of $K$ (in particular, they have the same set of vertices). For every $n\in \N_{\geq 2}$, the set of $n$-simplices of $A$ is given by the equivalence classes of $n$-simplices of $K$, where two $n$-simplices are equivalent if and only if they share the same 1-skeleton. It turns out that $A$ is complete, minimal and (the topological realization of) $A$ is indeed aspherical \cite[Theorem 3.31]{FM23}. One can define a simplicial projection $\pi\colon K\rightarrow A$ that restricts to the identity of $K^1=A^1$, and induces an isomorphism on fundamental groups \cite[Theorem 3.31]{FM23}. We call $A$ the \emph{aspherical quotient} of $K$.

\subsection{Minimal and aspherical multicomplexes of spaces}
\label{subsection: minimal and aspherical multicomplexes of spaces}
Let $X$ be a topological space and consider its singular multicomplex $\K(X)$.
The size of $\K(X)$ can be reduced without changing its homotopy type. 
We define a submulticomplex $\Elle(X)$ of $\K(X)$ in the following way: the 0-skeleton of $\Elle(X)$ coincides with the one of $\K(X)$ (hence with the set $X$); once the $n$-skeleton $\Elle(X)^n$ of $\Elle(X)$ is defined, we define the ($n$+1)-skeleton by adding to $\Elle(X)^n$ one ($n$+1)-simplex for each homotopy class of ($n$+1)-simplices of $\K(X)$ whose facets are contained in $\Elle(X)^n$.
By construction, the multicomplex $\Elle(X)$ is complete and minimal.
Moreover, one can define a simplicial retraction $\K(X)\rightarrow \Elle(X)$ whose topological realization is a strong deformation retraction \cite[Theorem 3.23]{FM23}. It follows that the inclusion $|\Elle(X)|\hookrightarrow|\K(X)|$ is a homotopy equivalence.
We denote by $\A(X)$ the aspherical quotient of $\Elle(X)$, so that its topological realization $|\A(X)|$ of $\A(X)$ is a model for the classifying space of the fundamental group of $X$.
A fundamental result of Gromov \cite[Section 3.3]{Gro82} \cite[Corollary 4.24]{FM23} affirms that the bounded cohomology of $X$ is isometrically isomorphic to the simplicial bounded cohomology of $\A(X)$. The functorial properties of this construction are given in Section \ref{sec: pair of multicomplexes for pairs of spaces}.

\subsection{The group $\Pi(X,X)$}
\label{subsec: the group Pi(X,X)}
Let $X$ be a topological space. First of all, we recall the definition of the group $\Pi(X,X)$, introduced by Gromov in \cite{Gro82}.
\begin{defn}
	\label{defn: group Pi(X,X)}
	Let $X_0$ be a subset of $X$.
	Let $\Omega(X,X_0)$ be the set whose elements are families of paths $\{\gamma_x\}_{x \in X_0}$ satisfying the following conditions:
	\begin{itemize}
		\item[(1)] for every $x \in X_0$, $\gamma_x\colon [0,1]\rightarrow X$ is a continuous path such that $\gamma_x(0)=x$ and $\gamma_x(1)\in X_0$;
		\item[(2)] the path $\gamma_x$ is constant for all but finitely many $x \in X_0$;
		\item[(3)] the map
		\[
		X_0\rightarrow X_0, \qquad x\mapsto\gamma_x(1),
		\]
		is a bijection (hence a permutation of $X$ with finite support).
	\end{itemize}
	Two elements $\{\gamma_x\}_{x \in X_0}$ and $\{\gamma'_x\}_{x \in X_0}$ of $\Omega(X,X_0)$ are said to be \emph{homotopic} if $\gamma_x$ is homotopic to $\gamma_x'$ in $X$ relative to the endpoints, for every $x \in X_0$.
	We denote by $\Pi(X,X_0)$ the set of homotopy classes of elements of $\Omega(X,X_0)$. This is a group with respect to the usual concatenation of paths.
\end{defn}
We usually denote elements of $\Pi(X,X)$ just by specifying the list of homotopically non-trivial paths $\{\gamma_1,\dots,\gamma_n\}$ in one of its representatives.
If $X_0=\{x_0\}$ consists of a single point, then $\Pi(X,X_0)=\pi_1(X,x_0)$ is the fundamental group of $X$ at the basepoint $x_0$. In general, there is an injective group homomorphism
\[
\bigoplus_{x \in X_0} \pi_1(X, x)\hookrightarrow \Pi(X,X_0).
\]
For every subset $A$ of $X$, the inclusion $A\hookrightarrow X$ induces an obvious group homomorphism
\[
\Pi(A,A)\rightarrow \Pi(X,X).
\]
We denote by $\Pi_X(A)$ the image of this homomorphism in $\Pi(X,X)$.

A subset $A$ of a topological space $X$ is called \emph{amenable} (in $X$) if, for every $x \in A$, the image of $\pi_1(A\hookrightarrow X,x)$ is an amenable subgroup of $\pi_1(X,x)$. The set $A$ is not assumed to be path-connected. 
\begin{lemma}[{\cite[Lemma 6.6]{FM23}, \cite[Lemma 4]{Kue15}}]
	\label{lem: Pi(A,A) is amenable}
	Let $(X,A)$ be a pair of topological spaces. If $A$ is amenable in $X$, then the subgroup $\Pi_X(A)\leq\Pi(X,X)$ is amenable.
\end{lemma}
If $A$ is $\pi_1$-injective in $X$, then $\Pi_X(A)$ is isomorphic to $\Pi(A,A)$.
Moreover, in this case, $A$ is an amenable subset of $X$ if and only if every connected component of $A$ has an amenable fundamental group.

\subsection{The action of $\Pi(X,X)$ on $\A(X)$}
\label{subsec: the action of Pi(X,X) on A(X)}
We define a simplicial action of $\Pi(X,X)$ on $\A(X)$ as follows.
Let $g\in \Pi(X,X)$ and let $\{\gamma_x\}_{x \in X}$ be a representative of $g$. 
On the 0-skeleton, we simply define the action of $g$ as the permutation (with finite support) induced by $g$ on $X=\A(X)^0$.
Let $e$ be a 1-simplex of $\A(X)$ with endpoints $v_0,v_1\in \A(X)^0=X$.
Recall that 1-simplices of $\Elle(X)$ (hence, of $\A(X)$) bijectively correspond to homotopy classes (relative to endpoints) of paths in $X$ with distinct endpoints. Therefore, we can take a representative $\gamma_e \colon [0,1] \rightarrow X$ of $e$ with endpoints $v_0$ and $v_1$ and we define $g\cdot e$ as the homotopy class (relative to endpoints) of the following concatenation of paths
\[
\bar{\gamma}_{v_0} * \gamma_e * \gamma_{v_1},
\]
where $\bar{\gamma}$ is the path $\bar{\gamma}(t)=\gamma(1-t)$.
In \cite[Section 5.2]{FM23} it is showed that we can extend uniquely $g$ to a simplicial automorphisms of the whole $\A(X)$, denoted by $\psi(g)$, which is simplicially homotopic to the identity \cite[Theorem 5.3]{FM23}.
In this way we have an action $\psi \colon \Pi(X,X)\rightarrow \Aut(\A(X))$.

\subsection{Universal covering of multicomplexes}
\label{subsection: universal covering of multicomplexes}

Let $K$ be a connected multicomplex and let $\widetilde{|K|}\rightarrow |K|$ denote the universal covering map.
Then $\widetilde{|K|}$ has a unique multicomplex structure that makes the covering map simplicial. 
We denote by $\widetilde{K}$ the corresponding multicomplex and by \[p\colon \widetilde{K} \rightarrow K\] the simplicial projection.
The universal covering of an aspherical multicomplex is clearly aspherical. 
Also completeness and minimality are inherited by universal coverings.
\begin{lemma}
	\label{lem: universal covering is complete and minimal}
	If $K$ is complete and minimal, then also $\widetilde{K}$ is complete and minimal.
\end{lemma}
\begin{proof}
	Let $f \colon \Delta^n \rightarrow |\tilde{K}|$ be a continuous map such that $f|_{\partial \Delta^n}$ is a simplicial embedding. 
	Then also the map $|p|\circ f\colon \Delta^n \rightarrow |K|$ restricts to a simplicial embedding on $\partial \Delta^n$. 
	Therefore, by completeness and minimality of $K$, there exists a unique simplex $\sigma \in K$ such that $|p|\circ f$ is homotopic in $|K|$ relative to $\partial \Delta^n$ to a characteristic map of $\sigma$. 
	By the homotopy lifting property of covering spaces, it follows that there exists a simplex $\tilde{\sigma}\in \tilde{K}$ such that $f$ is homotopic in $|\tilde{K}|$ relative to $\partial \Delta^n$ to a characteristic map of $\tilde{\sigma}$. By the unique lifting property, the simplex is also unique, and therefore $\tilde{K}$ is complete and minimal.
\end{proof}
The following observation is implicit in \cite[Section 1.7]{Kue15}.
\begin{lemma}
	\label{lem: uniqueness of simplices in universal covering}
	Let $K$ be a connected, complete and minimal multicomplex.
	Let $v,w \in \tilde{K}^0$ be distinct vertices of $\tilde{K}$. Then there exists a unique edge $e_{vw}\in \widetilde{K}^1$ with endpoints $v$ and $w$.
\end{lemma}
\begin{proof}
	Since $|\tilde{K}|$ is connected, there exists a path $\gamma\colon [0,1]\rightarrow |\tilde{K}|$ between $v$ and $w$. By completeness and minimality of $\tilde{K}$, there exists a unique edge $e_{vw}$ of $\tilde{K}$ that is homotopic to $\gamma$ relative to the endpoints. Since $\tilde{K}$ is simply connected, there is only one homotopy class of paths in $|\tilde{K}|$ joining $v$ and $w$. Hence, by minimality, $e_{vw}$ is the unique edge joining $v$ and $w$.
\end{proof}
We denote by $\pi_1(K, v) = \pi_1(|K|,v)$ the fundamental group of $|K|$ with respect to some basepoint $v\in K^0$.
The action of $\pi_1(K,v)$ on $|\widetilde{K}|$ by deck transformations induces a simplicial action $\pi_1(K,v)\acts \tilde{K}$.

\section{Bounded cohomology of pairs of multicomplexes}
\label{sec: bounded cohomology of pairs of multicomplexes}
In this section we state and prove some straightforward generalizations of some results from \cite{FM23}. 
Due to the reliance of our arguments on the actual proofs in \cite{FM23}, this section is far from self contained and we recommend the reader has \cite{FM23} at hand while reading it.

\subsection{Simplicial Approximation for pairs}
\label{subsec: simplicial approximation for pairs}
In the context of simplicial sets, it is well known that continuous maps can be assumed to be simplicial, up to suitably subdividing the domain. 
Provided that the target is complete, there is a stronger result that provides a controlled simplicial approximation for maps between multicomplexes. The following is the natural generalization of \cite[Proposition 3.11]{FM23} to pairs of multicomplexes.
\begin{prop}[Simplicial Approximation for pairs]
	\label{proposition: simplicial approximation for pairs}
	Let $(K,K_0)$ and $(L,L_0)$ be pairs of multicomplexes and assume that both $K$ and $K_0$ are complete.
	Let $f: (|L|,|L_0|)\rightarrow (|K|,|K_0|)$ be a continuous map that is simplicial on the 0-skeleton of $L$. Assume that, for every simplex $\sigma$ of $L$, $f$ is injective on the vertices of $\sigma$, and that $f$ is simplicial on a submulticomplex $L_1$ of $L$. Then, there exists a non-degenerate simplicial map $f'\colon (L,L_0)\rightarrow (K,K_0)$ such that $|f'|$ is homotopic to $f$ (as maps of pairs) relative to $V(L)\cup L_1$.
\end{prop}
\begin{proof}
	By induction on the skeleta of $L$, it is sufficient to define, for every $n\in \N$, continuous maps $f'_n\colon (|L|,|L_0|)\rightarrow (|K|,|K_0|)$ such that
	\begin{itemize}
		\item[(1)] $f_n'$ is simplicial on $|L|^n$;
		\item[(2)] $f_n'$ is homotopic to $f_{n-1}'$ relative to $L^{n-1}\cup L_1$ (as maps of pairs).
	\end{itemize}
	We set $f_0'=f$ and we construct $f_{n+1}'$, assuming that $f_n'$ as above is given. 
	First of all, we set $f_{n+1}' = f_n'$ on $|L|^n$. We extend $f'_{n+1}$ to the $(n+1)$-skeleton in the following way. Let $\sigma$ be a $(n+1)$-simplex of $L$. If $\sigma \in L_1$, we define $f'_{n+1}(\sigma)= f(\sigma)$. If $\sigma \in L_0$, we denote by $\alpha \colon \Delta^{n+1}\rightarrow |L_0|$ its characteristic map. 
	Since $f_n'\circ \alpha|_{\partial \Delta^n}$ is a simplicial embedding into $|K_0|$, by the completeness of $K_0$, there exists a simplex $\sigma'$ of $K_0$ such that $f_{n+1}'\circ \alpha$ is homotopic (in $|K_0|$) to the characteristic map of $\sigma'$. 
	Therefore, we define $f_{n+1}'(\sigma)=\sigma'$. 
	In the same way, by making use of the completeness of $K$, we can extend $f_{n+1}'$ to $L^{n+1}$.
	Then, using the homotopy extension property for CW pairs, we then extend $f_{n+1}'$ to the whole $|L|$ in such a way that conditions (1) and (2) above hold.
\end{proof}
Another feature of completeness is that it allows to turn continuous homotopies into simplicial ones. 
The following is a generalization of \cite[Lemma 3.17]{FM23}. 
A multicomplex $K$ is called \emph{large} if every connected component of $K$ contains infinitely many vertices.
\begin{prop}[Homotopy Lemma for pairs]
	\label{proposition: homotopy lemma for pairs}
	Let $(K,K_0)$ and $(L,L_0)$ be pairs of multicomplexes and assume that both $K$ and $K_0$ are large and complete.
	Let $f,g \colon (L,L_0)\rightarrow (K,K_0)$ be non-degenerate simplicial maps such that $|f|$ is homotopic to $|g|$ (as maps of pairs). Then $f$ is simplicially homotopic to $g$ (as maps of pairs) via a non-degenerate simplicial map.
\end{prop}
\begin{proof}
	The argument in \cite{FM23} can be easily adapted to this case.
\end{proof}
\begin{prop}
	\label{prop: homotopy equiv between minimal multicomplexes is simplicial isomorphism}
	Let $(K,K_0)$ and $(L,L_0)$ be pairs of multicomplexes. Assume that $K,K_0,L,L_0$ are minimal and complete.
	Let $g \colon (K,K_0)\rightarrow (L,L_0)$ be a simplicial map of pairs which is bijective on the 0-skeletons.
	If the geometric realization $|g| \colon (|K|,|K_0|)\rightarrow (|L|,|L_0|)$ is a homotopy equivalence (as map of pairs), then $g$ is a simplicial isomorphism of pairs.
\end{prop}
\begin{proof}
	One needs to apply \cite[Proposition 3.24]{FM23} twice. First one deduces that $g|_{K_0}\colon K_0\rightarrow L_0$ is a simplicial isomorphism. Then one deduces the same for $g\colon K\rightarrow L$. It follows that the inverse of $g$ is automatically a map of pairs.
\end{proof}

\subsection{Relative Isometry Lemma}
\label{subsec: relative isometry lemma}
Let $K$ be a multicomplex. The natural chain inclusions
\[
\phi_\bullet \colon C_\bullet(K)\rightarrow C_\bullet(|K|)
\]
induce maps $H_\bullet(\phi_\bullet)\colon H_\bullet(K)\rightarrow H_\bullet (|K|)$ and $H^\bullet(\phi_\bullet)\colon H^\bullet(|K|)\rightarrow H^\bullet(K)$ which are isomorphisms in every degree \cite[Theorem 1.11]{FM23}. This result cannot be true in general for bounded cohomology \cite[pag. 25]{FM23}. 
The situation is better for complete multicomplexes: indeed Gromov's Isometry Lemma \cite[pag. 43]{Gro82} states that the bounded cohomology of a large and complete multicomplex is isometrically isomorphic to the bounded cohomology of its geometric realization.
Given a pair of multicomplexes $(K,L)$, we can consider the following commutative diagram
\begin{equation}
	\label{eq: isometry lemma}
	% https://tikzcd.yichuanshen.de/#N4Igdg9gJgpgziAXAbVABwnAlgFyxMJZABgBoBGAXVJADcBDAGwFcYkRiQBfU9TXfIRRli1Ok1btOPPtjwEiAFgpiGLNog7deIDHMFLSommsmbpOvQIUpyKkxI0gAwgH0ARgD0AOt-fNGRhgcAAoAaVIAGQBKbVlrIWQAJntxdXY3L19-QODw2JldfnlEgGZU0ydMnz8AoNCYuKL9G2Q7YzSzFw8anPqQgB8wgdIByIGCy2KDFBSOyoye7Lq8oYmmqxKicvnHRaza3NCx9a4xGCgAc3giUAAzACcIAFskMhAcCCRFQseXpDsHy+iAA7L8nq9QTRPkgABzg-6IWHQ4EATgRkNRKKQpQx32xiAArHiiQSAGwkskElIfehYRjsSBgNgkkEE5SdJy+NAACywvRWOFcwAikS4TT+kORQKQhIc6U03L5AqOrjCEohSCxMsQVM57CV-OWqsi3EoXCAA
	\begin{tikzcd}
		0 \arrow[r] & {C_b^\bullet(|K|,|L|)} \arrow[r] \arrow[d, "{\phi^\bullet_{K,L}}"] & C_b^\bullet(|K|) \arrow[r] \arrow[d, "\phi^\bullet_K"] & C_b^\bullet(|L|) \arrow[r] \arrow[d, "\phi^\bullet_L"] & 0 \\
		0 \arrow[r] & {C_b^\bullet(K,L)} \arrow[r]                                       & C_b^\bullet(K) \arrow[r]                               & C_b^\bullet(L) \arrow[r]                               & 0,
	\end{tikzcd}
\end{equation}
where the vertical arrows are induced by the chain inclusions above (hence, they are norm non-increasing).

\begin{prop}[Relative Isometry Lemma, {\cite[Proposition 1]{Kue15}}]
	\label{proposition: relative isometry lemma}
	Let $(K,L)$ be a pair of multicomplexes and assume that both $K$ and $L$ are large and complete. Then, the map
	\[
	H^n_b(\phi_{K,L}^\bullet) \colon H^n_b(|K|,|L|) \rightarrow H^n_b(K,L)
	\]
	is an isometric isomorphism for every $n\in \N$.
\end{prop}
\begin{proof}
	In \cite[Theorem 2]{FM23} the absolute case (with $L=\emptyset$) is deduced from the construction of a norm non-increasing (partial) chain map $\psi^\bullet_K\colon C^\bullet_b(K)\rightarrow C^\bullet_b(|K|)$ that induces the inverse of $H^n_b(\phi_K^\bullet)$ in bounded cohomology.
	The long exact cohomology sequence associated to (\ref{eq: isometry lemma}) induces the following commutative diagram
	\[
	% https://tikzcd.yichuanshen.de/#N4Igdg9gJgpgziAXAbVABwnAlgFyxMJZAJgBoAGAXVJADcBDAGwFcYkQAJAPTAH0AjABQAfANLDSwgDLCAlCAC+pdJlz5CKAMwVqdJq3bc+QsXMXKQGbHgJEALDpoMWbRJx4CRM+UpXX1RACMjnouhlzAYAC0gQqe0ma+lqo2GsjkIc4GbtyRMXEm4j4WVmq2KBmBulmu7nmxnqLFfmVpwVVO+rW50Q1CUs3J-uUkpB2h2e7GgqKkA+YtqUTa4zXh000LQ632Y9Vd657zCrowUADm8ESgAGYAThAAtkgOIDgQSJpJ90+fNO9IcjfB7PRAZN4fRCxCw-UHBCFIYjA36IACs-0hADZkaDMRikAB2HGE-GIAAcxPJpIAnJTXgC0ZTtAjEHiQIx6PwYIwAAopAJuO5Yc4ACxwW1hgNJRJhIKQ8IZFNlKLILNplAUQA
	\begin{tikzcd}
		H^{n-1}_b(|K|) \arrow[r] \arrow[d] & H^{n-1}_b(|L|) \arrow[r] \arrow[d] & {H^n_b(|K|,|L|)} \arrow[r] \arrow[d] & H^n_b(|K|) \arrow[r] \arrow[d] & H^n_b(|L|) \arrow[d] \\
		H^{n-1}_b(K) \arrow[r]             & H^{n-1}_b(L) \arrow[r]             & {H^n_b(K,L)} \arrow[r]               & H^n_b(K) \arrow[r]             & H^n_b(L),            
	\end{tikzcd}
	\]
	and therefore, by the Five Lemma, one deduces that also $H^n_b(\phi_{K,L}^\bullet)$ is an isomorphisms in bounded cohomology. 
	We need to show that this isomorphism in isometric. 
	To this end, it is enough to show that the inverse of $H^n_b(\phi_{K,L}^\bullet)$ is also induced by $\psi_K^\bullet$. 
	In turn, this is a consequence of the fact that \emph{$\psi_K^\bullet$ is natural with respect to the inclusion of submulticomplexes} i.e. for every submulticomplex $L$ of $K$, we have that $\psi_K^\bullet(C^\bullet_b(L))\subseteq C^\bullet_b(|L|)$.
	Unfortunately, the map $\psi^\bullet_K$ constructed in \cite{FM23} does not satisfy this last property, and, therefore, we need to slightly modify the construction.
	In \cite{FM23} the map $\psi^\bullet_K$ is defined at the chain level as the composition of two norm non-increasing (partial) chain maps
	\[
	\theta_K^i \colon C^i(K)\rightarrow C^i(\tilde{K}_N), \qquad A_K^i \colon C^i(\tilde{K}_N)\rightarrow C^i(|K|), \qquad  i\leq N
	\]
	where $N\in \N$ is a sufficiently large integer and $\tilde{K}_N$ is a particular submulticomplex of $\K(|K|\times \Delta^N)$. 
	Using the very same definition in \cite{FM23}, one can see that $A^i_K$ is natural with respect to the inclusion $L\subseteq K$.
	On the contrary, when it comes to $\theta^i_K$, we need to slightly modify the construction by making a more careful use of the homotopy extension property for CW-pairs.
	
	Since $\tilde{K}_N$ is a submulticomplex of $\K(|K|\times \Delta^N)$, we have the following commutative diagram
	\[
	\begin{array}{ccccccc}
		|\tilde{K}_N| &\subseteq & |\K(|K|\times \Delta^N)| &\longrightarrow & |K|\times \Delta^N & \longrightarrow & |K| \\
		\vertsubseteq &			& \vertsubseteq & 				& \vertsubseteq & 			& \vertsubseteq \\
		|\tilde{L}_N| &\subseteq & |\K(|L|\times \Delta^N)| &\longrightarrow & |L|\times \Delta^N & \longrightarrow & |L|,
	\end{array}
	\]
	where the horizontal arrows are the natural projections and $\tilde{L}_N$ is the intersection of $\K(|L|\times \Delta^N)$ and $\tilde{K}_N$.	
	Let now $$f\colon (|\tilde{K}_N|,|\tilde{L}_N|) \rightarrow (|K|,|L|)$$ be the continuous map obtained by composing all the maps in the diagram. 
	Observe that $f$ is not simplicial in general. However, we want to show that $f$ is homotopic (as a map of pairs) to a non-degenerate simplicial map \[\tilde{f} \colon (\tilde{K}_N,\tilde{L}_N)\rightarrow (K,L).\]
	If we define $\theta^i_K$ to be the map induced by $\tilde{f}$ at the level of cochains, it clearly follows that also $\theta^i_K$ is natural with respect to the inclusion $L\subseteq K$.
	
	We can construct $\tilde{f}$ in the following way.
	Recall from \cite{FM23} that every vertex of $\tilde{K}_N$ can be described as a pair $(x,v)$ where $x \in |K|$ and $v$ is a vertex of $\Delta^N$. Define an equivalence relation on $V(\tilde{K}_N)$ by setting $(x,v)\sim(y,w)$ if and only if $(x,v)$ and $(y,w)$ do not span a 1-simplex in $\tilde{K}_N$.
	Since $|L|$ is path connected and has infinitely many vertices, by the homotopy extension property for the CW-pair $(|\tilde{L}_N|,V(\tilde{L}_N))$, we have that $f|_{{|\tilde{L}_N|}}$ is homotopic to a continuous map \[\bar{f}_L\colon |\tilde{L}_N|\rightarrow |L|\] sending all the vertices in the same equivalence class to the same vertex of $L$, with the requirement that two distinct equivalence classes are sent to different vertices of $L$. 
	It follows that $\bar{f}_L$ is injective on every set of vertices spanning a simplex. Therefore, by \cite[Proposition 3.11]{FM23}, $\bar{f}_L$ is homotopic relative $V(\tilde{L}_N)$ to a non-degenerate simplicial map $\tilde{f}_L\colon \tilde{L}_N\rightarrow L$.
	Now, by the homotopy extension property for the CW-pair $(|\tilde{K}_N|,|V(\tilde{K}_N)\cup \tilde{L}_N|)$, we obtain that $f$ is homotopic (as a map of pairs) to a map $$\bar{f}\colon (|\tilde{K}_N|,|\tilde{L}_N| ) \rightarrow (|K|,|L|)$$ sending all the vertices in the same equivalence class to the same vertex of $K$ (and different equivalence classes in different vertices), and such that $\bar{f}$ is a non-degenerate simplicial map on $V(\tilde{K}_N)\cup \tilde{L}_N$ . 
	Finally, by Proposition \ref{proposition: simplicial approximation for pairs}, $\bar{f}$ is homotopic (as a map of pairs) relative to $V(\tilde{K}_N)\cup \tilde{L}_N$ to a non-degenerate simplicial map $$\tilde{f}\colon (\tilde{K}_N,\tilde{L}_N)\rightarrow (K,L).$$
	
	In conclusion, we define $\theta^i_K\colon C^i(K)\rightarrow C^i(\tilde{K}_N)$ to be the map induced by $\tilde{f}$ at the level of cochains. Now, it is straightforward to adapt the argument in \cite{FM23} in order show that $\psi^\bullet_K$ still induces the inverse of $H^n_b(\phi^\bullet_K)$ in bounded cohomology, and the same holds for the corresponding restrictions to $L$. The conclusion follows.
\end{proof}

\subsection{Simplicial actions of amenable groups}
\label{subsec: simplicial actions of amenable groups}
We know that amenable groups are invisible to bounded cohomology.
To our purposes, we need the following result regarding amenable groups of simplicial automorphisms, which is a generalization of \cite[Theorem 4.21]{FM23} to pairs of multicomplexes.
If a group $G$ acts simplicially on a pair of multicomplexes $(K,L)$, then it induces a linear action $G\acts C_\bullet(K)$ on the space of algebraic simplices, which in turns induces a linear action $G\acts C^\bullet_b(K,L)$ on bounded cochains. We denote by $C^\bullet_b(K,L)^G\subseteq C^\bullet_b(K,L)$ the subcomplex of $G$-invariant bounded cochains, and we use a similar notation for the induced actions on $K$ and $L$. We consider the following commutative diagram
\[
% https://tikzcd.yichuanshen.de/#N4Igdg9gJgpgziAXAbVABwnAlgFyxMJZABgBoBGAXVJADcBDAGwFcYkRiQBfU9TXfIRRli1Ok1btOPPtjwEiAFgpiGLNog7deIDHMFLSommsmbpOvQIUpyKkxI0gAwgH0ARgD0AOt-fNGRhgcAAoAaVIAGQBKbVlrIWQAJntxdXY3L19-QODw2JldfnlEgGZU0ydMnz8AoNCYuKL9G2Q7YzSzFw8anPrwqOjPAHEmqxKiFI7KjJ7surywodHC8YMUcunHWaza3Ibl7jEYKABzeCJQADMAJwgAWyQyEBwIJEVC24ekOxe3xAA7J87o9ATRXkgABzA76ISHg-4AThhoMRCKQpRR73RiAArFi8TiAGwEok4lIvehYRjsAAWEAgAGsml9UcTwVSaZp6UyQA50ppfPgcPRegscK5IiyQVCcbiOdS6QzmfyukKICKxftXGFpbCATjlJTFdzlXzOk51Zr5trgAMYlwjlwgA
\begin{tikzcd}
	0 \arrow[r] & {C_b^\bullet(K,L)^G} \arrow[r] \arrow[d, "{\iota^\bullet_{K,L}}"] & C_b^\bullet(K)^G \arrow[r] \arrow[d, "\iota^\bullet_K"] & C_b^\bullet(L)^G \arrow[r] \arrow[d, "\iota^\bullet_L"] & 0 \\
	0 \arrow[r] & {C_b^\bullet(K,L)} \arrow[r]                                              & C_b^\bullet(K) \arrow[r]                                      & C_b^\bullet(L) \arrow[r]                                & 0,
\end{tikzcd}
\]
where the vertical arrows are inclusions of complexes.
The constant $C_i$, $i \in \N_{\geq 1}$, appearing in the following is introduced in Definition \ref{defn: constants C_n}.
\begin{prop}
	\label{proposition: amenable - invariant resolutions}
	Let $G\acts (K,L)$ be a group action on a pair of multicomplexes. For every $i \in \N$, let $G_i$ denote the subgroup of $G$ acting trivially on $K^i$.
	Assume that $G/G_i$ is amenable for every $i\in \N$ and that, for every $g\in G$, the simplicial automorphism $g: (K,L) \rightarrow (K,L)$ is simplicially homotopic to the identity as a map of pairs.
	Then, for every $k \in \N$, there exists a norm non-increasing (partial) chain map
	\[
	A^i_{K,L}\colon C^i_b(K,L)\rightarrow C^i_b(K,L)^G, \qquad i\in\{0,\dots, k\},
	\]
	such that $A^i_{K,L}\circ \iota^i_{K,L}$ is the identity of $C^i_b(K,L)^G$ and $\iota^i_{K,L} \circ A^i_{K,L}$ is chain homotopic to the identity via a (partial) algebraic homotopy \[T^i\colon C^i_b(K,L)\rightarrow C^{i-1}_b(K,L)\] such that $\|T^i\|\leq C_i$ for every $i\in\{0,\dots, k\}$.
	In particular, the inclusion of invariant cochains
	\[
	\iota^\bullet_{K,L} \colon C^\bullet_b(K,L)^G\hookrightarrow C^\bullet_b(K,L)
	\]
	induces, for every $n \in \N$, an isometric isomorphism
	\[
	H^n_b(\iota^\bullet_{K,L})\colon H^n( C^\bullet_b(K,L)^G) \rightarrow H^n_b(K,L).
	\]
\end{prop}
\begin{proof}
	Let $k \in \N$.
	By a standard averaging procedure on cochains using left-invariant means on $G/G_i$, in \cite{FM23} the authors construct a norm non-increasing (partial) chain map 
	\[
	A^i_{K}\colon C^i_b(K)\rightarrow C^i_b(K)^G, \qquad i\in\{0,\dots, k\},
	\]
	such that $A_{K}^i\circ \iota_{K}^i = \id$ for every $i\in\{0,\dots, k\}$.
	Since the action $G\acts K$ leaves $L$ invariant, $A^i_K$ is natural with respect to the inclusion $L\subseteq K$.
	Therefore we are able to define, for every $k\in \N$, a norm non-increasing (partial) chain map 
	\[
	A^i_{K,L}\colon C^i_b(K,L)\rightarrow C^i_b(K,L)^G, \qquad i\in\{0,\dots, k\},
	\]
	such that $A_{K,L}^i\circ \iota_{K,L}^i = \id$ for every $i\in\{0,\dots, k\}$.
	In order to construct an algebraic homotopy between $\iota^i_{K,L} \circ A^i_{K,L}$ and the identity we proceed as follows. For every $g\in G$, we denote by
	\[
	t^i_{K,L}(g)\colon C^i_b(K,L) \rightarrow C^i_b(K,L)
	\]
	the chain map induced by $g$. Similarly we have $t^i_K(g)$ and $t^i_L(g)$. 
	Since $g \colon (K,L)\rightarrow (K,L)$ is simplicially homotopic to the identity as a map of pairs, then also the induced maps on $K$ and $L$ are.
	Therefore there are chain homotopies $T^i_K(g)\colon C^i_b(K)\rightarrow C^{i-1}_b(K)$ and $T^i_L(g)\colon C^i_b(L)\rightarrow C^{i-1}_b(L)$ such that the right square of the following diagram is commutative
	\[
	% https://tikzcd.yichuanshen.de/#N4Igdg9gJgpgziAXAbVABwnAlgFyxMJZABgBpiBdUkANwEMAbAVxiRGJAF9T1Nd9CKMgEYqtRizYduvbHgJEALOTH1mrROy48QGOQKWlR1NZM3Sde-gpTCVJiRpABhAPoAjAHoAdb+6YMDDA4ABQA0qQAMgCU2rLWgsgATPbi6mxuXr7+gcHhsTK6fPKJAMyppk6ZPn4BQaExcUX6Nsh2xmlmLh6ewNl1wQC0wpzhUQWWxQYoKR2VGT19tbk4w5xhE-ElROVzjgteSzn1ayGNnGIwUADm8ESgAGYAThAAtkhkIDgQSIqFz28kHYvj9EAB2f4vd7g6jfJAADkhgMQ8NhoIAnEjoei0UhSljfrjEABWAkkokANjJFKJKS+dCwDDYAAsIBAANZNAHQ5QgpBghzpTQAFRqx2CrmAYxio2umxA3KQxKJqM6TlF-RWrjCITlXKhSBpfMQOLVbA1y3qrkiuoKFE4QA
	\begin{tikzcd}
		0 \arrow[r] & {C_b^i(K,L)} \arrow[r] \arrow[d, "{T^i_{K,L}(g)}"] & C_b^i(K) \arrow[r] \arrow[d, "T^i_K(g)"] & C_b^i(L) \arrow[r] \arrow[d, "T^i_L(g)"] & 0 \\
		0 \arrow[r] & {C_b^{i-1}(K,L)} \arrow[r]                                 & C_b^{i-1}(K) \arrow[r]                          & C_b^{i-1}(L) \arrow[r]                               & 0.
	\end{tikzcd}
	\]
	It follows that $T^i_K(g)$ induces an algebraic homotopy $T^i_{K,L}(g)$ between $t^i_{K,L}(g)$ and the identity. 
	In conclusion, using the same averaging operations as in \cite{FM23}, one can show that $\iota^i_{K,L}\circ A^i_{K,L}$ is homotopic to the identity via a (partial) algebraic homotopy \[T^i\colon C^i_b(K,L)\rightarrow C^{i-1}_b(K,L).\] 
	By Remark \ref{rem: algebraic homotopy is bounded}, we know that $\|t^i_K(g)\|\leq C_i$.
	Thanks to the properties of the means on $G/G_i$, it is easy to deduce that $\|T^i\|\leq C_i$ (see \cite{FM23} for the details).
\end{proof}

The fact that the action $G\acts K$ preserves the submulticomplex $L$ is crucial in this section. 
However, in our applications, we also deal with actions which do \emph{not} preserve a particular submulticomplex (see Remark \ref{rem: relative construction and regularity of action}).
In order to consider invariant cochains also in this case, some regularity of the action $G\acts K$ on $L$ is required (see Definition \ref{defn: having orbits in L induced by H}). We refer the reader to Section \ref{sec: regularity of simplicial actions} for a detailed discussion on this topic. 
\section{Pairs of multicomplexes from pairs of spaces}
\label{sec: pair of multicomplexes for pairs of spaces}
Let $X$ be a CW-complex.
A fundamental result of Gromov states that the bounded cohomology of $X$ is isometrically isomorphic to the simplicial bounded cohomology of a complete minimal and aspherical multicomplex $\A(X)$ associated to $X$ \cite[Section 3.3]{Gro82} \cite[Corollary 4.24]{FM23}. 
In this section we investigate the functorial properties of Gromov's construction and we establish similar results in the relative setting.

Let $(X,A)$ be a CW-pair.
In general, we do not assume $X$ or $A$ to be path-connected.
Being injective, the inclusion map $j\colon A \hookrightarrow X$ naturally induces a simplicial embedding $j_\K\colon \K(A)\hookrightarrow \K(X)$ of singular multicomplexes.
We know that the natural projection \[S_X\colon (|\K(X)|,|\K(A)|)\rightarrow (X,A)\] is a homotopy equivalence of pairs (Proposition \ref{prop: S_X:K(X) to X induces homotopy equivalence of pairs}). 
We consider the submulticomplexes $\Elle(A)$ and $\Elle(X)$ (see Subsection \ref{subsection: minimal and aspherical multicomplexes of spaces}), and we denote by
\[
i_A\colon \Elle(A)\hookrightarrow \K(A), \qquad i_X\colon \Elle(X)\hookrightarrow \K(X),
\]
the simplicial inclusions.
One can construct an obvious simplicial map $j_\Elle$ from $\Elle(A)$ to $\Elle(X)$ which sends each simplex $\sigma$ of $\Elle(A)$ to the unique simplex of $\Elle(X)$ which is homotopic to $\sigma$ relative to the 0-skeleton.
\begin{prop}{\cite[Section 1.3]{Kue15}}
	\label{prop: j_Elle exists}
	There is a simplicial map \[j_\Elle\colon \Elle(A)\rightarrow \Elle(X)\] such that $j_\K \circ i_A = i_X \circ j_\Elle$. 
	Moreover, for every simplex $\Delta$ of $\Elle(A)$, there is a homotopy $H_\Delta \colon |\Delta|\times [0,1]\rightarrow X$ relative to the vertices between a parametrization $\sigma \colon |\Delta|\rightarrow X$ of $\Delta$ and a parametrization of $j_\Elle(\Delta)$ such that, for every submulticomplex $\Delta'$ of $\Delta$, the restriction of $H_\Delta$ to $|\Delta'|\times [0,1]$ coincides with $H_{\Delta'}$.
\end{prop}
\begin{proof}
	We define $j_\Elle$ inductively on the skeleta. 
	On the 0-skeleton, $j_\Elle$ is simply defined as the map induced by the inclusion $j\colon A\hookrightarrow X$. In this case, for every vertex $v$ of $\Elle(A)$, we set $H_v$ to be the constant map on $v$.
	Assume we have defined $j_\Elle$ on the $n$-skeleton of $\Elle(A)$ and that, for every $k$-simplex $\Delta'$ of $\Elle(A)$, $k\in\{0,\dots, n\}$, we have an homotopy $H_{\Delta'}$ as in the statement.
	We define $j_\Elle$ on the $(n+1)$-skeleton as follows. 
	Let $\Delta$ be an $(n+1)$-simplex of $\Elle(A)$ and let $\sigma\colon |\Delta|\rightarrow X$ denote a parametrization of $\Delta$.
	We want to extend $\sigma$ to a continuous map $\bar{\sigma}\colon |\Delta|\rightarrow X$ whose facets are all contained in $\Elle(X)^n$. To this end, we exploit the homotopies $H_{\Delta'}$, where $\Delta'\subset \Delta$ is a proper face of $\Delta$, as follows. We consider the diagram
	\[
	% https://tikzcd.yichuanshen.de/#N4Igdg9gJgpgziAXAbVABwnAlgFyxMJZABgBpiBdUkANwEMAbAVxiRAB12ARGBnOgHrAwAagCMAX05wAjgGMmaTgCMsAc1kK0AfWCd+TAATSmyuDBzH22NQFs6EwwB99dJi-Z5b8QyVJiKEAlSdExcfEIUMXIqWkYWNgANIJCQDGw8AiIyAOp6ZlZEDm5efiFRSSDYmCg1eCJQADMAJwhbJDIQHAgkACZgptb2xE7upElUlra+6jHEaLiCthU6Zr1rdXsJEGoGOmVeAAUwzMiQZvUACxwqiSA
	\begin{tikzcd}
		{|\Delta|\sqcup\left(\displaystyle\bigsqcup_{\Delta' \subset \Delta} |\Delta'|\times [0,1]\right)} \arrow[d] \arrow[r, "f"] & X \arrow[d, equal] \\
		{|\Delta|} \arrow[r, "\bar{\sigma}"'] &  X,
	\end{tikzcd}
	\]
	where $f$ is induced by $\sigma$ and the homotopies $H_{\Delta'}$, respectively, while the vertical map is the quotient map induced by the obvious identification of faces. By the compatibility condition of the homotopies with respect to the inclusion of faces, it follows that $f$ factors through the quotient map.
	Therefore, it defines a singular simplex $\bar{\sigma}$ of $X$, whose facets, by construction, represent simplices in $\Elle(X)$. 
	By the definition of $\Elle(X)$, there exists a unique $(n+1)$-simplex $j_\Elle(\Delta)$ of $\Elle(X)$ which is homotopic to $\bar{\sigma}$ in $X$ relative to the boundary. 
	Finally, by construction, we have a homotopy \[H_\Delta\colon |\Delta|\times [0,1]\rightarrow X\] relative to the vertices between $\Delta$ and $j_\Elle(\Delta)$ which satisfies the compatibility conditions in the statement.
\end{proof}
The map $j_\Elle$ is not injective in general: to this end one needs to have more control on the homotopy of the pair $(X,A)$.
\begin{prop}
	\label{prop: when j_Elle is injective}
	Let $n \in \N_{\geq 1}$. The simplicial map $j_\Elle$ is injective on the $n$-skeleton if and only if the map $\pi_k(A\hookrightarrow X,x)$ is injective for every $x \in A$ and every $k \in \{0,\dots, n\}$.
\end{prop}
\begin{proof}
	For every $x \in A$ and every $k$-simplex $\Delta$ of $\K(A)$ having $x$ as a vertex, we can consider the following commutative diagram
	\begin{equation}
		\label{eq: j_Elle injective}
		% https://tikzcd.yichuanshen.de/#N4Igdg9gJgpgziAXAbVABwnAlgFyxMJZABgBpiBdUkANwEMAbAVxiRAB120sB9YTgNIAKAIIBKAL5DOAERgMcdMSAml0mXPkIoyARiq1GLNp259BQgBqTp7OQqUq1IDNjwEiu8gfrNWiDi5eAGshYAAfTgBRBgYYUTFw1QAPHmJlVXU3LU9SfWpfYwDTELDI9hi4q0SUtIznV00PFAAmbwKjf0CzUJFSVPSnLKbtZDb8wz8TIJ5Qy366lQMYKABzeCJQADMAJwgAWyQyEBwIJC9JopAAKx5ZeUUhkF2DpDaTs8QAZg6p4pnQuFbtFYjBwv1Ek8XodEAAWainJAAVl+VxKsyE1whUL2MOOiMQ70KXU4ABUABYwRQ8EQ416IC4En6XEnsClUug8Sx0mHvAnwlnTADGBFWPKQzIJKMF-xFYDFEgoEiAA
		\begin{tikzcd}
			\pi_{\Elle(A)}(\Delta) \arrow[d, "j_\Delta"] \arrow[r, "\Theta_A"] & {\pi_k({|\Elle(A)|},x)} \arrow[d, "{\pi_k(|j_\Elle|,x)|}"] \arrow[r, "\cong"] & {\pi_k(A,x)} \arrow[d, "{\pi_k(j,x)}"] \\
			\pi_{\Elle(X)}(\Delta) \arrow[r, "\Theta_X"]                       & {\pi_k({|\Elle(X)|},x)} \arrow[r, "\cong"]                                    & {\pi_k(X,x),}                          
		\end{tikzcd}
	\end{equation}
	where vertical arrows are induced either by $j_\Elle$ or by the inclusion $j$, horizontal arrows of the left squares are bijective (Proposition \ref{prop: homotopy groups of complete multicomplexes}) and horizontal arrows of the right square, induced by $S_X$ and $S_A$, are isomorphisms \cite[Corollary 3.27]{FM23}.
	By using the homotopies $H_{\Delta'}$ in Proposition \ref{prop: j_Elle exists}, one can show that the diagram commutes.
	Therefore, we have that $\pi_k(j,x)$ is injective if and only if $j_\Delta$ is injective. 
	
	Assume now that $j_\Elle$ is injective on the $n$-skeleton. Let $k\leq n$, $x \in A$ and let $\Delta_0$ be a $k$-simplex of $\K(A)$ having $x$ as a vertex. It follows from our assumptions that $j_{\Delta}$ is injective, hence also $\pi_k(j,x)$ is so.
	Assume on the contrary that $\pi_k(A\hookrightarrow X,x)$ is injective for every $x \in A$ and every $k \in \{0,\dots, n\}$. We show that $j_\Elle$ is injective by induction of the skeleta. This is clearly true on the 0-skeleton. 
	In general, let $\Delta$ and $\Delta'$ be $(k+1)$-simplices of $\Elle(A)$ such that $j_\Elle(\Delta)=j_\Elle(\Delta')$.
	By the inductive hypothesis, we deduce that $\Delta$ and $\Delta'$ are compatible. By considering diagram (\ref{eq: j_Elle injective}), we have that $j_\Delta$ is injective. Since $j_\Delta(\Delta')=j_\Delta(\Delta)$, it follows that $\Delta=\Delta'$.
\end{proof}
Let $\pi \colon \Elle(X)\rightarrow \A(X)$ denote the simplicial projection which identifies simplices sharing the same 1-skeleton. 
Of course, $j_\Elle$ factors to a well-defined simplicial map $j_\A\colon \A(A)\rightarrow \A(X)$.
We denote by $\A_X(A)$ the image of $\A(A)$ in $\A(X)$ via $j_\A$, so that the pair of multicomplexes $(\A(X),\A_X(A))$ is well-defined, and we still denote by $j_\A\colon \A(A)\rightarrow \A_X(A)$ the surjective map induced by $j_\A$.
In short, we have the following commutative diagram of simplicial maps
\[
% https://tikzcd.yichuanshen.de/#N4Igdg9gJgpgziAXAbVABwnAlgFyxMJZABgBpiBdUkANwEMAbAVxiRAB12BpACgEEAlCAC+pdJlz5CKAIzkqtRizacAogwYx+Q0eOx4CRMjIX1mrRB248AGjrEgM+qUTknqZ5ZbUatdkQ5OkoYoAEzyHkoWVnzaAXrB0sjh7ormKuyx-rqOEgZJAMwRaV4xAPo2cTlB+URFqZ7RnFk6CjBQAObwRKAAZgBOEAC2SHIgOBBIZON0WAxsABYQEADWjDggkemWWGV88SADw0hF45OI4TNzi8trDBtbpbs2B0cjiNMTSJeNbABWZU4XE2V3mliWq1eg3eYy+iFOv0sAJ8mihx0QsPOABZHk12GgsGj3qc4QBWXEZAlEpA4s5IcklaLIzIgnCzMHgAisHJvGnUOEANgpljQ1MQDLhAHZhVYsFAxUK6YhpaCbpDhBRhEA
\begin{tikzcd}
	\K(A) \arrow[d, "j_\K", hook] & \Elle(A) \arrow[l, "i_A"', hook'] \arrow[d, "j_\Elle"] \arrow[r, "\pi"] & \A(A) \arrow[d, "j_\A"] \arrow[r, "j_\A"] & \A_X(A) \arrow[d, hook] \\
	\K(X)                         & \Elle(X) \arrow[l, "i_X"', hook'] \arrow[r, "\pi"]                      & \A(X) \arrow[r, "\id"]                 & \A(X).                  
\end{tikzcd}
\]

Since simplices of $\A(X)$ are uniquely determined by their 1-skeleton, the following observation can be easily deduced from Proposition \ref{prop: when j_Elle is injective}.
\begin{prop}{\cite[Section 1.3]{Kue15}}
	\label{prop: j_A is a simplicial embedding}
	Let $(X,A)$ be a CW-pair such that $A$ is $\pi_1$-injective in $X$. Then $j_\A\colon \A(A)\rightarrow \A(X)$ is a simplicial embedding and the submulticomplex $\A_X(A)$ of $\A(X)$ is canonically isomorphic to $\A(A)$.
\end{prop}

\subsection{The homotopy type of $\A_X(A)$}
\label{subsec: the homotopy type of A_X(A)}
Let $(X,A)$ be a CW-pair. 
The multicomplex $\A_X(A)$ introduced above admits the following equivalent characterization.
We consider the action on $\A(A)$ of $\bigoplus_{x \in A}\pi_1(A,x)$, understood as a subgroup of $\Pi(A,A)$ (see Subsection \ref{subsec: the action of Pi(X,X) on A(X)}).
For every $x \in A$, we set $\Gamma_x = \pi_1(A,x)$.
Let $\{A_i\,\colon \, i \in I\}$ denote the set of path-connected components of $A$.
For every $i \in I$, we fix a basepoint $\bar{x}_i\in A_i$ and we set $\Gamma_i = \Gamma_{\bar{x}_i}$. Let $N_i$ be the kernel of the morphism $\pi_1(A\hookrightarrow X, \bar{x}_i)$.
For every $x \in A_i$, there is an isomorphism $\Gamma_i \cong \Gamma_x$, which is canonical up to conjugacy. 
Therefore, being $N_i$ normal, we can find a well-defined isomorphic image $N_x$ of $N_i$ inside $\Gamma_x$, so that $N_{\bar{x}_i}=N_i$. We set 
\[
\hat{N}=\bigoplus_{i \in I}\bigoplus_{x \in A_i} N_x.
\]
Since the action of $\hat{N}$ restrict to the identity on the 0-skeleton of $\A(A)$, the quotient $\A(A)/\hat{N}$ is a well-defined multicomplex \cite[Proposition 1.14]{FM23}. We denote by $q \colon \A(A)\rightarrow \A(A)/\hat{N}$ the (simplicial) quotient map.

\begin{lemma}
	\label{lemma: homotopy tyoe  pe of A_X(A)}
	The multicomplex $\A_X(A)$ is simplicially isomorphic to the multicomplex $\A(A)/\hat{N}$.
	In particular, it is complete, minimal and aspherical and its topological realization $|\A_X(A)|$ is $\pi_1$-injective in $|\A(X)|$.
\end{lemma}
\begin{proof}	
	We first show that $\A_X(A)$ is simplicially isomorphic to $\A(A)/\hat{N}$.
	Let $e,e'$ be edges of $\A(A)$ sharing the same endpoints. We denote by $e * e'$ the loop obtained as one of the possible concatenations of $e$ and $e'$.
	By construction, $j_\A(e)=j_\A(e')$ if and only if $e * e'$ is in $N$.
	Moreover, by \cite[Lemma 5.5]{FM23}, we have that $q(e)=q(e')$ if and only if $e * e'$ is in $N$. 
	It follows that $j_\A$ factorizes through $q$ on the 1-skeleton.
	Since $\A(X)$ and $\A(A)/\hat{N}$ are complete, minimal and aspherical \cite[Theorem 5.9]{FM23}, then the factorization can be uniquely extended to a simplicial isomorphism 
	\[
	\hat{j}_\A\colon \A(A)/\hat{N}\rightarrow \A_X(A)
	\]
	such that $\hat{j}_\A\circ q = j_\A$.
	The second statement is a direct consequence of \cite[Theorem 5.9]{FM23}.
\end{proof}

\section{Proof of Theorem \ref{thm: relative mapping theorem, isometric version}}
\label{sec: proof of mapping theorem, isometric}

We have seen in the previous section how to associate to a CW-pair $(X,A)$ a pair of complete, minimal and aspherical multicomplexes $(\A(X),\A_X(A))$.
However, under additional assumptions on the homotopy of the pair, we can have more control on the objects involved.
\begin{defn}
	\label{defn: good pair}
	A CW-pair $(X,A)$ is called \emph{good} if the following conditions hold:
	\begin{itemize}
		\item[(1)] $\pi_1(A\hookrightarrow X,x)$ is injective, for every $x\in A$;
		\item[(2)] $\pi_n(A\hookrightarrow X,x)$ is an isomorphism, for every $x\in A$ and $n\in \N_{\geq 2}$.
	\end{itemize} 
\end{defn}
If $(X,A)$ is a good pair, we know that the multicomplex $\A_X(A)$ is simplicially isomorphic to the aspherical multicomplex $\A(A)$ (Proposition \ref{prop: j_A is a simplicial embedding}). 
Therefore the pair of multicomplexes $(\A(X),\A(A))$ is well-defined.
In this section, we show the bounded cohomology of a good pair $(X,A)$ is canonically \emph{isometrically} isomorphic to the \emph{simplicial} bounded cohomology of the pair $(\A(X),\A(A))$. Then, we discuss the role of higher homotopy in our framework and in Kuessner's work on relative bounded cohomology and multicomplexes \cite{Kue15}.

\subsection{Retractions on minimal multicomplexes}
Let $(X,A)$ be a CW-pair.
In Section \ref{sec: pair of multicomplexes for pairs of spaces} we have seen how to associate to $(X,A)$ a pair of complete multicomplexes $(\K(X),\K(A))$ and a simplicial map $j_\Elle\colon \Elle(A)\rightarrow \Elle(X)$ between the corresponding minimal multicomplexes.
If the map $\pi_n(A\hookrightarrow X,x)$ is injective for every $x \in A$ and every $n \in \N_{\geq 1}$, we know by Proposition \ref{prop: when j_Elle is injective} that $j_\Elle$ is a simplicial embedding. In this case, we still denote by $\Elle(A)$ the isomorphic image of $\Elle(A)$ inside $\Elle(X)$.
\begin{setup}
	\label{setup: (X,A) pi_n injective}
	Let $(X,A)$ be a CW-pair. Assume that, for every $x\in A$ and for every $n\in \N_{\geq 1}$, the map $\pi_n(A\hookrightarrow X,x)$ is injective, so that $\Elle(A)\subseteq \Elle(X)$.
\end{setup}
We know that the pair $(\K(X),\Elle(X))$ is equipped with a simplicial retraction $r_X\colon \K(X)\rightarrow \Elle(X)$ which makes the (geometric realization of the) inclusion map $i_X\colon \Elle(X)\hookrightarrow \K(X)$ a homotopy equivalence \cite[Theorem 3.23]{FM23}.
The following result shows that, in the situation of Setup \ref{setup: (X,A) pi_n injective}, one can construct a retraction $r_X$ which is functorial with respect to the inclusion $\K(A)\subseteq \K(X)$.
\begin{prop}
	\label{prop: retraction on minimal multicomplexes, for pairs}
	Let $(X,A)$ be as in Setup \ref{setup: (X,A) pi_n injective}.
	There exists a simplicial retraction $r\colon (\K(X),\K(A))\rightarrow (\Elle(X),\Elle(A))$ whose geometric realization realizes the pair $(|\Elle(X)|,|\Elle(A)|)$ as a strong deformation retract of the pair $(|\K(X)|,|\K(A)|)$. In particular, the inclusion \[(|\Elle(X)|,|\Elle(A)|)\hookrightarrow (|\K(X)|,|\K(A)|)\] is a homotopy equivalence of pairs.
\end{prop}
\begin{proof}
	First of all, we observe that, under our assumptions, for every $n\in \N_{\geq 1}$ and for every pair of $n$-simplices $\Delta$, $\Delta'$ of $\K(A)$, we have that $\Delta$ and $\Delta'$ are homotopic in $\K(X)$ if and only if they are homotopic in $\K(A)$. In fact, if $\Delta$ and $\Delta'$ are homotopic in $\K(A)$, they are clearly homotopic in $\K(X)$. On the other hand, if $x$ denotes a vertex of $\Delta$, we can consider the following commutative diagram
	\begin{equation}
		\label{eq: (K(X),K(Y)) is full}
		% https://tikzcd.yichuanshen.de/#N4Igdg9gJgpgziAXAbVABwnAlgFyxMJZABgBpiBdUkANwEMAbAVxiRAB120sB9YTgNIAKAIIBKAL5DOAERgMcdMSAml0mXPkIoyARiq1GLNp259BQgBqTp7OQqUq1IDNjwEiu8gfrNWiDi5eMCFgAB8LcTDVAA8eYmVVdTctT1J9al9jANNg0Ij2YWto0jiEp2TNDxQAJm9Mo39AsxCRUvjE51cq7WQ6jMM-EyCeEMt28okDGCgAc3giUAAzACcIAFskMhAcCCQvQeyQACseWXlFCpBVjaQ6nb3EAGYGoYDTwSubzcQAFmpdkgAKyvI7HL5rH7bQGIe5ZJqcAAqAAsYIoeCIIbdEAcYS9Dgj2Ci0XQeJYsT97jD-gThgBjAizClIfEwkG0nLsBlgJlTCRAA
		\begin{tikzcd}
			\pi_{\K(A)}(\Delta) \arrow[d, "j_\Delta"] \arrow[r, "\Theta_A"] & {\pi_n({|\K(A)|},x)} \arrow[d, "{\pi_n(|j_\K|,x)}"] \arrow[r, "\cong"] & {\pi_n(A,x)} \arrow[d, "{\pi_n(j,x)}"] \\
			\pi_{\K(X)}(\Delta) \arrow[r, "\Theta_X"]                       & {\pi_n({|\K(X)|},x)} \arrow[r, "\cong"]                   & {\pi_n(X,x)},               
		\end{tikzcd}
	\end{equation}
	where the vertical arrows are induced by inclusions, the horizontal arrows of the left squares are surjective (Proposition \ref{prop: homotopy groups of complete multicomplexes}) and the horizontal arrows of the right square, induced by $S_X$ and $S_A$, are isomorphisms (Proposition \ref{prop: S_X:K(X) to X induces homotopy equivalence of pairs}).
	Since $\Delta$ and $\Delta'$ are compatible, we have that $\Delta'\in \pi_{\K(A)}(\Delta)$.
	By Proposition \ref{prop: homotopy groups of complete multicomplexes}, since $\Delta$ and $\Delta'$ are homotopic in $\K(X)$, we have that $\Theta_X(j_\Delta(\Delta))=\Theta_X(j_\Delta(\Delta')),$ and hence, by the commutativity of (\ref{eq: (K(X),K(Y)) is full}), we get $\pi_n(|j_\K|,x)\circ \Theta_A(\Delta) = \pi_n(|j_\K|,x)\circ\Theta_A(\Delta')$.
	Since $\pi_n(j,x)$ is injective (and thus $\pi_n(|j_\K|,x)$ is also injective), we have that $\Theta_A(\Delta)= \Theta_A(\Delta')$. It follows that $\Delta$ and $\Delta'$ are homotopic in $\K(A)$.
	
	Since the map $j_\Elle\colon \Elle(A)\rightarrow \Elle(X)$ is a simplicial embedding, the pair $(\Elle(X), \Elle(A))$ can be described inductively in the following way (see Subsection \ref{subsection: minimal and aspherical multicomplexes of spaces}).
	First, we have that $\Elle(X)^0=\K(X)^0$.
	Once $\Elle(X)^n$ has been constructed, we define $\Elle(X)^{n+1}$ by adding to $\Elle(X)^n$ one $(n+1)$-simplex for every homotopy class of $(n+1)$-simplices of $\K(X)$ whose facets are all contained in $\Elle(X)^n$. 
	Moreover, if the facets are all contained in $\Elle(X)^n\cap \K(A)$ and if there is a simplex of $\K(A)$ in the corresponding homotopy class, then we chose this as a representative. In the previous paragraph, when we refer to homotopic simplices, we mean \emph{homotopic in $\K(X)$}, but we have seen that this is equivalent to require them to be \emph{homotopic in $\K(A)$}.
	In this way we obtain that $\Elle(A)=\Elle(X)\cap \K(A)$.
	
	We prove now that the pair $(|\Elle(X)|,|\Elle(A)|)$ is a strong deformation retract of $(|\K(X)|,|\K(A)|)$. As in \cite{FM23}, it is sufficient to construct, for every $n\in \N$, a map $r_n: (|\K(X)|,|\K(A)|)\rightarrow (|\K(X)|,|\K(A)|)$ and a homotopy (of pairs) $h_n \colon |\K(X)|\times [0,1]\rightarrow |\K(X)|$ between $r_n$ and $r_{n+1}$ such that the following properties hold: $r_0$ is the identity of $|\K(X)|$; $r_n|_{|\K(X)^n|}$ is a simplicial retraction of $(|\K(X)^n|,|\K(A)^n|)$ onto $(|\Elle(X)^n|,|\Elle(A)^n|)$; $r_{n+1}$ and $r_n$ coincide when restricted to $|\K(X)^n|$; $h_{n+1}(x,t)=r_n(x)$ for every $x \in |\K(X)^n|$ and $t \in [0,1]$.
	
	For $n=0$, we set $r_0=\text{id}_{|\K(X)|}$. Assume we have defined $r_n$ and $h_{n-1}$. 
	We set then $r_{n+1}|_{|\K(X)^n|}=r_n|_{|\K(X)^n|}$ and we define $r_{n+1}$ on $(n+1)$-simplices as follows. Let $\Delta$ be a $(n+1)$-simplex of $\K(X)$. 
	If $\Delta \in \K(A)$, since $r_n|_{\partial |\Delta|}$ is a simplicial embedding, by making use of the completeness of $\K(A)$ and the definition of $\Elle(X)$, we deduce that there exists a homotopy $h_\sigma \colon |\Delta|\times [0,1]\rightarrow |\K(A)|$ between the characteristic map of $\Delta$ and the characteristic map of a simplex $\Delta'\in \Elle(A)$.
	If $\Delta$ is a simplex of $\K(X)$, we invoke the completeness of $\K(X)$ instead to obtain a simplex $\Delta' \in \Elle(X)$. 
	In both cases, we set $r_{n+1}(\Delta)=\Delta '$.
	We obtain in this way the desired maps (see \cite{FM23} for the details), and this concludes the construction of the deformation retraction $r$. Notice that the map $r$ is simplicial.
\end{proof}
\begin{cor}
	\label{cor: projection L(X) to X is a homotopy equivalence of pairs}
	Let $(X,A)$ be as in Setup \ref{setup: (X,A) pi_n injective} and let 
	\[
	i \colon (|\Elle(X)|,|\Elle(A)|)\rightarrow (|\K(X)|,|\K(A)|)
	\]
	denote the inclusion. Then the composition
	\[
	S \circ i \colon (|\Elle(X)|,|\Elle(A)|) \rightarrow (X,A)
	\]
	is a homotopy equivalence of pairs.
\end{cor}
\subsection{Aspherical multicomplexes for good pairs}
Let $(X,A)$ be as in Setup \ref{setup: (X,A) pi_n injective}.
We denote by $\A(X)$ the aspherical quotient of $\Elle(X)$ and by $\pi\colon \Elle(X) \rightarrow \A(X)$ the corresponding simplicial projection (which identifies all the simplices sharing the same 1-skeleton).
Recall that, since we are assuming $A$ to be $\pi_1$-injective in $X$, we have that $\A_X(A)$ is canonically isomorphic to $\A(A)$ (Proposition \ref{prop: j_A is a simplicial embedding}). 
Therefore, we write $\A(A)\subseteq \A(X)$, so that $\pi$ induces a well-defined map of pairs
\[
\pi \colon (\Elle(X), \Elle(A)) \rightarrow (\A(X), \A(A)).
\]

The following property is crucial in the sequel, when we show that the pair $(\A(X),\A(A))$ can be obtained as a quotient of $(\Elle(X),\Elle(A))$ by the action of a group of simplicial automorphisms (Corollary \ref{cor: aspherical pair is the quotient of the minimal one}).
\begin{lemma}
	\label{lem: pi^-1(A(A))=L(A) for admissible pairs}
	Let $(X,A)$ be as in Setup \ref{setup: (X,A) pi_n injective}. We have that $\pi^{-1}(\A(A))=\Elle(A)$ if and only if the map $\pi_n(A\hookrightarrow X,x)$ is surjective for every $x\in A$ and for every $n\in \N_{\geq 2}$.
\end{lemma}
\begin{proof}
	For every $n$-simplex $\Delta$ of $\Elle(A)$ and every vertex $x$ of $\Delta$, we have the following commutative diagram
	\begin{equation}
		\label{eq: pi^-1(A(A))=L(A) for admissible pairs}
		% https://tikzcd.yichuanshen.de/#N4Igdg9gJgpgziAXAbVABwnAlgFyxMJZABgBpiBdUkANwEMAbAVxiRAB120sB9YTgKIMGMABQBBAJQBfUZwAiMBjjoBySSGml0mXPkIoyARiq1GLNp259BwsQA0Zc9ouVqNWndjwEiR8qb0zKyIHFy8YKLAAD62IhKS0VoAHjzEHtogGN76fqQm1EEWoVYRUbHsQvGOSaSp6ZqZ2Xq+KABMAYXmIWHWkeJ1aRleLQbIHQVmwZbhPJH2gw3SpjBQAObwRKAAZgBOEAC2SGQgOBBI-lPFIABWPApKKo07+0eIHafniADMniB7hyQABZqGckABWLrTUI3Z7-V7HUFfD5FHqcADGBDWcIBb0uYJ+UOuGKxOIR7yRwKJaPYmLA2L+uKQ30piEhVxpdIZFGkQA
		\begin{tikzcd}
			\pi_{\Elle(A)}(\Delta) \arrow[d, "j_{\Delta}"] \arrow[r] & {\pi_n({|\Elle(A)|},x)} \arrow[d] \arrow[r, "\cong"] & {\pi_n(A,x)} \arrow[d, "{\pi_n(j,x)}"] \\
			\pi_{\Elle(X)}(\Delta) \arrow[r]                       & {\pi_n({|\Elle(X)|},x)} \arrow[r, "\cong"]           & {\pi_n(X,x),}               
		\end{tikzcd}
	\end{equation}
	where vertical arrows are induced by the inclusion and horizontal arrows are bijective (Proposition \ref{prop: homotopy groups of complete multicomplexes} and Proposition \ref{prop: S_X:K(X) to X induces homotopy equivalence of pairs}). 
	Hence, we have that $\pi_n(j,x)$ is surjective if and only if $j_\Delta$ is surjective. 
	Moreover, since $j_\Delta$ denotes just the inclusion of simplices, we have that $j_\Delta$ is surjective if and only if every simplex of $\Elle(X)$ compatible with $\Delta$ is contained in $\Elle(A)$.
	
	Assume that we have $\pi^{-1}(\A(A))=\Elle(A)$. Let $n\in \N_{\geq 2}$ and $x \in A$ and let $\Delta$ be a $n$-simplex of $\Elle(A)$ having $x$ as a vertex (if $\Delta$ does not exists, the result is trivially true). 
	We want to show that $j_\Delta$ in (\ref{eq: pi^-1(A(A))=L(A) for admissible pairs}) is surjective.
	Let $\Delta'$ be a $n$-simplex of $\Elle(X)$ that is compatible with $\Delta$. Since $\Delta$ and $\Delta'$ share the same 1-skeleton, we have that $\pi(\Delta')=\pi(\Delta)\in \A(A)$. By assumption, this implies that $\Delta' \in \pi^{-1}(\A(A))=\Elle(A)$.
	
	Vice versa, assume that the map $\pi_n(j,x)$ is surjective, for every $x\in A$ and for every $n\in \N_{\geq 2}$.
	First of all, since $\A(A)=\pi(\Elle(A))$, the inclusion $\Elle(A) \subseteq \pi^{-1}(\A(A))$ trivially holds.
	We prove the opposite inclusion by induction on the $n$-skeleton of $\Elle(X)$.
	For $n\in \{0,1\}$, the inclusion is clear, since $\pi$ induces the identity on the $1$-skeleton of $\Elle(X)$.
	When it comes to the 2-skeleton, let $\Delta'$ be a 2-simplex $\Elle(X)$ such that $\pi(\Delta') \in \A(A)$. 
	We want to show that $\Delta' \in \Elle(A)$. 
	Let $\Delta \in \Elle(A)$ such that $\pi(\Delta)=\pi(\Delta') \in \A(A)$. 
	By definition of the aspherical quotient, we have that $\Delta$ and $\Delta'$ share the same 1-skeleton, hence they are compatible.
	Consider now diagram (\ref{eq: pi^-1(A(A))=L(A) for admissible pairs}).
	By assumption $j$ (hence $j_{\Delta}$) is surjective, and therefore, since $\Delta'$ is compatible with $\Delta$ in $\Elle(X)$, it follows that $\Delta' \in \Elle(A)$.
	In general, let $\Delta'$ be a $n$-simplex of $\Elle(X)$ such that $\pi(\Delta') \in \A(A)$, $n\in \N_{\geq 3}$. Let $F_0,\dots, F_n$ be the facets of $\Delta'$. By induction we have that $F_i\in \Elle(A)$. Moreover, $\pi(F_i)$ are the facets of $\pi(\Delta') \in \A(A)$, and therefore there exists a $n$-simplex $\Delta$ of $\Elle(A)$ with facets $F_0,\dots, F_n$. Since $\Delta$ and $\Delta'$ are compatible, we can repeat the same argument above on diagram (\ref{eq: pi^-1(A(A))=L(A) for admissible pairs}) to deduce that $\Delta' \in \Elle(A)$.
\end{proof}

\subsection{Amenable groups of simplicial automorphisms}
Let $(X,A)$ be a good pair i.e. a CW-pair such that
\begin{itemize}
	\item[(1)] $\pi_1(A\hookrightarrow X,x)$ is injective, for every $x\in A$;
	\item[(2)] $\pi_n(A\hookrightarrow X,x)$ is an isomorphism, for every $x\in A$ and $n\in \N_{\geq 2}$.
\end{itemize}
In this section we prove that the pair $(\A(X),\A(A))$ can be obtained from the pair $(\Elle(X),\Elle(A))$ by taking the quotient with respect to a particular simplicial action (Corollary \ref{cor: aspherical pair is the quotient of the minimal one}). 
It turns out that this action is equivalent to the action by an amenable group. 
Since amenable groups are invisible to bounded cohomology (Proposition \ref{proposition: amenable - invariant resolutions}), this fact will be the key ingredient in our proof of Theorem \ref{thm: relative mapping theorem, isometric version}.
We begin with the following definition.

\begin{defn}
	\label{definition: group of simplicial automorphisms}
	Let $(K,L)$ be a pair of multicomplexes. We define $\Gamma(K,L)$ to be the group of all simplicial automorphisms in $\Aut(K,L)$ that are (topologically) homotopic to the identity (as maps of pairs) relative to the 0-skeleton of $K$.
	Moreover, for every $i\in \N_{\geq 1}$, we define
	\[
	\Gamma_i(K,L) = \bigl\{g\in \Gamma(K,L) \,\colon\, g|_{K^i} = \id \bigr\}.
	\]
	If $L=\emptyset$, we write $\Gamma_i(K)=\Gamma_i(K,\emptyset)$.
\end{defn}

We want to show that the groups $\Gamma_i(K,L)$ act as transitively as possible on the set of $(i+1)$-simplices of \emph{both} $K$ and $L$. To this end, we introduce the following notion.

\begin{defn}
	\label{definition: i-coherent pair}
	Let $i\in \N_{\geq 1}$. A pair of multicomplexes $(K,L)$ is \emph{$i$-coherent} if, for every pair of simplices $\Delta$, $\Delta'$ of $K$, sharing the same $i$-skeleton, one has $\Delta \in L$ if and only if $\Delta' \in L$.
\end{defn}

We recall that, for every simplex $\Delta$ of $K$, $\pi_K(\Delta)$ denotes the set of simplices of $K$ that are compatible with $\Delta$. If $(K,L)$ is a pair of multicomplexes, for every simplex $\Delta$ of $L$, we have in general just the inclusion $\pi_L(\Delta)\subseteq \pi_K(\Delta)$. However, if we assume the pair $(K,L)$ to be $i$-coherent, we have that \[\pi_K(\Delta_0)= \pi_L(\Delta_0),\] for every $(i+1)$-simplex $\Delta_0$ of $L$.
Moreover, if $(K,L)$ is 1-coherent, then it is $n$-coherent for every $n \in \N_{\geq 1}$.
\begin{lemma}
	\label{lem: (L(X),pi^{-1}A(A)) is 1-coherent}
	Let $(X,A)$ be a CW-pair. Then the pair \[(\Elle(X),\pi^{-1}(\A_X(A)))\] is 1-coherent. In particular, if $(X,A)$ is good, then the pair $(\Elle(X),\Elle(A))$ is 1-coherent.
\end{lemma}
\begin{proof}
	The first statement is obvious, since $\A(X)$ is obtained by identifying simplices of $\Elle(X)$ sharing the same 1-skeleton. The second statement follows from the fact that $\Elle(A)=\pi^{-1}(\A(A))$ for good pairs (Lemma \ref{lem: pi^-1(A(A))=L(A) for admissible pairs}).
\end{proof}
We are able to generalize \cite[Lemma 4.12]{FM23} to coherent pairs of multicomplexes.
\begin{lemma}
	\label{lemma: Gamma acts transitively on compatible simplices of pairs}
	Let $i\in \N_{\geq 1}$. Let $(K,L)$ be a $i$-coherent pair of multicomplexes. Assume that both $K$ and $L$ are complete and minimal. Let $\Delta_0$ be an $(i+1)$-simplex of $K$, and let $\Delta \in \pi_K(\Delta_0)$. Let $F$ be a facet of $\Delta_0$. Then there exists an element $g\in\Gamma_i(K,L)$ such that the following statements holds: $g(\Delta_0) = \Delta$, and $g(\Delta')=\Delta'$ for every $m$-simplex $\Delta'$, $m\geq i$, which does not contain $F$.
\end{lemma}
\begin{proof}
	The absolute case (with $L=\emptyset$) is treated in \cite{FM23}, where an element $g \in \Gamma_i(K)$ with the desired properties is constructed.
	We just need to make sure that $g$ preserves $L$. Consider the map $\tilde{f} \colon K^i \cup \Delta_0 \rightarrow K$, which sends $K^i \cup \Delta_0$ isomorphically to $K^i \cup \Delta$, extending the identity on $K^i$ and a linear isomorphism between $\Delta$ and $\Delta_0$. 
	Since $(K,L)$ is $i$-coherent, then $\tilde{f}$ induces the following map of pairs
	\[
	\tilde{f} \colon (K^i \cup \Delta_0, (K^i \cup \Delta_0)\cap L) \rightarrow (K,L).
	\]
	Let $K'\subseteq K^i$ be the submulticomplex of $K^i$ obtained by removing from $K^i$ the interior of $F$. 
	Using the very same homotopy $h\colon |K^i\cup \Delta_0|\times [0,1]\rightarrow |K|$ constructed in \cite{FM23}, it is easy to check that $\tilde{f}$ is homotopic (as a map of pairs) to the inclusion relative to $K'$. 
	Then, by the homotopy extension property of the CW-pair $(|K|, |K^i\cup \Delta_0|)$, there exists a homotopy $H\colon |K|\times [0,1]\rightarrow |K|$ extending $h$. 
	Let $\tilde{f}_1 = H_1 \colon |K| \rightarrow |K|$ and let $K''\subseteq K$ be the subcomplex of $K$ obtained by removing from $K$ the interior of $F$ and all the simplices containing $F$. 
	By looking at the explicit construction of $H$ given in \cite[Proposition 0.16]{Hat}, one realizes that $H$ may be assumed to be constant on $K''$ and such that $H(|L| \times I)\subseteq |L|$. 
	Therefore $\tilde{f}_1$ defines a map of pairs $\tilde{f}_1 \colon (|K|,|L|)\rightarrow(|K|,|L|)$ that is equal to identity on $K''$, hence it is simplicial when restricted on $K''\cup \Delta_0$. 
	By the completeness of $K$ and $L$, it follows from Proposition \ref{proposition: simplicial approximation for pairs} that there exists a non-degenerate simplicial map \[g \colon (K,L)\rightarrow (K,L)\] such that $\tilde{f}_1$ is homotopic to $|g|$ relative to $V(K)\cup K'' \cup \Delta_0$ (as maps of pairs). 
	By construction, $g$ is homotopic to the identity relative to the 0-skeleton, fixes $K^i$, sends $\Delta_0$ to $\Delta$, and restricts to the identity on $K''$.
	Moreover, $g$ is indeed an isomorphism of pairs of multicomplexes by Proposition \ref{prop: homotopy equiv between minimal multicomplexes is simplicial isomorphism}, and therefore it defines an element of $\Gamma_i(K,L)$.
\end{proof}

If $(K,L)$ is a $i$-coherent pair of multicomplexes, then it follows from Lemma \ref{lemma: Gamma acts transitively on compatible simplices of pairs} that two $(i+1)$-simplices of $K$ are compatible if and only if they are in the same $\Gamma_i(K,L)$ orbit. 
In other words, for every $(i+1)$-simplex $\Delta_0$ of $K$, the group $\Gamma_i(K,L)$ acts transitively on $\pi_K(\Delta_0) = \pi_L(\Delta_0)$.
This observation can be generalized in the following way.

\begin{prop}
	\label{prop: Gamma acts transitively on simplices}
	Let $i,n \in \N_{\geq 1}$ be such that $n\geq i+1$. Let $(K,L)$ be a $i$-coherent pair of multicomplexes and assume that both $K$ and $L$ are complete and minimal. Let $\Delta, \Delta'$ be $n$-simplices of $K$. Then $\Delta$ and $\Delta'$ are in the same $\Gamma_i(K,L)$-orbit if and only if they share the same $i$-skeleton.
\end{prop}
\begin{proof}
	The argument is the same as the one of \cite[Proposition 4.14]{FM23}, invoking Lemma \ref{lemma: Gamma acts transitively on compatible simplices of pairs} when needed.
\end{proof}

Let now $(X,A)$ be a good pair and let $\Gamma_1=\Gamma_1(\Elle(X),\Elle(A))$. Since the action $\Gamma_1 \acts (\Elle(X),\Elle(A))$ is trivial on the 0-skeleton by definition, we obtain a well defined pair of multicomplexes
\[
(\Elle(X),\Elle(A))/\Gamma_1 = (\Elle(X)/\Gamma_1,\Elle(A)/\Gamma_1).
\]
Moreover, since the pair $(\Elle(X),\Elle(A))$ is 1-coherent (Lemma \ref{lem: (L(X),pi^{-1}A(A)) is 1-coherent}), Proposition \ref{prop: Gamma acts transitively on simplices} shows that we can characterize the aspherical quotient $(\A(X),\A(A))$ of $(\Elle(X),\Elle(A))$ as the quotient of the simplicial action $\Gamma_1\acts(\Elle(X),\Elle(A))$.
\begin{cor}
	\label{cor: aspherical pair is the quotient of the minimal one}
	Let $(X,A)$ be a good pair of topological spaces and let $\Gamma_1 = \Gamma_1(\Elle(X),\Elle(A))$.
	Then $(\Elle(X),\Elle(A))/\Gamma_1$ is canonically isomorphic to $(\A(X),\A(A))$.
\end{cor}

Let $K$ be a minimal multicomplex. The group $\Gamma_1(K)$ is \emph{not} amenable in general. However, using the fact that higher homotopy groups are abelian, it turns out that the quotients $\Gamma_1(K)/\Gamma_i(K)$ are amenable for every $i\in \N_{\geq 1}$ \cite[Section 3.3]{Gro82}\cite[Corollary 4.20]{FM23}. The same result holds for pairs of multicomplexes.
\begin{lemma}
	\label{lemma: amenability of Gamma(K,L)}
	Let $(K,L)$ be a pair of multicomplexes and assume that both $K$ and $L$ are minimal.
	Then, for every $i\in \N_{\geq 2}$, $\Gamma_{i}(K,L)$ is a normal subgroup of $\Gamma_{i-1}(K,L)$ and the quotient $\Gamma_{i-1}(K,L) / \Gamma_i(K,L)$ is abelian. Therefore $\Gamma_{1}(K,L) / \Gamma_i(K,L)$ is solvable, hence amenable.
\end{lemma}
\begin{proof}
	In \cite{FM23} the authors construct a group homomorphism
	\[
	\phi^{(i)}_{K} \colon \Gamma_{i-1}(K) \rightarrow \prod_{\alpha \in A} \pi_i(|K|,p_\alpha),
	\]
	where $A$ is some index set and $p_\alpha$ is a vertex of $K$, for every $\alpha \in A$.
	Since $i\in \N_{\geq 2}$, the target group is abelian.
	Therefore, by composing with the inclusion $\Gamma_{i-1}(K,L) \leq \Gamma_{i-1}(K)$, we get a group homomorphism
	\[
	\phi^{(i)}_{K,L} \colon \Gamma_{i-1}(K,L) \rightarrow \prod_{\alpha \in A} \pi_i(|K|,p_\alpha).
	\]
	Finally, by the very same argument of \cite[Lemma 4.18]{FM23}, we can deduce that $\ker(\phi^{(i)}_{K,L})=\Gamma_i(K,L)$, and the statement easily follows.
\end{proof}

\subsection{Proof of Theorem \ref{thm: relative mapping theorem, isometric version}}
Let $(X,A)$ be a good pair and let $n\in \N$. We want to show that $H^n_b(X,A)$ is canonically isometrically isomorphic to $H^n_b(\A(X),\A(A))$.
We know from Corollary \ref{cor: projection L(X) to X is a homotopy equivalence of pairs} that the map
\[
S\circ i \colon (|\Elle(X)|,|\Elle(A)|)\rightarrow (X,A)
\]
induces an isometric isomorphism in bounded cohomology. Moreover, by the Relative Isometry Lemma (Proposition \ref{proposition: relative isometry lemma}), also the map
\[
H^n_b(\phi^\bullet) \colon H^n_b(|\Elle(X)|,|\Elle(A)|)\rightarrow H^n_b(\Elle(X),\Elle(A))
\]
is an isometric isomorphism. 

We denote by $\Gamma_i$ the group of simplicial automorphisms $\Gamma_i(\Elle(X),\Elle(A))$, $i\in \N_{\geq 1}$.
Recall from Corollary \ref{cor: aspherical pair is the quotient of the minimal one} that the pair $(\A(X),\A(A))$ is canonically isomorphic to the quotient $(\Elle(X),\Elle(A))/\Gamma_1$. Therefore there is an obvious chain isomorphism between $C^\bullet_b(\A(X),\A(A))$ and $C^\bullet_b(\Elle(X),\Elle(A))^{\Gamma_1}$. Under this identification, the projection
\[
\pi : (\Elle(X),\Elle(A)) \rightarrow (\A(X),\A(A))
\]
induces the inclusion $C^\bullet_b(\Elle(X),\Elle(A))^{\Gamma_1}\hookrightarrow C^\bullet_b(\Elle(X),\Elle(A))$.
Moreover, the elements of $\Gamma_1$ are topologically homotopic to the identity (as maps of pairs), and therefore they are simplicially homotopic to the identity (as maps of pairs) by Proposition \ref{proposition: homotopy lemma for pairs}.
Since the groups $\Gamma_1/\Gamma_i$ are amenable for every $i\in \N_{\geq 1}$ (Proposition \ref{lemma: amenability of Gamma(K,L)}), we deduce from Proposition \ref{proposition: amenable - invariant resolutions} that the projection $\pi$ induces an isometric isomorphism in bounded cohomology, and the statement follows.

\subsection{The role of higher homotopy}
\label{subsec: role of higher homotopy}
We conclude this section by discussing the role of higher homotopy in our setting.
This discussion aims to clarify the connection between our framework and Kuessner's work \cite{Kue15} on relative bounded cohomology via multicomplexes.

In fact, \cite[Proposition 3]{Kue15} states that, for every pair of topological spaces $(X,A)$ and every $n \in \N$, we have that $H^n_b(X,A)$ is isometrically isomorphic to $H^n_b(\A(X),\A(A))$, provided that the map $\pi_1(A\hookrightarrow X,x)$ is injective for every $x \in A$.
This statement is obviously stronger than Theorem \ref{thm: relative mapping theorem, isometric version}, and was used by Kuessner to deduce similar conclusions of our Theorem \ref{thm: additivity full boundary components} without assuming the manifolds to be aspherical.
However, in the proof of \cite[Proposition 3]{Kue15}, Kuessner seems to implicitly exploit some further assumptions, concerning the higher homotopy of the pair $(X,A)$.
In order to address these issues, which were initially highlighted by Moraschini in his PhD thesis \cite[Section 4.4]{Mor18}, we first recall Kuessner's setting.
Unfortunately, our notation differs from Kuessner's one: our multicomplex $\Elle(X)$ is denoted by $\hat K (X)$ in \cite{Kue15} and the aspherical multicomplex $\A(X)$ is denoted by $K(X)$.

We fix a CW-pair $(X,A)$ such that $A$ is $\pi_1$-injective in $X$.
Since $A$ is $\pi_1$-injective in $X$, we can write $\A(A)\subseteq \A(X)$ (Proposition \ref{prop: j_A is a simplicial embedding}). However, without any further assumptions on the higher homotopy of the pair $(X,A)$, we do \emph{not} have in general an inclusion of $\Elle(A)$ inside $\Elle(X)$ (Proposition \ref{prop: when j_Elle is injective}). Let $\pi\colon \Elle(X)\rightarrow \A(X)$ be the quotient map which identifies simplices haring the same 1-skeleton.
We consider the submulticomplex $\pi^{-1}(\A(A))$ of $\Elle(X)$ and we denote by \[\iota\colon \pi^{-1}(\A(A))\hookrightarrow \Elle(X)\] the inclusion map.
In the proof of \cite[Proposition 3]{Kue15} the following statements appear:
\begin{itemize}
	\item[(a)] The multicomplex $\pi^{-1}(\A(A))$ is complete. This is in fact required to invoke the Relative Isometry Lemma \cite[Proposition 1]{Kue15} (see also our Proposition \ref{proposition: relative isometry lemma}).
	\item[(b)] The continuous map $S_X\colon |\Elle(X)|\rightarrow X$ maps $|\pi^{-1}(\A(A))|$ to $A$.
\end{itemize}
We want to show that, if we assume (a) and (b), then $\pi_n(A\hookrightarrow X,x)$ is an isomorphism for every $x \in A$ and every $n\in \N_{\geq 2}$, which is indeed the assumption we are making in our paper.

Let $n \in \N_{\geq 2}$ and let $x \in A$. Let $\Delta$ be an $n$-simplex of $\Elle(A)$ having $x$ as a vertex. We consider the following diagram
\[
% https://tikzcd.yichuanshen.de/#N4Igdg9gJgpgziAXAbVABwnAlgFyxMJZABgBpiBdUkANwEMAbAVxiRAB120sB9YTgKIMGMABQBBAJQBfUZwAiMBjjqSQ00uky58hFGQCMVWoxZtO3PmgB6wALQHZncRMkzRAKx6DhYhUpU3dU0QDGw8AiIyACZjemZWRA4uXn52IRFRAA13Lx9M-2VVNQ0tcN0iaPI400TkyzAJUgAPEpCwnUiUKtjqeLMki15GrJa2ss69ZANqvtrzFJ5GgB98sSllseCJiKmZozmEhYbRZZt7Rzl2FylJTdbt0O1dohnekyPBxZW17Lut0pPcpdZBVA4fAb1YZNB7SYwwKAAc3gRFAADMAE4QAC2SDIIBwECQMwhdU4jDQAAs6CBqAw6AAjJQABWeFSSIjROEemJxxOohKQVVJCyZKlpIHpTIYrOBeklMC5PKxuMQ+MFiAArIdIZwACqUmAqHjiZV8rUComIADMOrJ31EAGUTQCQrzVdqCVaAGx247Q5Z5dK+e5qOmMllsroKpWA91IX1epAAdj9XxOyywoYlUsjcrYnO5cZVKctSAALGmoUsnTxRrC3SXEMKNamRV8DUa6HWzaqSRrE-17Z3xcXzYmNQAOMeq21JxCTqucbF0HCUjHY4BYKDSXtIRfzyvt6uNDwAijSIA
\begin{tikzcd}
	\pi_{\Elle(A)}(\Delta) \arrow[d, "\alpha"] \arrow[r, "\Theta_A"]            & {\pi_n(|\Elle(A)|,x)} \arrow[r, "{\pi_n(S_A,x)}"] \arrow[d, "{\pi_n(|j_\Elle|,x)}"] & {\pi_n(A,x)} \arrow[d, "\mathrm{id}"]  \\
	\pi_{p^{-1}(\A(A))}(j_\Elle(\Delta)) \arrow[d, "\beta"] \arrow[r, "\Theta"] & {\pi_n(|\pi^{-1}(\A(A))|,x)} \arrow[d, "{\pi_n(|\iota|,x)}"] \arrow[r]                    & {\pi_n(A,x)} \arrow[d, "{\pi_n(j,x)}"] \\
	\pi_{\Elle(X)}(j_\Elle(\Delta)) \arrow[r, "\Theta_X"]                       & {\pi_n(|\Elle(X)|,x)} \arrow[r, "{\pi_n(S_X,x)}"]                                   & {\pi_n(X,x),}                          
\end{tikzcd}
\]
where we need to make the following clarifications:
\begin{itemize}
	\item The maps $\Theta$, $\Theta_A$ and $\Theta_X$ are defined in Proposition \ref{prop: homotopy groups of complete multicomplexes}. Notice that we have used (a) to define the map $\Theta$, which is surjective. Moreover, since $\Elle(A)$ and $\Elle(X)$ are minimal, then $\Theta_A$ and $\Theta_X$ are bijective.
	\item  Since $S_A$ and $S_X$ are homotopy equivalences \cite[Corollary 3.27]{FM23}, it follows that $\pi_n(S_A,x)$ and $\pi_n(S_X,x)$ are both isomorphisms.
	\item The map $\alpha$ is induced by $j_\Elle$, whose image is clearly contained in $\pi^{-1}(\A(A))$, while $\beta$ denotes just the inclusion of simplices.
	\item By condition (b), every square of the diagram is commutative.
\end{itemize}
By Lemma \ref{lem: (L(X),pi^{-1}A(A)) is 1-coherent}, we have that the pair $(\Elle(X), \pi^{-1}(\A(A)))$ is 1-coherent, hence $n$-coherent (see Definition \ref{definition: i-coherent pair}). It follows that $\beta$ is a bijective map. Therefore, since $\Theta_X\circ \beta = \pi_n(|\iota|,x)\circ \Theta$ is bijective, we can deduce that $\Theta$ is injective (hence bijective). 
Similarly, being $\pi_n(S_A,x)$ an isomorphism, we can deduce that $\pi_n(|j_\Elle|,x)$ is injective, and hence also $\alpha$ is injective. It follows from these observations that $\pi_n(j,x)$ is an isomorphism.

\section{Proof of Theorem \ref{thm: relative bounded cohomology and amenable connected components}}
\label{sec: proof of theorem on amenable connected components}

Let $(X,A)$ be a CW-pair.
In this section, we first show that the subgroup $\Pi_X(A)$ of $\Pi(X,X)$ induces a well-defined action on the pair of multicomplexes $(\A(X),\A_X(A))$. Then, we establish a crucial observation by Gromov (Lemma \ref{lema: a crucial observation}), which is the key step in our proof of Theorem \ref{thm: relative bounded cohomology and amenable connected components}.\\

We consider the group $\Pi(X,X)$, whose elements are given by finite collections of (homotopy classes of) paths in $X$ (see Subsection \ref{subsec: the group Pi(X,X)}).
We consider the subgroup $\Pi_X(A)$ of $\Pi(X,X)$ given by of paths supported on $A$.
We have seen in Section \ref{sec: pair of multicomplexes for pairs of spaces} that there is a well-defined pair of complete, minimal and aspherical multicomplexes $(\A(X),\A_X(A))$. Moreover, we have constructed in Subsection \ref{subsec: the action of Pi(X,X) on A(X)} an action
\[
\psi\colon \Pi(X,X)\rightarrow \Aut(\A(X)),
\]
where $\Pi(X,X)$ acts on $\A(X)$ by conjugating the edges of $\A(X)$ (which correspond to paths in $X$) with the elements of $\Pi(X,X)$.
The action $\Pi(X,X)\acts \A(X)$ does \emph{not} preserve the submulticomplex $\A_X(A)$. However, if we restrict to the subgroup $\Pi_X(A)$, we have that $\psi$ induces a well defined action of pairs
\[
\Pi_X(A)\rightarrow \Aut(\A(X),\A_X(A)).
\]
To this end, we just need to notice that every element of $\Pi_X(A)$ can be represented by a collection of paths $\{\gamma_x\}_{x\in A}$ which are supported on $A$.
It clearly follows from the construction in Subsection \ref{subsec: the action of Pi(X,X) on A(X)} that, for every $g\in \Pi_X(A)$, we have $g\cdot \A_X(A)\subseteq \A_X(A)$, and the action $\Pi_X(A)\acts (\A(X),\A_X(A))$ is indeed well defined.
\begin{prop}
	\label{prop: elements of Pi(A,A) are homotopic to the indentity}
	Let $(X,A)$ be a CW-pair. Let
	\[
	\psi \colon \Pi_X(A) \rightarrow \Aut(\A(X),\A_X(A))
	\]
	be the action described above. Then $\psi(g)$ is simplicially homotopic (as a map of pairs) to the identity for every $g\in \Pi_X(A)$.
\end{prop}
\begin{proof}
	For every $x\in A=\A_X(A)^0$, we know that $\psi(g)$ acts as the identity on both $\pi_1(|\A(X)|,x)$ and $\pi_1(|\A_X(A)|,x)$ \cite[Theorem 5.3]{FM23}.
	Since $|\A(X)|$ and $|\A_X(A)|$ are aspherical (Lemma \ref{lemma: homotopy tyoe  pe of A_X(A)}), this implies that the maps
	\[
	|\A(X)|\rightarrow |\A(X)|, \qquad |\A_X(A)|\rightarrow |\A_X(A)|,
	\]
	induced by $\psi(g)$ are both homotopic to the identity \cite[Proposition 1B.9]{Hat}.
	Since the inclusion $|\A_X(A)|\hookrightarrow |\A(X)|$ is a cofibration, this is indeed enough to conclude that $|\psi(g)|$ is homotopic to the identity \emph{as a map of pairs} \cite[7.4.2]{Bro06}.
	Finally the Homotopy Lemma for pairs (Proposition \ref{proposition: homotopy lemma for pairs}) implies that $\psi(g)$ is simplicially homotopic to the identity as a map of pairs.
\end{proof}
As already suggested by Gromov \cite[pag. 57]{Gro82}, the following simple observation is crucial.
\begin{lemma}{\cite[Observation 1]{Kue15}}
	\label{lema: a crucial observation}
	Let $(X,A)$ be a CW-pair and let $G = \Pi_X(A)$. Let $n\in \N$ and let $z \in C^n_b(\A(X))^G_{\alt}$ be a $G$-invariant alternating cochain. Let $(\Delta,(x_0,\dots,x_n))$ be an algebraic simplex of $\A(X)$ and assume that $\Delta$ has at least one edge in $\A_X(A)$. Then $z(\Delta,(x_0,\dots,x_n))=0$.
\end{lemma}
\begin{proof}
	If $x_h = x_k$ for some $h\neq k$, then $z(\Delta,(x_0,\dots,x_n))=0$, since $z$ is alternating. Therefore, we can assume that $x_h\neq x_k$ for every $h\neq k$. Let $e$ be an edge of $\Delta$ contained in $\A_X(A)^1$ and let, without loss of generality, $x_0$ and $x_1$ be its endpoints. We know that $e$ represents a homotopy class (relative to endpoints) of paths in $A$ joining $x_0$ and $x_1$. Let $\gamma\colon [0,1]\rightarrow A$ be a representative of $e$ and let $g=\{\gamma, \bar{\gamma} \}\in \Pi_X(A)$. It is clear that $g\cdot \Delta = \Delta$, $g\cdot x_0 = x_1$, $g\cdot x_1 = x_0$ and $g\cdot x_l = x_l$ for every $l \in \{2,\dots, n \}$. Since $z$ is alternating and $G$-invariant, we get that
	\[
	z(\Delta, (x_0,\dots,x_n)) = -z(g\cdot (\Delta, (x_0,\dots,x_n))) = - z(\Delta, (x_0,\dots,x_n)),
	\]
	hence $z(\Delta,(x_0,\dots,x_n))=0$.
\end{proof}

\subsection{Proof of Theorem \ref{thm: relative bounded cohomology and amenable connected components}}
	Let $(X,Y)$ be a good pair of topological spaces and let $A\subseteq Y$ be the union of some connected components of $Y$. Let $B=Y\setminus A$ be the union of the remaining connected components and let $j\colon (X,B)\rightarrow (X,Y)$ denote the inclusion.
	Assume that every connected component of $A$ has an amenable fundamental group.
	We want to show that the map 
	\[
	H^n_b(j)\colon H^n_b(X,Y)\rightarrow H^n_b(X,B)
	\]
	is an isometric isomorphism for every $n \in \N$.
	We consider the following commutative diagram
	\[
	% https://tikzcd.yichuanshen.de/#N4Igdg9gJgpgziAXAbVABwnAlgFyxMJZABgBpiBdUkANwEMAbAVxiRAAkA9MAfQCMAFAA1SATQCUIAL6l0mXPkIoAjOSq1GLNl16CRAIUky52PASJll6+s1aIO3fgIA6zgILDxpVx4lHZIBimikSqVtQ2WvY6Tj6e3u4Chv4mCuYoZABM1pp2DmACAMKcrnxMDAwwOLGJQl5xfpwA4jyujDgpgfJmSsiq2RG52txFJc5lFVU1HnUJHsnNrc7tRuowUADm8ESgAGYAThAAtkhkIDgQSKoatmwAVtIBB8dImdQXp8YgzyeIAMzvS6IZRfH6vQFIP6gw6-AAsEMQmWhL0QAFYEVCnjCkPDzkDUVIKFIgA
	\begin{tikzcd}
		{H^n_b(X,Y)} \arrow[r, "H^n_b(j)"]                                & {H^n_b(X,B)}                                     \\
		{H^n_b(\A(X),\A(Y))} \arrow[u] \arrow[r]                   & {H^n_b(\A(X),\A(B))} \arrow[u]                   \\
		{H^n(C^\bullet_b(\A(X),\A(Y))^G_\alt)} \arrow[u] \arrow[r] & {H^n(C^\bullet_b(\A(X),\A(B))^G_\alt),} \arrow[u]
	\end{tikzcd}
	\]
	where the following considerations need to be made:
	\begin{itemize}
		\item The horizontal arrows are induced by inclusions.
		\item Since the pair $(X,Y)$ is good, also the pairs $(X,A)$ and $(X,B)$ are good, hence $\A_X(Y)\cong \A(Y)$ and $\A_X(B)\cong \A(B)$ by Proposition \ref{prop: j_A is a simplicial embedding}. The upper vertical arrows denote the isometric isomorphisms of Theorem \ref{thm: relative mapping theorem, isometric version}.
		\item By Lemma \ref{lem: Pi(A,A) is amenable}, $G=\Pi_X(A)$ is an amenable group.
		Since $A$ is a collection of path connected components of $Y$, one easily checks that the action $G\acts \A(X)$ induces well defined actions $G\acts (\A(X),\A(Y))$ and $G\acts (\A(X),\A(B))$. These actions consist of elements that are simplicially homotopic to the identity as maps of pairs (Proposition \ref{prop: elements of Pi(A,A) are homotopic to the indentity}). It follows from Proposition \ref{proposition: amenable - invariant resolutions} that the lower vertical arrows are isometric isomorphisms.
	\end{itemize}
	In order to conclude, it is sufficient to show that the map
	\[
	j^n \colon C^n_b(\A(X),\A(Y))^G_\alt\rightarrow C^n_b(\A(X),\A(B))^G_\alt
	\]
	induces an isometric isomorphism in cohomology. But since every bounded alternating $G$-invariant cochain vanishes on simplices supported in $\A(A)\subseteq \A(Y)$ (Lemma \ref{lema: a crucial observation}), it follows that $j^n$ is bijective. Thus $j^n$ itself is an isometric isomorphism and the statement follows.

\section{Mapping Cones}
\label{sec: mapping cones}
In this section we introduce the machinery of mapping cones. 
The technique was introduced by Park in \cite{Par03} and played a fundamental role in several developments of bounded cohomology theory \cite{Loh08, BBF+14}.\\

Let $(M^\bullet,\delta_M^\bullet)$ and $(N^\bullet,\delta_N^\bullet)$ be normed cochain complexes and let $\varepsilon^\bullet\colon M^\bullet\rightarrow N^\bullet$ be a chain map between them.
For every $n \in \N$ we set
\[
\Cone(\varepsilon)^n = M^n \oplus N^{n-1}
\]
and we define a boundary operator $\delta_C^n\colon \Cone(\varepsilon)^n \rightarrow \Cone(\varepsilon)^{n+1}$ by setting
\[
\delta_C^n(u,v)=(\delta^n_M(u), -\varepsilon^n(u)-\delta^{n-1}_N(v)).
\]
The complex $(\Cone(\varepsilon)^\bullet,\delta_C^\bullet)$ is called the \emph{mapping cone complex} associated to the chain map $\varepsilon^\bullet$.
We equip $\Cone^n(\varepsilon)$ with the norm
\[
\|(u,v)\|=\max \left\{ \|u\|, \|v\| \right\},
\]
which induces a seminorm on the cohomology module $H^n(\Cone(\varepsilon)^\bullet)$.
We observe that every commutative diagram of chain maps
\begin{equation}
	\label{eq: commuttaive diagram chain maps}
	% https://tikzcd.yichuanshen.de/#N4Igdg9gJgpgziAXAbVABwnAlgFyxMJZABgBpiBdUkANwEMAbAVxiRAB12IaYAnBrGBjAAsgF8AepwBGTBgxg4QY0uky58hFAEZyVWoxZtO3PgKHAAcpJlyFSlWux4CRMtv31mrRCBFT2WXlFZVUQDGdNIl0Pai8jX0sAoPtlfRgoAHN4IlAAM14IAFskMhAcCCRdA29jQLpeYE5GNAALOhtAuxDHEALipAAmagqkAGY4wx8Odhb25O6HMP6SxDLRxGGahJnpRToAfX9bYKX8wtXqjYnt6Zl9g6ST1LEKMSA
	\begin{tikzcd}
		\overline{M}^\bullet \arrow[r, "\bar{\varepsilon}^\bullet"] \arrow[d, "\eta_M^\bullet"] & \overline{N}^\bullet \arrow[d, "\eta_N^\bullet"] \\
		M^\bullet \arrow[r, "\varepsilon^\bullet"]                                               & N^\bullet                                        
	\end{tikzcd}
\end{equation}
induces a chain map
\[
\eta^n \colon \Cone(\bar{\varepsilon})^n \rightarrow \Cone(\varepsilon)^n, \qquad (u,v)\mapsto (\eta^n_M(u), \eta^{n-1}_N(v)).
\]
Let us denote by $\Sigma N^\bullet$ the \emph{suspension} of $N^\bullet$, i.e. the complex obtained by setting $(\Sigma N)^n=N^{n-1}$, and consider the short exact sequence of complexes
\[
0\rightarrow\Sigma N^\bullet \overset{\iota}{\rightarrow} \Cone(\varepsilon)^\bullet \overset{\xi}{\rightarrow} M^\bullet \rightarrow 0,
\]
where $\iota(v)=(0,v)$ and $\xi(u,v)=u$.
Of course, we have $H^n(\Sigma N^\bullet)=H^{n-1}(N^\bullet)$, therefore the corresponding long exact sequence is given by 
\[
\dots \rightarrow H^n(\Cone(\varepsilon)^\bullet)\rightarrow H^n(M^\bullet)\rightarrow H^n(N^\bullet) \rightarrow H^{n+1}(\Cone(\varepsilon)^\bullet) \rightarrow \dots
\]
By the naturality of the long exact sequence, the following diagram
\[
% https://tikzcd.yichuanshen.de/#N4Igdg9gJgpgziAXAbVABwnAlgFyxMJZABgBpiBdUkANwEMAbAVxiRAAkA9YMAWgEYAvgAoAOqIg0YAJwZYwMYAFlBncQCMmDBjBwBKEINLpMufIRT9yVWoxZsuPASPGSZchcAByqjVp36hsYgGNh4BEQAzNbU9MysiBzcYC4SUrLyiipqopraugZGJmHmRAAsMbbxDsmpbhmePjl5AYXBoWYRKGT8NnH2iY58QsJKzf4FQcWdFshWvbF2CUlOI17j+YFFIabhs9ELVQMrKaMbrVM7JV3IFYf9y0Mi636bbdN7RABMlQ81YGJRABhAgwQHqOjSYDiRhoAAWdEEenOk22HU+KB+9yW-0BIIUgNhCORrwughsMCgAHN4ERQAAzaQQAC2SDIIBwECQQmCjJZSAqHK5iAArNs+azEFYhUgABzipmS2XUTlIL4K-mIH4yxCRDWSgBsKuFAHZ9UgAJzGpAG82i62IC126Wq+1HR7JZzg3R0AD6L1yEy2vMVbIdgr+g09Iw0Pt9Y1JqJDmu1rqN7pqPFS6jjCcDb0uEqQ0R1JsW1SjWe9OD9AZaSYZocQyp1VozUYBsZrKK2FEEQA
\begin{tikzcd}
	H^{n-1}(\overline{M}^\bullet) \arrow[r] \arrow[d, "H^{n-1}(\eta_M^\bullet)"] & H^{n-1}(\overline{N}^\bullet) \arrow[r] \arrow[d, "H^{n-1}(\eta_N^\bullet)"] & H^n(\Cone(\bar{\varepsilon})^\bullet) \arrow[r] \arrow[d, "H^n(\eta^\bullet)"] & H^{n}(\overline{M}^\bullet) \arrow[r] \arrow[d, "H^{n}(\eta_M^\bullet)"] & H^{n}(\overline{N}^\bullet) \arrow[d, "H^{n}(\eta_N^\bullet)"] \\
	H^{n-1}(M^\bullet) \arrow[r]                                                  & H^{n-1}(N^\bullet) \arrow[r]                                                  & H^n(\Cone(\varepsilon)^\bullet) \arrow[r]                                       & H^{n}(M^\bullet) \arrow[r]                                                & H^{n}(N^\bullet)                                               
\end{tikzcd}
\]
is commutative and has exact rows.
Therefore, by the Five Lemma, we know that if $\eta_M^\bullet$ and $\eta_N^\bullet$ induce isomorphisms in cohomology, then also $\eta^\bullet$ does. 
In order to retain further control over the norms we need to make the maps at the level of cochains as explicit as possible. 
For this reason we introduce the following property.
\begin{defn}
	\label{defn: C_n-chain homotopy inverse}
	Let $(K_n)$, $n \in \N_{\geq 1}$, be a sequence of positive real numbers and let $\eta_M^\bullet \colon \overline{M}^\bullet \rightarrow M^\bullet$ be a norm non-increasing chain map.
	We say that $\eta_M^\bullet$ has a \emph{$(K_n)$-chain homotopy inverse} if there exists a norm non-increasing chain map $A^\bullet_M \colon M^\bullet \rightarrow \overline{M}^\bullet$ such that the following conditions hold:
	\begin{itemize}
		\item[(i)] $A_M^\bullet \circ \eta^\bullet_M=\id$;
		\item[(ii)] The composition $\eta_M^\bullet \circ A^\bullet_M$ is chain homotopic to the identity via an algebraic homotopy $T^\bullet \colon M^\bullet \rightarrow M^{\bullet-1}$ such that $\|T^n\|\leq K_n$ for every $n \in \N_{\geq 1}$. 
		In particular, for every $u \in M^n$ such that $\delta_M^n(u)=0$, there exists $u'\in M^{n-1}$ such that 
		\[
		\eta^n_M\circ A^n_M(u)-u = \delta^{n-1}_M(u'), 	\qquad \|u'\|\leq K_n\cdot \|u\|.
		\]
	\end{itemize}
\end{defn}
Let $G$ be an amenable group acting on a multicomplex $K$ by automorphisms which are simplicially homotopic to the identity. It follows from Proposition \ref{proposition: amenable - invariant resolutions} that the inclusion $C^\bullet_b(K)^G\hookrightarrow C^\bullet_b(K)$ of invariant cochains admits a $(C_n)$-chain homotopy inverse, where the sequence $(C_n)$ is introduced in Definition \ref{defn: constants C_n}.
Another important example of a chain map admitting a $(C_n)$-chain homotopy inverse is given by the map $C^\bullet_b(\K(X))\rightarrow C^\bullet_b(\Elle(X))$ induced by the inclusion $\Elle(X)\hookrightarrow \K(X)$ for every CW-complex $X$ (see Lemma \ref{lem: r_X has C_n chain homotopy inverse} below).
These examples play a fundamental role in our proof of Theorem \ref{thm: relative mapping theorem, biLipschitz version}.
\begin{lemma}
	\label{lem: les and five lemma for mapping cones}
	Consider the commutative diagram (\ref{eq: commuttaive diagram chain maps}) and assume that $\eta_M^\bullet$, $\eta_N^\bullet$, $\varepsilon^\bullet$ and $\bar{\varepsilon}^\bullet$ are norm non-increasing.
	If there exists a sequence $(K_n)$, $n \in \N_{\geq 1}$, of positive real numbers such that the following conditions hold:
	\begin{enumerate}
		\item[$(1)$] $\eta_M^\bullet$ has a $(K_n)$-chain homotopy inverse $A_M^\bullet$;
		\item[$(2)$] $\eta_N^\bullet$ has a $(K_n)$-chain homotopy inverse $A_N^\bullet$;
	\end{enumerate}
	then $\eta^\bullet$ induces a bi-Lipschitz isomorphism in cohomology such that
	\[
	(2 K_n)^{-1}\cdot \|\alpha \| \leq \|H^n(\eta^\bullet)(\alpha)\| \leq \|\alpha \|,
	\]
	for every $n \in \N$ and every $\alpha \in H^n(\Cone(\bar{\varepsilon})^\bullet)$.
\end{lemma}
\begin{proof}
	Conditions (1) and (2) clearly imply that $\eta^\bullet_M$ and $\eta^\bullet_N$ induce isometric isomorphisms in cohomology.
	Moreover, if $\eta^\bullet_M$ and $\eta^\bullet_N$ are norm non-increasing, then also $\eta^\bullet$ is norm non-increasing, hence also the induced isomorphism in cohomology is so.
	In order to retain further control over the norms, we proceed as follows.
	We take a class in $H^n(\Cone(\varepsilon)^\bullet)$ and we consider a cocycle $(u,v)\in \Cone(\varepsilon)^n$ representing it. 
	Since $(u,v)$ is a cocycle, by the definition of the differential of $\Cone(\varepsilon)^\bullet$, we have that
	\[
	\delta_M^n(u)= 0, \qquad \varepsilon^n(u)+\delta^{n-1}_N(v)=0.
	\]
	Since $\eta^\bullet_M$ has a $(K_n)$-chain homotopy inverse, there exists $u'\in M^{n-1}$ such that 
	\[
	\eta^n_M\circ A^n_M(u)-u - \delta^{n-1}_M(u') = 0, \qquad \|u'\|\leq K_n\cdot \|u\|.
	\]
	By applying $\varepsilon^n$ and $A^n_N$ on both sides of the equality, we get
	\begin{align*}
		0 & = \varepsilon^n \circ \eta^n_M\circ A^n_M(u)-\varepsilon^n(u) - \varepsilon^n \circ \delta^{n-1}_M(u')\\
		0 & = \eta^n_N \circ \bar{\varepsilon}^n \circ A^n_M(u) + \delta^{n-1}_N(v) - \delta^{n-1}_N \circ \varepsilon^{n-1}(u') \\
		0 & = A^n_N\circ \eta^n_N \circ \bar{\varepsilon}^n \circ A^n_M(u) + A^n_N\circ \delta^{n-1}_N(v) - A^n_N\circ \delta^{n-1}_N \circ \varepsilon^{n-1}(u') \\
		0 & = \bar{\varepsilon}^n \circ A^n_M(u) + \delta^{n-1}_{\overline{N}} \circ A^{n-1}_N (v) - \delta^{n-1}_{\overline{N}} \circ A^{n-1}_N \circ \varepsilon^{n-1}(u') \\
		0 & = \bar{\varepsilon}^n \circ A^n_M(u) + \delta^{n-1}_{\overline{N}} \circ A^{n-1}_N (v-\varepsilon^{n-1}(u')),
	\end{align*}
	where $\delta^{n-1}_{\overline{N}}$ denotes the differential of the cochain complex $\overline{N}^\bullet$.
	It readily follows that 
	\[
	(z,w)=(A^n_M(u), A^{n-1}_N(v- \varepsilon^{n-1}(u'))) \in \Cone(\bar{\varepsilon})^n
	\]
	is a cocycle.
	Moreover, we have that
	\begin{align*}
		& \eta^n(z,w) - (u,v) \\
		& =  (\eta^n_M\circ A^n_M(u), \eta^{n-1}_N \circ A^{n-1}_N(v-\varepsilon^{n-1}(u'))) - (u,v) \\
		& = (\delta_M^{n-1}(u'), \eta^{n-1}_N \circ A^{n-1}_N(v-\varepsilon^{n-1}(u'))-(v - \varepsilon^{n-1}(u'))- \varepsilon^{n-1}(u')) \\
		& = (0, x) + \delta_C^{n-1}(u',0),
	\end{align*}
	where $x= \eta^{n-1}_N \circ A^{n-1}_N(v-\varepsilon^{n-1}(u'))-(v - \varepsilon^{n-1}(u'))$. Since both $\eta^n(z,w)$ and $(u,v)$ are cocycles, it follows that also $(0,x)$ is a cocycle, hence $\delta_N^{n-1}(x)=0$.
	Since $\eta^{\bullet}_N$ has a $(K_{n})$-chain homotopy inverse, there exists $x'\in N^{n-1}$ such that
	\[
	\eta_N^{n-1}\circ A_N^{n-1}(x) - x = \delta_N^{n-2}(x').
	\]
	It is easy to check that $\eta_N^{n-1}\circ A_N^{n-1}(x)=0$, hence $x = -\delta_N^{n-2}(x')$.
	It follows that
	\begin{align*}
		\eta^n(z,w) - (u,v)& =  (0,x) + \delta_C^{n-1}(u',0) \\
		& = \delta_C^{n-1}(0,x') +  \delta_C^{n-1}(u',0) = \delta_C^{n-1}(u',x'),
	\end{align*}
	which implies that $\eta^n(z,w)$ and $(u,v)$ represent the same class in cohomology.
	In conclusion, from the following inequalities
	\begin{align*}
		\|(z,w)\| & = \max \left\{\|z\|,\|w\| \right\} \\
		& = \max \left\{\|A^n_M(u) \|,\|A^{n-1}_N(\varepsilon^{n-1}(u')-v)\| \right\} \\
		& \leq \max \left\{\|u \|,\|\alpha^{n-1}(u')-v\| \right\} \\
		& \leq \max \left\{\|u \|,\|u'\| + \|v\| \right\} \\
		& \leq \max \left\{\|u \|,K_n \cdot \|u\| + \|v\| \right\} \\
		& = K_n \cdot \|u\| + \|v\| \\
		& \leq K_n\cdot (\|u\| + \|v\|)\\
		& \leq 2 K_n \cdot \max \left\{\|u\|,\|v\| \right\} = 2 K_n\cdot \|(u,v)\|,
	\end{align*}
	we can deduce the estimates over the norms in the statement.
\end{proof}

\begin{rem}
	\label{rem: averaging is not relative}
	In the assumptions of Lemma \ref{lem: les and five lemma for mapping cones} we are \emph{not} requiring that $A^\bullet_M$ and $A^\bullet_N$ commute with the chain maps $\varepsilon^\bullet$ and $\bar{\varepsilon}^\bullet$, i.e. $A^\bullet_N \circ \varepsilon^\bullet = \bar{\varepsilon}^\bullet \circ A^\bullet_M$. If this were the case, it would be easy to exhibit a norm non-increasing chain map which is a left inverse of $\eta^\bullet$, thus obtaining that $\eta^\bullet$ induces an isometric isomorphism in cohomology.
\end{rem}
\subsection{Mapping cones of pairs of multicomplexes}
\label{subsec: mapping cones of pairs}
Let $f \colon L\rightarrow K$ be a simplicial map between multicomplexes.
We denote by 
\[
C^\bullet_b(f\colon L\rightarrow K)
\]
the mapping cone complex associated to the cochain map $C^\bullet_b(K)\rightarrow C^\bullet_b(L)$ induced by $f$, and we denote by $H^\bullet_b(f\colon L\rightarrow K)$ its cohomology.
One can use the very same argument in \cite[Theorem 3.19]{Par03} to prove the following.
\begin{lemma}
	\label{lem: beta induces bilipschitz isomorphism}
	Let $(K,L)$ be a pair of multicomplexes, and let $L\hookrightarrow K$ be the inclusion map.
	The chain map 
	\[
	\beta^n\colon C^n_b(K,L)\rightarrow C^n_b(L \hookrightarrow K), \qquad \beta^n(u)=(u,0),
	\]
	induces, for every $n \in\N$, a bi-Lipschitz isomorphism of vector spaces
	\[
	H^n(\beta^\bullet)\colon H^n_b(K,L)\rightarrow H^n_b(L \hookrightarrow K),
	\]
	such that, for every $\alpha \in H^n_b(K,L)$,
	\[
	(n+2)^{-1}\cdot \|\alpha \|_\infty \leq \| H^n_b(\beta^\bullet)(\alpha)\|_\infty \leq \|\alpha\|_\infty.
	\]
\end{lemma}

\section{Proof of Theorem \ref{thm: relative mapping theorem, biLipschitz version}}
\label{sec: proof of mapping theorem, bilipschitz}
Let $(X,A)$ be a CW-pair such that the kernel of the morphism $\pi_1(A\hookrightarrow X, x)$ is amenable for every $x \in A$.
Unlike in Section \ref{sec: proof of mapping theorem, isometric}, we are \emph{not} requiring the pair $(X,A)$ to be good.
In Section \ref{sec: pair of multicomplexes for pairs of spaces} we have constructed a well-defined pair of aspherical multicomplexes $(\A(X),\A_X(A))$.
Let $n \in \N$. Since the path-connected components of $\A(X)$ (resp. $\A_X(A)$) bijectively correspond to the path-connected components of $X$ (resp. $A$), the statement of Theorem \ref{thm: relative mapping theorem, biLipschitz version} is trivially true for $n=0$. We assume therefore that $n \in \N_{\geq 1}$.
We want to show that there is a bi-Lipschitz isomorphism of vector spaces
\[
\Psi^n\colon H^n_b(\A(X),\A_X(A))\rightarrow H^n_b(X,A),
\]
such that, for every $\alpha \in H^n_b(\A(X),\A_X(A))$,
\[
(2C_n)^{-2}(n+2)^{-1} \cdot \|\alpha \|_\infty \leq \| \Psi^n(\alpha)\|_\infty \leq 2C_n(n+2)\cdot \|\alpha\|_\infty,
\]
where the constant $C_n$ denotes the number of $n$-simplices in the natural structure of multicomplex of the $n$-dimensional prism $\Delta^{n-1}\times [0,1]$ (see Definition \ref{defn: constants C_n}).
Recall from Section \ref{sec: pair of multicomplexes for pairs of spaces} that we have the following commutative diagram of simplicial maps
\begin{equation}
	\label{eq: diagram simplicial maps for pairs}
	% https://tikzcd.yichuanshen.de/#N4Igdg9gJgpgziAXAbVABwnAlgFyxMJZABgBpiBdUkANwEMAbAVxiRAB12BpACgEEAlCAC+pdJlz5CKAIzkqtRizacAogwYx+Q0eOx4CRMjIX1mrRB248AGjrEgM+qUTknqZ5ZbUatdkQ5OkoYoAEzyHkoWVnzaAXrB0sjh7ormKuyx-rqOEgZJAMwRaV4xAPo2cTlB+URFqZ7RnFk6CjBQAObwRKAAZgBOEAC2SHIgOBBIZON0WAxsABYQEADWjDggkemWWGV88SADw0hF45OI4TNzi8trDBtbpbs2B0cjiNMTSJeNbABWZU4XE2V3mliWq1eg3eYy+iFOv0sAJ8mihx0QsPOABZHk12GgsGj3qc4QBWXEZAlEpA4s5IcklaLIzIgnCzMHgAisHJvGnUOEANgpljQ1MQDLhAHZhVYsFAxUK6YhpaCbpDhBRhEA
	\begin{tikzcd}
		\K(A) \arrow[d, "j_\K", hook] & \Elle(A) \arrow[l, "i_A"', hook'] \arrow[d, "j_\Elle"] \arrow[r, "\pi"] & \A(A) \arrow[d, "j_\A"] \arrow[r, "j_\A"] & \A_X(A) \arrow[d, hook] \\
		\K(X)                         & \Elle(X) \arrow[l, "i_X"', hook'] \arrow[r, "\pi"]                      & \A(X) \arrow[r, "\id"]                 & \A(X).                  
	\end{tikzcd}
\end{equation}
Our strategy to prove Theorem \ref{thm: relative mapping theorem, biLipschitz version} is to pass through mapping cones and invoke Lemma \ref{lem: les and five lemma for mapping cones} for every square of diagram (\ref{eq: diagram simplicial maps for pairs}).
To this end, we first show that the maps $i_X$, $i_A$, $\pi$ and $j_\A$ have $(C_n)$-chain homotopy inverses (see Definition \ref{defn: C_n-chain homotopy inverse}).
\begin{lemma}
	\label{lem: r_X has C_n chain homotopy inverse}
	Let $X$ be a CW-complex. Then the map induced by $i_X\colon \Elle(X)\hookrightarrow \K(X)$ on bounded cochains has a $(C_n)$-chain homotopy inverse.
\end{lemma}
\begin{proof}
	We know that there exists a simplicial retraction $r_X\colon \K(X)\rightarrow \Elle(X)$ such that $r_X\circ i_X$ is the identity of $\Elle(X)$ \cite[Theorem 3.23]{FM23}. 
	We claim that the chain map $r_X^\bullet \colon C^\bullet_b(\Elle(X))\rightarrow C^\bullet_b(\K(X))$ induced by $r_X$ provides a $(C_n)$-chain homotopy inverse of $i_X^\bullet$.
	In fact, the geometric realization of $r_X$ realizes $|\Elle(X)|$ as a strong deformation retract of $|\K(X)|$. 
	Since $\K(X)$ is large and complete, it follows from the Homotopy Lemma \cite[Lemma 3.17]{FM23} that $i_X \circ r_X$ is simplicially homotopic to the identity.
	Therefore, by Remark \ref{rem: algebraic homotopy is bounded}, there exists an algebraic homotopy $T^\bullet\colon C^\bullet_b(\K(X))\rightarrow C^{\bullet-1}_b(\K(X))$ between $r_X^\bullet \circ i_X^\bullet$ and the identity such that $\|T^n\|\leq C_n$, for every $n \in \N$.
\end{proof}
\begin{lemma}
	\label{lem: pi has C_n chain homotopy inverse}
	Let $X$ be a CW-complex. Then the map induced by $\pi\colon \Elle(X)\rightarrow \A(X)$ on bounded cochains has a $(C_n)$-chain homotopy inverse.
\end{lemma}
\begin{proof}
	By \cite[Corollary 4.15]{FM23}, we know that $\A(X)=\Elle(X)/\Gamma_1$, where $\Gamma_1$ denotes the group $\Gamma_1(\Elle(X))$ (see Definition \ref{definition: group of simplicial automorphisms}). It follows that the projection $\pi\colon \Elle(X)\rightarrow \A(X)$ induces the inclusion
	\[
	\iota^\bullet \colon C^\bullet_b(\Elle(X))^{\Gamma_1}\hookrightarrow C^\bullet_b(\Elle(X))
	\]
	of invariant cochains. We know that the groups $\Gamma_1/\Gamma_m$, $m \in \N$, are amenable (Lemma \ref{lemma: amenability of Gamma(K,L)}). 
	Moreover $\Gamma_1$ acts by definition on $\Elle(X)$ by automorphisms which are simplicially homotopic to the identity.
	In conclusion, we obtain by Proposition \ref{proposition: amenable - invariant resolutions} that $\iota^\bullet$ admits a $(C_n)$-chain homotopy inverse.
\end{proof}
\begin{lemma}
	\label{lem: j_A has C_n chain homotopy inverse}
	Let $(X,A)$ be a CW-pair such that the kernel of the morphism $\pi_1(A\hookrightarrow X, x)$ is amenable for every $x \in A$. Then the map induced by $j_\A\colon \A(A)\rightarrow \A_X(A)$ on bounded cochains has a $(C_n)$-chain homotopy inverse.
\end{lemma}
\begin{proof}
	By Lemma \ref{lemma: homotopy tyoe  pe of A_X(A)}, we know that the projection map $j_\A$ can be described as the quotient of $\A(A)$ by the action of the group
	\[
	\hat{N}= \bigoplus_{x \in A} N_x \leq \Pi(A,A).
	\]
	Here $N_x$ denotes the kernel of the morphism $\pi_1(A\hookrightarrow X,x)$, which is amenable by assumption. Since direct products of amenable groups is amenable, we have that $\hat{N}$ is amenable. Moreover, we know that the action $\hat{N}\acts \A(A)$ is by automorphisms which are simplicially homotopic to the identity \cite[Theorem 5.3]{FM23}.
	Hence we conclude by Proposition \ref{proposition: amenable - invariant resolutions}.
\end{proof}
We proceed with the proof of Theorem \ref{thm: relative mapping theorem, biLipschitz version}.
We know that the projection $S\colon (|\K(X)|,|\K(A)|)\rightarrow (X,A)$ induces a homotopy equivalence of pairs (Proposition \ref{prop: S_X:K(X) to X induces homotopy equivalence of pairs}), hence an isometric isomorphism 
\[
H^n_b(X,A)\rightarrow H^n_b(|\K(X)|,|\K(A)|).
\]
Moreover, by the Relative Isometry Lemma (Proposition \ref{proposition: relative isometry lemma}) we know that the natural chain inclusion $\varphi_\bullet\colon C_\bullet(\K(X))\rightarrow C_\bullet(|\K(X)|)$ induces an isometric isomorphism
\[
H^n_b(|\K(X)|,|\K(A)|) \rightarrow H^n_b(\K(X),\K(A)).
\]
By Lemma \ref{lem: beta induces bilipschitz isomorphism}, there is a bi-Lipschitz isomorphism of vector spaces
\[
H^n_b(\K(X),\K(A))\rightarrow H^n_b(j_\K\colon \K(A)\hookrightarrow \K(X)).
\]
By Lemma \ref{lem: r_X has C_n chain homotopy inverse}, Lemma \ref{lem: j_A has C_n chain homotopy inverse} and Lemma \ref{lem: pi has C_n chain homotopy inverse} above, we can invoke Lemma \ref{lem: les and five lemma for mapping cones} for every square of diagram (\ref{eq: diagram simplicial maps for pairs}). Therefore, we get the following bi-Lipschitz isomorphisms of vector spaces
\begin{align*}
	H^n_b(j_\K\colon\K(A)\hookrightarrow \K(X)) & \rightarrow H^n_b(j_\Elle\colon \Elle(A)\rightarrow \Elle(X)), \\
	H^n_b(j_\A \colon \A(\A)\rightarrow \A(X)) & \rightarrow H^n_b(j_\Elle\colon \Elle(A)\rightarrow \Elle(X)), \\
	H^n_b(\A_X(A) \hookrightarrow \A(X)) & \label{eq: equivalence theorem, second losing norm} \rightarrow H^n_b(j_\A \colon \A(\A)\rightarrow \A(X)). 
\end{align*}
In conclusion, since that pair of multicomplexes $(\A(X),\A_X(A))$ is well-defined, by Lemma \ref{lem: beta induces bilipschitz isomorphism}, we have a bi-Lipschitz isomorphism
\[
H^n_b(\A(X),\A_X(A)) \rightarrow H^n_b(\A_X(A)\hookrightarrow \A(X)).
\]
By carefully keeping track of the norms, we get the estimates in the statement.
\begin{rem}
	\label{rem: not optimality of constant}
	The bi-Lipschitz constants appearing in Theorem \ref{thm: relative mapping theorem, biLipschitz version} are far from being optimal. For example, if $A$ is $\pi_1$-injective in $X$, then we know that $\A_X(A)$ is simplicially isomorphic to $\A(A)$, hence embeds into $\A(X)$ (Proposition \ref{prop: j_A is a simplicial embedding}). In this case the passage through mapping cones in the rightmost square of diagram (\ref{eq: diagram simplicial maps for pairs}) is unnecessary. It follows that $\Psi^n$ actually satisfies the following inequalities: for every $\alpha \in H^n_b(\A(X),\A(A))$,
	\[
	(2C_n)^{-1}(n+2)^{-1} \cdot \|\alpha \|_\infty \leq \| \Psi^n(\alpha)\|_\infty \leq 2C_n(n+2)\cdot \|\alpha\|_\infty.
	\]
\end{rem}
\section{Regularity of simplicial actions on pairs of multicomplexes}
\label{sec: regularity of simplicial actions}

Let $(K,L)$ be a pair of multicomplexes and let $G$ be a group acting simplicially on $K$.
If the action $G\acts K$ preserves $L$, we know by Proposition \ref{proposition: amenable - invariant resolutions} that every relative bounded coclass in $H^\bullet_b(K,L)$ can be represented by a $G$-invariant cocycle.
The goal of this section is obtain a similar result in the case in which the action $G\acts K$ does \emph{not} preserve $L$.
To this end, we introduce a regularity condition of the action $G\acts K$ on $L$ (Definition \ref{defn: having orbits in L induced by H}).
Unlike in Proposition \ref{proposition: amenable - invariant resolutions}, in this setting we are only able to retain a bi-Lipschitz control over the norms.

Given an action $G\acts K$, we consider the induced action $G\acts C_\bullet(K)$ on algebraic simplices, which in turns induces an action $G\acts C^\bullet_b(K)$ on bounded cochains.
If $H$ is a subgroup of $G$ whose action on $K$ preserves a submulticomplex $L$ of $K$, there is an obvious restriction map \[r_{G,H}^\bullet\colon C^\bullet_b(K)^G\rightarrow C^\bullet_b(L)^H,\] whose kernel is denoted by $C^\bullet_b(K,L)^{G,H}$.
Unlike the case described in Proposition \ref{proposition: amenable - invariant resolutions}, $H$ could be a proper subgroup of $G$.
\begin{defn}
	\label{defn: having orbits in L induced by H}
	Let $(K,L)$ be a pair of multicomplexes. Let $G\acts K$ be a simplicial action and let $H$ be a subgroup of $G$ which induces an action $H\acts (K,L)$.
	We say that the action \emph{$G\acts K$ has orbits in $L$ induced by $H$} if, for every \emph{algebraic} simplex $\sigma$ of $L$ and every $g \in G$ such that $g\cdot \sigma$ is an algebraic simplex of $L$, there exists $h \in H$ such that $h\cdot \sigma = g \cdot \sigma$.
\end{defn}
The following result shows that, under the previous condition, the restriction map $r_{G,H}^\bullet$ is surjective, and will allow us, under suitable circumstances, to take $G$-invariant cocycles in every relative coclass of $H^\bullet_b(K,L)$.
The constant $C_n$ appearing in the following is described in Definition \ref{defn: constants C_n}.
\begin{lemma}
	\label{lemma: invariant cochains}
	Let $(K,L)$ be a pair of multicomplexes. Let $G\acts K$ be a simplicial action and let $H$ be a subgroup of $G$ which induces an action $H\acts (K,L)$.
	Assume that the orbits of $G$ in $L$ are induced by $H$. Then the rows of the following commutative diagram are both exact
	\[
	% https://tikzcd.yichuanshen.de/#N4Igdg9gJgpgziAXAbVABwnAlgFyxMJZABgBpiBdUkANwEMAbAVxiRGJAF9T1Nd9CKAIzkqtRizYBhAHoAdOQCMmDBjBwB9RQAoA0qQAyAShnAA4qQASnLjxAZseAkQBMo6vWatEIWQuWq6lp6Jma2vI4CRADM7uJe0vJKKmqaOsYyluH2fE6CyAAscZ6SPhzcEfzOKEVCYiXe7NkOVflkdR4SjeV2LXlEIh3xpb5JAanB+sbNuVEobkMNif4pQTq6RjOR1cixi13LyYFp2tOcYjBQAObwRKAAZgBOEAC2SGQgOBBIQhUgT68ftQvkgXH8AW9EG5Pt9ENFwc9IbEYUgCgjAYgAGzA2EAdnRkNxOKQAA4CaTiYgAJydBI+R4acxWGzk6mUgCsrORIOprOhPLJdghQJRiHxFE4QA
	\begin{tikzcd}
		0 \arrow[r] & {C^\bullet_b(K,L)^{G,H}} \arrow[r] \arrow[d, "{\iota_{G,H}^\bullet}"] & C^\bullet_b(K)^G \arrow[r, "{r_{G,H}^\bullet}"] \arrow[d, "{\iota_{G}^\bullet}"]  & C^\bullet_b(L)^H \arrow[r] \arrow[d, "{\iota_{H}^\bullet}"] & 0 \\
		0 \arrow[r] & {C^\bullet_b(K,L)} \arrow[r]                 & C^\bullet_b(K) \arrow[r, "r^\bullet"] & C^\bullet_b(L) \arrow[r]             & 0
	\end{tikzcd}
	\]
	where the vertical arrows are inclusions of cochains. Moreover, if $G$ is amenable and the actions $G\acts K$ and $H\acts L$ are by automorphisms which are simplicially homotopic to the identity, then $\iota_{G,H}^\bullet$ induces a bi-Lipschitz isomorphism in cohomology such that
	\[
	(n+2)^{-1}(2 C_n)^{-1} \cdot \|\alpha \|_\infty \leq \| H^n_b(\iota_{G,H}^\bullet)(\alpha)\|_\infty \leq \|\alpha\|_\infty,
	\]
	for every $n \in \N$ and every $\alpha \in H^n(C^\bullet_b(K,L)^{G,H})$.
\end{lemma}
\begin{proof}
	First of all, we show that we can extend $H$-invariant chains on $L$ to $G$-invariant chains on $K$, i.e. $r_{G,H}^\bullet$ is surjective. To this end our conditions on the action $G\acts K$ are crucial.
	Let $n \in\N$ and let $z \in C^n_b(L)^H$. We define a cochain $z' \in C^n_b(K)^G$ in the following way: for every algebraic $n$-simplex $\sigma$ of $K$, we set $z'(\sigma)= z(\hat{\sigma})$, if $\sigma = g\cdot \hat{\sigma}$ for some $g \in G$ and $\hat{\sigma}\in C_n(L)$, and $z'(\sigma)=0$, otherwise.
	Since the action $G\acts K$ has orbits in $L$ induced by $H$ and $z$ is $H$-invariant, it is easy to check that $z'$ is indeed well-defined.
	Moreover it is clear from the construction that $z'$ is $G$-invariant and satisfies $r_{G,H}^n(z')=z$.
	
	Assume now that $G$ (hence $H$) is amenable and the actions $G\acts K$ and $H\acts L$ are by automorphisms which are simplicially homotopic to the identity. It follows that $\iota_G^\bullet$ and $\iota^\bullet_H$ induce isometric isomorphisms in cohomology (Proposition \ref{proposition: amenable - invariant resolutions}). Therefore, as an application of the Five Lemma, we can already deduce that $\iota_{G,H}^\bullet$ induces an isomorphism in cohomology. 
	Moreover, $H^n_b(\iota^\bullet_{G,H}$) is of course norm non-increasing.
	In order to retain further control over the norms, we need to pass through mapping cones in the following way.
	We consider the mapping cone complex $\Cone(r_{G,H})^\bullet$ associated to the restriction map $r^\bullet_{G,H}$, and we denote by $\bar{\delta}^\bullet$ its differential.
	Similarly, we consider the mapping cone complex
	\[
	\Cone(r)^\bullet = C^\bullet_b(L\hookrightarrow K)
	\]
	associated to the restriction map $r^\bullet\colon C^\bullet_b(K)\rightarrow C^\bullet_b(L)$.
	There is a natural chain map induced by the inclusion of invariant cochains:
	\[
	\varphi^n\colon \Cone(r_{G,H})^n \rightarrow \Cone(r)^n, \qquad (u,v)\rightarrow (\iota^n_G(u),\iota^{n-1}_H(v)).
	\]
	We consider the following commutative diagram
	\[
	% https://tikzcd.yichuanshen.de/#N4Igdg9gJgpgziAXAbVABwnAlgFyxMJZABgBpiBdUkANwEMAbAVxiRAGEA9MAfQCMAFAGlSAGQCUnYAHFSACQC+IBaXSZc+QigCM5KrUYs2XXoJETlqkBmx4CRMtv31mrRB279hAHW8ALCAgAawAnLABzPxw6EJCIAHcAAgkpWUVLNVtNIl0nahcjdxMvIV8A4LDI6NiE5PFlfRgocPgiUAAzOIBbJDIQHAgkXX66LAY2cqCQfMM3EF98aM8ZeSUVDu7e6gGkACYZ1zZfPhgl3hX09ZBOiB7EYZ3EAGYDwvnvE7OM683Eff7Bs9XnNfPQQmg-FhuA0FEA
	\begin{tikzcd}
		{C^n_b(K,L)^{G,H}} \arrow[r, "{\iota^n_{G,H}}", hook] \arrow[d, "{\beta^n_{G,H}}"] & {C^n_b(K,L)} \arrow[d, "\beta^n"] \\
		{\Cone(r_{G,H})^n} \arrow[r, "\varphi^n"]                           & \Cone(r)^n,        
	\end{tikzcd}
	\]
	where $\beta^n$ was defined in Subsection \ref{subsec: mapping cones of pairs} and $\beta^n_{G,H}$ is defined by the formula \[\beta^n_{G,H}(u)=(u,0).\]
	Of course, if $\beta^\bullet$, $\beta^\bullet_{G,H}$ and $\varphi^\bullet$ induce bi-Lipschitz isomorphisms in cohomology, then also $\iota^\bullet_{G,H}$ does.
	By Lemma \ref{lem: beta induces bilipschitz isomorphism}, we know that $\beta^\bullet$ induces a bi-Lipschitz isomorphism in cohomology.
	We now show that $\beta^\bullet_{G,H}$ induces a bi-Lipschitz isomorphism in cohomology such that
	\begin{equation}
		\label{eq: beta_G,H is bilipschitz}
		(n+2)^{-1} \cdot \|\alpha \|_\infty \leq \| H^n_b(\beta^\bullet_{G,H})(\alpha)\|_\infty \leq \|\alpha\|_\infty,
	\end{equation}
 	for every $\alpha \in H^n(C^\bullet_b(K,L)^{G,H})$.
	In order to prove that $H^n_b(\beta^\bullet_{G,H})$ is surjective, we take a class in $H^n(\Cone(r_{G,H})^\bullet)$ and we consider a cocycle $(u,v) \in \Cone(r_{G,H})^n$ representing it.
	The extension $v'\in C^{n-1}_b(K)^G$ of $v \in C^{n-1}_b(L)^H$ is well-defined because the action $G\acts K$ has orbits in $L$ induced by $H$. 
	Hence, we can consider the cochain $u+\delta^{n-1} v' \in C^n_b(K)^{G}$. 
	Since $(u,v)$ is a cocycle, it follows that $u + \delta^{n-1} v' \in C^n_b(K,L)^{G,H}$ is a relative cocycle.
	Moreover, since 
	\[
	(u+\delta^{n-1} v',0)- (u,v)=\bar{\delta}^{n-1}(v', 0),
	\]
	we obtain that $\beta^n_{G,H}(u + \delta v')$ and $(u,v)$ represent the same class in cohomology.
	This shows that $H^n_b(\beta^\bullet_{G,H})$ is surjective.
	Moreover, if $f\in C^n_b(K,L)$ is a relative bounded cocycle such that $\beta^n_{G,H}(f)=(f,0)=\bar{\delta}^{n-1}(z,w)$, then $z+\delta^{n-2}w'$ belongs to $C^{n-1}_b(K,L)^{G,H}$ and $\delta^{n-1}(z+\delta^{n-2}w')=f$. Therefore $H^n_b(\beta^\bullet_{G,H})$ is also injective, hence an isomorphism.
	In order to prove (\ref{eq: beta_G,H is bilipschitz}), we first observe that $H^n_b(\beta^\bullet_{G,H})$ is norm non-increasing.	
	On the other hand, since $\beta^n_{G,H}(u + \delta v')$ and $(u,v)$ represent the same class in cohomology, we can deduce the estimate in (\ref{eq: beta_G,H is bilipschitz}) from the following inequalities: 
	\begin{align*}
		\| u + \delta v'\|_\infty & \leq \|u\|_\infty + \|\delta^n\| \cdot \|v'\|_\infty \\
		& \leq \|u\|_\infty + (n+1)\cdot \|v'\|_\infty \\
		& \leq \|u\|_\infty + (n+1)\cdot \|v\|_\infty \\
		& \leq (n+2)\cdot \max\left\{\|u\|_\infty, \|v\|_\infty \right\} \\
		& = (n+2)\cdot \|(u,v)\|_\infty.
	\end{align*}
	Since $G$ is amenable and the action $G\acts K$ is by automorphisms which are simplicially homotopic to the identity, we know that $\iota^\bullet_G$ admits a $(C_n)$-chain homotopy inverse (Proposition \ref{proposition: amenable - invariant resolutions}).
	The same holds for the action $H\acts L$, so that $\iota^\bullet_H$ admits a $(C_n)$-chain homotopy inverse.
	It follows from Lemma \ref{lem: les and five lemma for mapping cones} that $\varphi^\bullet$ induces a bi-Lipschitz isomorphism in cohomology such that
	\[
	(2 C_n)^{-1}\cdot \|\alpha \| \leq \|H^n(\varphi^\bullet)(\alpha)\| \leq \|\alpha \|,
	\]
	for every $\alpha \in H^n(\Cone(r_{G,H})^\bullet)$.
	In conclusion, since $\beta^\bullet$, $\beta^\bullet_{G,H}$ and $\varphi^\bullet$ induce bi-Lipschitz isomorphisms in cohomology, then also $\iota^\bullet_{G,H}$ does and the estimate over the norms in the statement easily follows.
\end{proof}

\begin{rem}
	\label{rem: invariant chains in Kuessner}
	Let $(K,L)$ be a pair of multicomplexes and let $G$ be a group acting on $K$. Assume that the action $G\acts K$ does \emph{not} preserve the submulticomplex $L$. 
	The set $GL$ of $G$-translates of $L$, is a submulticomplex of $K$ and there is a well-defined action $G\acts (K,GL)$.
	In a previous version of this paper and in Kuessner's work on multicomplexes (see, for example, \cite[Corollary 2]{Kue15}) the following statement is claimed to be true. 
	If $G$ is amenable and it acts on $K$ by automorphisms which are simplicially homotopic to the identity, then the relative bounded cohomology $H^n_b(K,L)$ can be isometrically computed via the resolution $C^\bullet_b(K,GL)^G$ of $G$-invariant cochains on $K$ relative to $GL$.
	We believe that this approach underestimates some difficulties, which could arise in geometric contexts, if we do not require some regularity of the action $G\acts K$ on $L$. We refer the reader to Remark \ref{rem: relative construction and regularity of action} for a geometric motivation in our context and to \cite{Cap} for a more precise argument in the context of amenable open covers.
\end{rem}

\section{Proof of Theorem \ref{thm: superadditivity}}
\label{sec: proof of superadditivity}

Let $n \in \N_{\geq 2}$. Let $M_1,\dots, M_k$ be oriented compact connected triangulated $n$-manifolds with non-empty boundary.
Let \[(S_1^+,S_1^-),\dots,(S_h^+,S_h^-)\] be a pairing of some oriented compact connected pairwise-disjoint $(n-1)$-submanifolds of $\partial M_1\sqcup\cdots \sqcup \partial M_k$.
Let $f_i\colon S_i^+\rightarrow S_i^-$ be an orientation-reversing homeomorphism, $i\in\{1,\dots,h\}$. 
Let $M$ be the manifold obtained by gluing $M_1,\dots, M_k$ along $f_1,\dots, f_h$ and assume that $M$ is connected.
For every $j \in \{1,\dots, k\}$, let $q_j\colon M_j\rightarrow M$ denote the quotient map (which is an embedding, provided that $M_j$ does not admit self-gluings).
We still denote by $M_j$ the image of $M_j$ in $M$ via $q_j$.
We denote by $\S$ the union of the manifolds $S_i^\pm$, $i\in \{1,\dots, h\}$, understood as a subset of $M$.
Let \[B_j = q_j^{-1}(\partial M) = \partial M \cap M_j,\] understood a subset of $M_j$.
In this section, we first deduce Theorem \ref{thm: superadditivity} from the following more general result, which is proved later in this section.
The constant $C_n$ in the statement is described in Definition \ref{defn: constants C_n}.
\begin{prop}
	\label{prop: superadditivity}
	If the following conditions hold for every $i\in \{1,\dots,h\}$ and $j\in \{1,\dots, k\}$:
	\begin{enumerate}
		\item[$\mathrm{(1)}$] The kernel of $\pi_1(\partial M \hookrightarrow M,x)$ is amenable for every $x \in \partial M$;
		\item[$\mathrm{(2)}$] The kernel of $\pi_1(B_j \hookrightarrow M_j,x)$ is amenable for every $x \in B_j$;
		\item[$\mathrm{(3)}$] $M_j$ is $\pi_1$-injective in $M$;
		\item[$\mathrm{(4)}$] $\S_i^\pm$ is $\pi_1$-injective in $M_j$ and has an amenable fundamental group;
		\item[$\mathrm{(5)}$] $\partial S_i^{\pm}$ is path-connected and the map $\pi_1(\partial S_i^{\pm} \hookrightarrow S_i^{\pm},x)$ is an isomorphism for one (whence every) $x \in \partial S_i^{\pm}$;
	\end{enumerate}
	then
	\begin{equation}
		\label{eq: inequality superadditivity}
		(2 C_n)^4 (n+2)^3\cdot \|M,\partial M\| \geq \|M_1,\partial M_1\| + \dots + \|M_k,\partial M_k\|.
	\end{equation}
	Moreover, if $\partial M$ and $B_j$ are $\pi_1$-injective in $M$, then 
	\begin{equation}
		\label{eq: inequality superadditivity, improved}
		(2 C_n)^3 (n+2)^3\cdot \|M,\partial M\| \geq \|M_1,\partial M_1\| + \dots + \|M_k,\partial M_k\|.
	\end{equation}
\end{prop}
\begin{rem}
	\label{rem: wlog embeddings}
	We can always reduce to the case when all the maps $q_j\colon M_j\rightarrow M$ are embeddings. In fact, problems arise only in presence of self-gluings. In this case, instead of considering the self-gluing $f_i\colon S_i^+ \rightarrow S_i^-$ of $M_j$, we add another piece \[N_i = S_i^+\times [0,1],\] which is glued to $M_j$ via the maps
	\begin{align*}
		S^+_i\times \{0\} \rightarrow S^+_i, & \quad (x,0)\mapsto x, \\
		S_i^+\times \{1\} \rightarrow S_i^-, & \quad (x,1)\mapsto f_i(x).
	\end{align*}
	The manifold $N$ obtained in this way is homeomorphic to $M$. Hence, if inequalities (\ref{eq: inequality superadditivity}) or (\ref{eq: inequality superadditivity, improved}) hold for $N$, then they also hold for $M$. 
	Notice that $N$ is constructed by gluing embedded pieces, which satisfy the assumptions of Proposition \ref{prop: superadditivity}.
	\textbf{As a consequence, throughout the sequel, we always assume that $M_1,\dots, M_k$ are embedded in $M$.}
\end{rem}

\begin{proof}[Proof of Theorem \ref{thm: superadditivity}]
We assume that the manifolds $M_1,\dots, M_k$ have $\pi_1$-injective boundary.
We need to check that conditions (1) - (5) of Proposition \ref{prop: superadditivity} are satisfied under the assumptions of Theorem \ref{thm: superadditivity}, which we recall for the convenience of the reader:
\begin{itemize}
	\item[$(i)$] $S_i^\pm$ has an amenable fundamental group;
	\item[$(ii)$] $\partial S_i^\pm$ is path-connected, $\pi_1$-injective in the corresponding $\partial M_j$ and the map $\pi_1(\partial S_i^{\pm} \hookrightarrow S_i^{\pm},x)$ is an isomorphism for every $x \in \partial S_i^{\pm}$.
\end{itemize}

First, we observe that $S_i^\pm$ and $B_j$ are $\pi_1$-injective in $\partial M_j$. In fact, by Seifert - Van Kampen theorem, $\pi_1(\partial M_j)$ can be described as the fundamental group of a graph of groups with vertex groups $\pi_1(S_i^\pm)$, $\pi_1(B_j)$ and edge groups $\pi_1(\partial S_i^\pm)$. Under our assumptions, the edge morphisms are injective.  Consequently, we know that the vertex groups embeds into $\pi_1(\partial M_j)$ \cite{Ser02}.
Moreover, since \[\partial B_j = \partial (\S \cap M_j) = B_j \cap (\S \cap M_j)\] has the same number of connected components of $\S \cap M_j$, it follows from the Mayer - Vietoris exact sequence that $B_j$ has the same number of connected components as $\partial M_j$.

In order to get (1), we prove that $\partial M$ is $\pi_1$-injective in $M$. 
To this end, we observe that $\pi_1(M)$ can be described as the fundamental group of a graph of groups with vertex groups $\pi_1(M_j)$ and edge groups $\pi_1(S_i^\pm)$, while $\pi_1(\partial M)$ can be described as the fundamental group of a graph of groups with vertex groups $\pi_1(B_j)$ and edge groups $\pi_1(\partial S_i^\pm)$. Since the underlying graphs are isomorphic (here we use that $B_j$ has the same number of connected components of $\partial M_j$) and the inclusion $\partial M\hookrightarrow M$ induces injective morphisms at the level of vertex and edge groups, it follows that $\partial M$ is $\pi_1$-injective in $M$.
In order to get (2), we just recall that $B_j$ is $\pi_1$-injective in $\partial M_j$, which is in turn $\pi_1$-injective in $M_j$.
In order to get (3), we observe that $\pi_1(M)$ is the fundamental group of a graph of groups with injective edge morphisms, therefore the vertex groups $\pi_1(M_j)$ embed into $\pi_1(M)$.
Conditions (4) and (5) clearly follow from $(i)$ and $(ii)$.
In conclusion, since in our case $\partial M$ and $B_j$ are $\pi_1$-injective in $M$, then inequality (\ref{eq: inequality superadditivity, improved}) applies.
\end{proof}

\begin{rem}
	\label{rem: relative construction and regularity of action}
	Some difficulties could arise when gluings are performed along proper submanifolds of the boundary.
	Our general strategy to get the estimate in Proposition \ref{prop: superadditivity} can be summarized as follows.
	
	In order to obtain a fundamental cocycle of $M$ with controlled norm, we average some fundamental cocycles of the manifolds $M_1,\dots, M_k$ on the gluing locus $\S$.
	To this end, we need to consider the action of the group $G=\Pi_M(\S)$ on the multicomplex $\A(M)$, thus passing to $G$-invariant cochains.
	This averaging procedure is possible when the group $G$ is amenable, e.g. when the connected components of $\S$ have amenable fundamental group (Lemma \ref{lem: Pi(A,A) is amenable}).
	If $\partial \S$ is non-empty, the action $G\acts \A(M)$ does \emph{not} preserve the submulticomplex $\A(\partial M)$. 
	Since the boundary $\partial M$ is also obtained by gluing the remaining boundaries of $M_1,\dots, M_k$ along $\partial \S$ (unless we are gluing entire boundary components), the averaging fails to be relative unless we assume some regularity of the action of $G$ on $\partial \S$.
	Therefore, in order to consider invariant chains in the relative setting, our solution is to show that the action $G\acts \A(M)$ has orbits in $\A(\partial M)$ induced by $H=\Pi_M(\partial \S)$ (Lemma \ref{lem: G acts K has orbits on L induced by H}).
	This translates in the assumption (5) of Proposition \ref{prop: superadditivity}.
	
	When we are gluing along entire boundary components, passing to invariant chains is more straightforward. This case is presented in Section \ref{sec: additivity} to prove Theorem \ref{thm: additivity full boundary components}.
\end{rem}

\subsection{Proof of Proposition \ref{prop: superadditivity}}
\label{sec: superadditivity}
We define nested pairs of multicomplexes from the gluing data.
For every $j \in \{1,\dots, k\}$, we consider the following pairs of multicomplexes
\[
(K,K')=(\A(M),\A_M(\partial M)),
\]
\[
(L_j,L_j')=(\A(M_j), \A_{M_j}(B_j)),
\]
where $\A_M(\partial M)$ (resp. $\A_{M_j}(B_j)$) denotes the image of $\A(\partial M)$ (resp. $\A(B_j)$) inside $\A(M)$ (resp. $\A(M_j)$).
Since $M_j$ is $\pi_1$-injective in $M$, we have that $\A(M_j)$ is naturally isomorphic to $\A_M(M_j)$ and $\A_{M_j}(B_j)$ is naturally isomorphic to $\A_M(B_j)$ (Proposition \ref{prop: j_A is a simplicial embedding}).
It follows that there is a natural inclusion of pairs
\[
(L_j,L_j')\subseteq (K,K').
\]
Let $A$ denote the multicomplex $\A(\S)$. Since $\S$ is $\pi_1$-injective in $M$, we have that $A$ naturally sits in $K$ as a submulticomplex (Proposition \ref{prop: j_A is a simplicial embedding}). 
We set \[G=\Pi_M(\S).\]
We know that $G$ acts on $K=\A(M)$ as described in Subsection \ref{subsec: the action of Pi(X,X) on A(X)}.
Since both $M_j$ and $\S$ are $\pi_1$-injective in $M$, it follows that $A\cap L_j$ can be identified with the aspherical multicomplex $\A(M_j\cap \S)$ associated to the collection of path-connected components of $\S$ contained in $M_j$.
It follows that $G$ defines an action on the pair $(K,L_j)$, where $G$ acts on $L_j$ as $\Pi_{M_j}(M_j\cap \S)$ does. By Proposition \ref{prop: elements of Pi(A,A) are homotopic to the indentity}, the action $G\acts (K,L_j)$ is by automorphisms that are simplicially homotopic (as maps of pairs) to the identity.
The actions $G\acts K$ and $G\acts L_j$ do \emph{not} preserve the submulticomplexes $K'$ and $L_j'$ respectively.
However, by setting \[H =\Pi_M(\partial S),\] which is a subgroup of $G$ in our assumptions, we have that these actions have orbit induced by $H$ on those multicomplexes (see Definition \ref{defn: having orbits in L induced by H}). In order to show this, the assumption (5) in Proposition \ref{prop: superadditivity} is crucial.
\begin{lemma}
	\label{lem: G acts K has orbits on L induced by H}
	The action $G\acts K$ (resp. $G\acts L_j$) has orbits in $K'$ (resp. $L_j'$) induced by $H$.
\end{lemma}
\begin{proof}
	We prove the statement for the action $G\acts K$, since the argument is the same for the other action.
	Let $\sigma=(\Delta,(x_0,\dots, x_n))$ be an algebraic $n$-simplex of $K'=\A(\partial M)$ and let $g \in G$ be such that $g\cdot \sigma$  is an algebraic $n$-simplex of $\A(\partial M)$. 
	We need to show that there exists an element $h \in H$ such that $h\cdot \sigma = g\cdot \sigma$.
	Let $\{v_0,\dots, v_k\}$ be the vertices of $\Delta$, with $k\leq n$.
	Since $v_j$ and $g\cdot v_j$ are both vertices of $\A(\partial M)$ (hence points of $\partial M$), then, for every $j \in \{0,\dots,k\}$, there is a representative $\{\gamma_x\}_{x \in X}$ of $g$ such that $\gamma_{v_j}\colon [0,1]\rightarrow \S$ has both endpoints in $\partial M\cap \S =\partial S$.
	By assumption (5), we can invoke Lemma \ref{lem: rephrase condition (5)} below to deduce that there exist paths
	\[
	\lambda_{v_j}\colon [0,1] \rightarrow \partial \S, \qquad j \in \{0,\dots, k\},
	\]
	such that $\lambda_{v_j}$ is homotopic to $\gamma_{v_j}$ in $M$ relative to the endpoints.
	Therefore $h=\{\lambda_0, \dots, \lambda_k\}$ defines an element of $H$. 
	By construction we have that $h\cdot e = g\cdot e$, for every edge $e$ of $\Delta$, hence $g\cdot \Delta^1=h\cdot \Delta^1$, where $\Delta^1$ denotes the 1-skeleton of $\Delta$. Since $K=\A(M)$ is an aspherical multicomplex, this is indeed sufficient to conclude \cite[Proposition 3.30]{FM23}.
\end{proof}
\begin{lemma}
	\label{lem: rephrase condition (5)}
	For every path $\gamma\colon [0,1]\rightarrow \S$ with endpoints in $\partial \S$, there exists a path $\lambda \colon [0,1]\rightarrow \partial \S$ such that $\lambda$ is homotopic in $M$ to $\gamma$ relative to the endpoints.
\end{lemma}
\begin{proof}
	The path-connected components of $\S$ can be identified with $S_i^{\pm}$, for $i \in \{1,\dots, k\}$.
	Since $\partial S_i^{\pm}$ is path-connected, there exists a path $\varepsilon \colon [0,1]\rightarrow\partial \S$ from $\gamma(0)$ to $\gamma(1)$. The homotopy class of the loop $\gamma * \bar{\varepsilon}$ identifies an element of $\pi_1(S_i^{\pm},\gamma(0))$. Since $\pi_1(\partial S_i^{\pm} \hookrightarrow S_i^{\pm},\gamma(0))$ is an isomorphism, there exists a loop $\delta \colon [0,1] \rightarrow \partial S_i^{\pm}$ which is homotopic in $S_i^{\pm}$ (hence in $M$) to $\gamma * \bar{\varepsilon}$ relative to the endpoints. It follows that $\lambda \coloneqq \delta * \varepsilon$ satisfies the statement.
\end{proof}
Consider now the diagram in Figure \ref{eq: diagram superadditivity}, where the following observations need to be made.
By conditions (1) and (2) of Proposition \ref{prop: superadditivity}, we can invoke Theorem \ref{thm: relative mapping theorem, biLipschitz version} to get canonical bi-Lipschitz isomorphisms $\Psi^n$ and $\Psi^n_j$.
It is straightforward to check that the middle square is commutative.
The maps $I^n$ and $I_j^n$ are induced by the inclusions of complexes
\[
C^\bullet_b(K,K')^{G,H}_\alt\hookrightarrow C^\bullet_b(K,K'),\qquad C^\bullet_b(L_j,L_j')^{G,H}_\alt\hookrightarrow C^\bullet_b(L_j,L_j').
\]
Since $G$ (hence $H$) is amenable and since the actions $G\acts (K,L_j)$, $H\acts K'$ and $H\acts L_j'$ are by automorphisms that are simplicially homotopic to the identity (Proposition \ref{prop: elements of Pi(A,A) are homotopic to the indentity}), it follows from Lemma \ref{lemma: invariant cochains} that $I^n$ and $I_j^n$ are bi-Lipschitz isomorphisms, for every $j \in \{1,\dots, k\}$.
The lower square is of course commutative.\\
\begin{figure}
	% https://tikzcd.yichuanshen.de/#N4Igdg9gJgpgziAXAbVABwnAlgFyxMJZARgBoAGAXVJADcBDAGwFcYkQAdDgIywHMIaFnC6MsAW1xwA+sABWAXmIBfAHoBrAAQAJVWGncAFAFlpc0lzT0ATniabTcgJQhlpdJlz5CKMsWp0TKzsXLwCQswiHGKSODLySmpauvpGjqQAQmYubh7YeAREZABMAQwsbIicPPyCwqISUrKKKho6egaGADJmpD1yAOQ57iAY+d5E5KT+NOXBVSmdxhYcVrZY9sbDeV6FKFOls0GVIItGANKk50OuI2O7PiSkAMxlxyE14fXRjXHNiW0UoYAMKqULMRiMGA4Tr9UgAcX6Q1U8OkXCYOG2o08BUeU1eRwqHzCdUiDVi8RaSXaYBBYJ4EKhMIuCOuThRaI4GJyARgUD48CIoAAZtYIOIkM8aDgIEgyIEiVUuKS4DTOlhsiAaIx6NwYIwAAo4iZVaz8AAWOFuIrFEsQUxAMrlhPm1RVaqM6MYaHN9E12t1+qN4z2IDNfEt1pAovFSGK0tliHlcxOyoiqoAkmYozG7QAWBOSl0nDM521IACshcQ8YVrrTwk0aGzuWj5cQADZqwW6yc0GXY4ge06a63c0gu47E1Xex93YwW5RlEA
	\begin{tikzcd}
		{H^n_b(M,\partial M)} \arrow[r, "\oplus H^n_b(q_j)"] \arrow[d, equal] & {\bigoplus\limits_{j=1}^k H^n_b(M_j,\partial M_j)} \arrow[d, "\oplus H^n_b(\alpha_j)"'] \\
		{H^n_b(M,\partial M)} \arrow[r, "\oplus H^n_b(q_j)"]                                            & {\bigoplus\limits_{j=1}^k H^n_b(M_j,B_j)}                                               \\
		{H^n_b(K,K')} \arrow[u, "\Psi^n"] \arrow[r]                                                           & {\bigoplus\limits_{j=1}^k H^n_b(L_j,L_j')} \arrow[u, "\oplus \Psi_j^n"]                      \\
		{H^n(C^\bullet_b(K,K')^{G,H}_\alt)} \arrow[u, "I^n"] \arrow[r, "\oplus l_j"] & {\bigoplus\limits_{j=1}^k H^n(C^\bullet_b(L_j,L_j')^{G,H}_\alt),} \arrow[u, "\oplus I_j^n"]  
	\end{tikzcd}
	\caption{Proof of Theorem \ref{thm: superadditivity}. Horizontal maps in the diagram are induced by inclusions. The maps $\alpha_j\colon (M_j,B_j)\rightarrow (M_j,\partial M_j)$ denote the inclusions of pairs.}
	\label{eq: diagram superadditivity}
\end{figure}

Our proof of Proposition \ref{prop: superadditivity} is based on the following extension property for bounded invariant coclasses.
\begin{prop}
	\label{prop: extension of bounded coclasses}
	Let $\varepsilon>0$ and take for every $j\in \{1,\dots, k\}$ an element
	\[
	\phi_j \in H^n(C^\bullet_b(L_j,L_j')^{G,H}_\alt).
	\]
	Then there exists a coclass $\phi \in H^n(C^\bullet_b(K,K')^{G,H}_\alt)$ such that $l_j(\phi)=\phi_j$ for every $j\in \{1,\dots, k\}$ and
	\[
	\|\phi\|_\infty \leq \max \bigl\{\|\phi_j\|\colon \; j\in \{1,\dots, k\}\bigr\} + \varepsilon.
	\]
\end{prop}
We first show how Proposition \ref{prop: extension of bounded coclasses} can be used to conclude the proof of Proposition \ref{prop: superadditivity}. By the duality between bounded cohomology and simplicial volume (Proposition \ref{prop: duality principle}), for every $j\in \{1,\dots, k\}$ we can take an element $\gamma_j \in H^n_b(M_j,\partial M_j)$ such that
\begin{align*}
	\|M_j,\partial M_j\| = \langle \gamma_j, [M_j,\partial M_j] \rangle, \qquad \|\gamma_j\|_\infty \leq 1.
\end{align*}
Let $\gamma_j' \in H^n(C^\bullet_b(L_j,L_j')^{G,H}_\alt)$ be such that \[\Psi^n_j\circ I_j^n(\gamma_j')= H^n_b(\alpha_j)(\gamma_j).\] Since $H^n_b(\alpha_j)$ is norm non-increasing and $\Psi_j^n$ and $I_j^n$ are bi-Lipschitz isomorphisms, we deduce from the estimates in Lemma \ref{lemma: invariant cochains} and Theorem \ref{thm: relative mapping theorem, biLipschitz version} that
\[
\|\gamma_j'\|_\infty \leq (2C_n)^3(n+2)^2 \cdot \|\gamma_j\|_\infty \leq (2C_n)^3(n+2)^2.
\]
Let $\varepsilon >0$.
By Proposition \ref{prop: extension of bounded coclasses}, there exists $\gamma' \in H^n(C^\bullet_b(K,K')^{G,H}_\alt)$ such that 
\[
\|\gamma'\|_\infty \leq (2C_n)^3(n+2)^2 + \varepsilon, \qquad l_j(\gamma')=\gamma_j',\qquad j\in \{1,\dots, k\}.
\]
Let $\gamma = \Psi^n\circ I^n (\gamma') \in H^n_b(M,\partial M)$. 
Since $\Psi^n$ and $I^n$ are both bi-Lipschitz isomorphisms, again from the estimates in Lemma \ref{lemma: invariant cochains} and Theorem \ref{thm: relative mapping theorem, biLipschitz version} we deduce that
\[
\|\gamma \|_\infty \leq 2C_n(n+2) \cdot \|\gamma'\|_\infty \leq 2C_n(n+2)\cdot ((2C_n)^3(n+2)^2 + \varepsilon).
\]
We proceed with the proof of Proposition \ref{prop: superadditivity} by constructing a fundamental cycle of $M$ that is well behaved with respect to the gluings. In the following, with an abuse of notation, we identify any chain in $M_j$ and $\S$ with the corresponding chain in $M$ via the corresponding inclusions.
\begin{lemma}
	\label{lem: fundamental class of M}
	There exists a real relative fundamental cycle $z \in C_n(M)$ of $M$ such that $z = z_1 + \dots + z_k$, where $z_j \in C_n(M_j)\subseteq C_n(M)$ is a real relative fundamental cycle of $M_j$, for every $j\in \{1,\dots, k\}$.
\end{lemma}
\begin{proof}
	Let $w_j \in C_n(M_j)$ be a real chain representing a real fundamental class of $M_j$ and let $w = w_1 + \dots + w_k \in C_n(M)$.
	We need to modify $w$ in order to obtain the desired fundamental cycle for $M$.
	
	For every $i \in \{1,\dots, h\}$, let $\sigma_i^+\in C_{n-1}(S_i^+)$ be a real relative fundamental cycle of $S_i^+$. 
	We denote by \[\sigma_i^-=-(f_i)_*(\sigma_i^+)\] the fundamental cycle of $S_i^-$ induced by the gluing map $f_i$.
	Let $\beta_j \in C_{n-1}(B_j)$ be a real relative fundamental cycle of $B_j$, $j \in\{1,\dots, k\}$. 
	We denote by $\sigma_j$ the sum of the $\sigma_i^\pm$ such that $S_i^\pm\subseteq M_j$.
	Since $\partial \sigma_j$ and $\partial \beta_j$ represent a fundamental class of $\partial B_j$ with opposite orientation, it follows that there exists $d_j \in C_{n-1}(\partial B_j)\subseteq C_{n-1}(B_j)$ such that $\partial d_j = \partial \sigma_j + \partial \beta_j$.
	Therefore $b_j=d_j- \beta_j$ is a real relative fundamental cycle of $B_j$. 
	We set $c_j = \sigma_j - b_j \in C_{n-1}(\partial M_j)$.
	Since both $c_j$ and $\partial w_j$ represent a fundamental class of $\partial M_j$, there exists $u_j \in C_n(\partial M_j)$ such that $c_j - \partial w_j = \partial u_j$. 
	Therefore \[z_j=w_j + u_j \in C_n(M_j)\] is a real relative fundamental cycle of $(M_j,\partial M_j)$ such that $\partial z_j = \sigma_j - b_j$.
	In conclusion, since $f$ is orientation reversing, we have that $z=z_1 + \dots + z_k \in C_n(M)$ is a real relative fundamental cycle of $(M,\partial M)$.
\end{proof}
Finally, by Lemma \ref{lem: fundamental class of M} and the commutativity of (\ref{eq: diagram superadditivity}), we deduce that
\begin{align*}
	\langle \gamma, [M,\partial M] \rangle = c(z) & = \sum_{j=1}^k c(z_i) = \sum_{j=1}^k \langle \gamma, [z_j] \rangle  \\	
	& = \sum_{j=1}^k \langle H^n_b(q_j)(\gamma), [z_j] \rangle = \sum_{j=1}^k \langle H^n_b(\alpha_j)(\gamma_j), [z_j] \rangle  \\
	& = \sum_{j=1}^k \langle \gamma_j, [z_j] \rangle = \sum_{j=1}^k \langle \gamma_j, [M_j,\partial M_j] \rangle,  
\end{align*}
where, with an abuse of notation, $[z_j]$ represents the class represented by the same chain $z_j \in C_n(M_j)\subseteq C_n(M)$, either in $H_n(M,\partial M)$, $H_n(M_j,B_j)$ or $H_n(M_j,\partial M_j)$.
In conclusion, we get
\begin{align*}
	\sum_{j=1}^k\|M_j,\partial M_j\| & = \sum_{j=1}^k \langle \gamma_j, [M_j,\partial M_j] \rangle \\
	& = \langle \gamma, [M,\partial M]\rangle \\
	& \leq \|\gamma\|_\infty \cdot \|M,\partial M\| \\
	& \leq 2C_n(n+2)\cdot ((2C_n)^3(n+2)^2 + \varepsilon) \cdot  \|M,\partial M\|,
\end{align*}
and, since $\epsilon$ is arbitrary, we obtain inequality (\ref{eq: inequality superadditivity}).

\begin{rem}
	\label{rem: proof of superadditivity, improved}
	If $\partial M$ is $\pi_1$-injective in $M$ and $B_j$ is $\pi_1$-injective in $M_j$, we can improve the bi-Lipschitz constants of $\Psi^n$ and $\Psi_j^n$ according to Remark \ref{rem: not optimality of constant}. It follows that we can take 
	\[
	\|\gamma \|_\infty \leq 2C_n(n+2) \cdot \|\gamma'\|_\infty \leq 2C_n(n+2)\cdot ((2C_n)^2(n+2)^2 + \varepsilon)
	\]
	in the construction above, from which we deduce inequality (\ref{eq: inequality superadditivity, improved}).
\end{rem}

In order to conclude the proof of Proposition \ref{prop: superadditivity} we are left to prove Proposition \ref{prop: extension of bounded coclasses}.
Consider the submulticomplex \[\cup L_j = L_1 \cup \dots \cup L_k\] of $K$. Since every $L_j=\A(M_j)$ is connected and $K=\A(M)$ is connected by assumption, it follows that also $\cup L_j$ is connected. 
Let $p\colon \tilde{K}\rightarrow K$ denote the universal covering (simplicial) map (see Section \ref{subsection: universal covering of multicomplexes}).
Let $\hat{L}_j= p^{-1}(L_j)$ and $\hat{A}= p^{-1}(A)$.
The intersection $\hat{L}_i \cap \hat{L}_j$ consists of a collection of connected components of $\hat{A}$, for every $i\neq j$.
Consider the multicomplex 
\[
\cup \hat{L}_j = p^{-1}(\cup L_j)= \hat{L}_1 \cup \cdots \cup \hat{L}_k.
\]

Let $p_M\colon \tilde{M} \rightarrow M$ denote the universal covering map of the manifold $M$.
We describe the structure of $\tilde{M}$ as a \emph{tree of spaces}, following the construction in \cite[Chapter 9]{Fri17}.
We define a tree $T'$ as follows. 
We pick a vertex for every connected component of $p_M^{-1}(M_j)$, $j\in \{1,\dots, k\}$, and we join two vertices if the corresponding connected components intersect (along a connected component of $p_M^{-1}(\S)$).
Since $T'$ is a deformation retract of $\tilde{M}$, then it is connected and simply connected, i.e. a tree.

Similarly, we can describe the structure of $\cup \hat{L}_j$ as a \emph{tree of multicomplexes}.
We construct a tree $T$ as follows. We pick a vertex for every connected component of $\hat{L}_j$, $j\in \{1,\dots, k\}$, and we join two vertices if the corresponding components intersect (along a connected component of $\hat{A}$). Therefore, edges of $T$ correspond to the connected components of $\hat{A}$. 
Since the connected components of $\hat{L}_j$ and $\hat{A}$ bijectively correspond to the connected components of $p_M^{-1}(M_j)$ and $p_M^{-1}(\S)$, respectively, it follows that $T$ and $T'$ are simplicially isomorphic. Hence also $T$ is a tree.
We denote by $V(T)$ the set of vertices of $T$ and, for every $v \in V(T)$, we denote by $\tilde{L}_v$ the connected component of $\hat{L}_j$ corresponding to $v$.

\begin{lemma}
	\label{lem: universal coverings}
	For every $v \in V(T)$ there exists $j(v)\in \{1,\dots, k\}$ and a universal covering (simplicial) map $p_v\colon \tilde{L}_v\rightarrow L_{j(v)}$ such that the following diagram commutes
	\[
	% https://tikzcd.yichuanshen.de/#N4Igdg9gJgpgziAXAbVABwnAlgFyxMJZABgBpiBdUkANwEMAbAVxiRAB128HZgAZAL4B9GiAGl0mXPkIoAjOSq1GLNnyHAAVgAoaASgFiJIDNjwEiCuUvrNWiEAGkxSmFADm8IqABmAJwgAWyQFEBwIJAAmahw6LAY2AAsICABrI18A4MQyMIjEUNtVB040LBEMkH8gpFzwqOoi+w52Mo0AH05uXkERQ2oGOgAjGAYABSlzWRA-LHdEnBcBIA
	\begin{tikzcd}
		\tilde{L}_v \arrow[r, "p_v"] \arrow[rd, "p|_{\tilde{L}_v}"'] & L_{j(v)} \arrow[d, hook] \\
		& K.                       
	\end{tikzcd}
	\]
	Moreover, for every path connected component $C$ of $\hat{A}$, there exists a universal covering (simplicial) map $p_C\colon C \rightarrow p(C)$ such that the following diagram commutes
	\[
	% https://tikzcd.yichuanshen.de/#N4Igdg9gJgpgziAXAbVABwnAlgFyxMJZABgBpiBdUkANwEMAbAVxiRAB128HZgBBAL4gBpdJlz5CKAIzkqtRizZ9hokBmx4CRWdPn1mrRCADSw+TCgBzeEVAAzAE4QAtklkgcEJACZqOOiwGNgALCAgAa1UHZzdEMk9vRD8FQzZONCwAfWAAH05uXkEhagY6ACMYBgAFcS0pEEcsKxCcaJAnVyQEr3dqAyVjDOzgAqCigSEBCgEgA
	\begin{tikzcd}
		C \arrow[rd, "p|_{C}"'] \arrow[r, "p_C"] & p(C) \arrow[d, hook] \\
		& K.                
	\end{tikzcd}
	\]
\end{lemma}
\begin{proof}
	The aspherical multicomplexes $K = \A(M)$, $L_j=\A(M_j)$ and $A=\A(\S)$ are classifying spaces for the fundamental groups of $M$, $M_j$ and $\S$ respectively \cite[Proposition 3.33]{FM23}.
	By assumption we know that $M_j$ and $\S$ are $\pi_1$-injective in $M$. It follows that also $L_j$ and $p(C)$ are $\pi_1$-injective in $K$.
	The statement now follows from standard covering theory.
\end{proof}
Observe that the 0-skeleton of $\cup L_j$ coincides with the 0-skeleton of $K=\A(M)$, hence, with the points of $M$.
\begin{lemma}
	\label{lem: pi_1(cup L_j in K) is an isomorphism}
	For every $x_0 \in K^0$, $\pi_1(\cup L_j \hookrightarrow K, x_0)$ is an isomorphism.
\end{lemma}
\begin{proof}
	Let $x_0$ be a vertex of $K$.
	Let $g\colon \pi_1(\cup L_j, x_0) \rightarrow \pi_1(M, x_0)$ denote the composition of $\pi_1(\cup L_j \hookrightarrow K, x_0)$ with the map $\pi_1(|K|, x_0)\rightarrow\pi_1(M, x_0)$, induced by the projection $S_M\colon \K(M)\rightarrow M$. Since the latter is an isomorphism \cite[Proposition 3.33]{FM23}, in order to conclude it is enough to show that $g$ is an isomorphism.
	
	We argue via Bass-Serre theory.
	The group $\pi_1(M,x_0)$ acts simplicially on $T'$ with vertex stabilizers of the form $\pi_1(M_j, x)$, $j \in \{1,\dots, h\}$, and edge stabilizers of the form $\pi_1(\S,x)$, for some $x \in \S$.
	Similarly, we have that $\pi_1(\cup L_j, x_0)$ acts simplicially on $T$ with vertex stabilizers of the form $\pi_1(L_j, x)$, and edge stabilizers of the form $\pi_1(A,x)$, for some $x \in A^0=\S$. 
	We have that $g$ and the simplicial isomorphism between $T$ and $T'$ induce the same identification of the vertex and edge stabilizers of the actions $\pi_1(\cup L_j, x_0)\acts T$ and $\pi_1(M, x_0)\acts T'\cong T$.
	In this way, by Bass-Serre theory \cite{Ser02}, we get that $g$ is an isomorphism.
\end{proof}

By standard covering theory, it follows from Lemma \ref{lem: pi_1(cup L_j in K) is an isomorphism} that $\cup \hat{L}_j = p^{-1}(\cup L_j)$ is path-connected and that the map \[p|_{\cup \hat{L}_j}\colon \cup \hat{L}_j\rightarrow \cup L_j\] is a universal covering map.
Let $x_0$ be a vertex of $K$, hence, a point of $M$.
We denote by $\Gamma$ the fundamental group of $M$ based at $x_0$. By Lemma \ref{lem: pi_1(cup L_j in K) is an isomorphism} and \cite[Proposition 3.33]{FM23}, there are canonical isomorphisms
\[
\Gamma = \pi_1(M,x_0) \cong \pi_1(K,x_0) \cong \pi_1(\cup L_j,x_0).
\]
Therefore we can consider the simplicial action $\Gamma \acts \cup\hat{L}_j$ by deck transformations. Let $\Gamma_v$ denote the stabilizer of $\tilde{L}_v$ in $\Gamma$, and observe that $\Gamma_v$ is canonically isomorphic to $\pi_1(L_{j(v)}, x_v)$, for some $x_v \in L_{j(v)}^0=M_{j(v)}$.
Let $G = \Pi_M(\S)$ and let $\psi \colon G \rightarrow \Aut(K)$ denote the representation induced by the simplicial action $G\acts K$.
Consider the group
\[
\Lambda(G) = \bigl\{(f,g)\in \Aut(\tilde{K})\times G\;| \; p \circ f = \psi(g)\circ p \bigr\}.
\]
There is a canonical injection
\[
\Gamma \rightarrow \Lambda(G), \qquad \gamma \mapsto (\gamma, e),
\]
where $e$ denotes the neutral element of $G$, and a canonical projection
\[
\Lambda(G) \rightarrow G, \qquad (f,g)\mapsto g,
\]
such that the sequence $1 \rightarrow \Gamma \rightarrow \Lambda(G) \rightarrow G \rightarrow 1$ is exact.
Intuitively speaking, the group $\Lambda(G)$ encodes all the possible lifts to the universal covering of the endomorphisms induced by $G$.
By projecting on the first factor, $\Lambda(G)$ induces a simplicial action $\Lambda(G) \acts \tilde{K}$.
Let $\Lambda_v(G)\leq \Lambda(G)$ denote the stabilizer of $\tilde{L}_v$ in $\Lambda(G)$, $v \in V(T)$. 
After identifying $\Gamma$ with a subgroup of $\Lambda(G)$, we have that $\Gamma \cap \Lambda_v(G) = \Gamma_v$.
As an application of the unique lifting property for covering spaces, since the action $G \acts K$ preserves $L_j$ for every $j\in\{1,\dots, k\}$, the projection $\Lambda_v(G)\rightarrow G$ is surjective i.e. also the second row of the following commutative diagram
\[
% https://tikzcd.yichuanshen.de/#N4Igdg9gJgpgziAXAbVABwnAlgFyxMJZABgBpiBdUkANwEMAbAVxiRAEYQBfU9TXfIRTtyVWoxZsAOlIDidALYK63XiAzY8BIgCZR1es1aIQACVV9NgogGZ94o21kX1-LUOQAWe4cknOPJYC2ihk7GK+xhwuGsEeIuEGElEy8kp0APo0MW7WKHqJDn5mWTlWIch2hZFOZXFE3tXJbAFiMFAA5vBEoABmAE4QCkhkIDgQSOyBIANDk9TjSDrTs8OIemMTiDYrg2t2m0ieu3OIAKwLWwBsJ2tXl0gA7LdPD4gAHC8fbwCcX-eHRBTNSrV6A5YgvZId5vA4MLBgKJQCBMABGDFY1AAFjA6FA2JBESAFnQsAwCQRWFwKFwgA
\begin{tikzcd}
	1 \arrow[r] & \Gamma \arrow[r]             & \Lambda(G) \arrow[r]             & G \arrow[r]                                & 1 \\
	1 \arrow[r] & \Gamma_v \arrow[r] \arrow[u] & \Lambda_v(G) \arrow[r] \arrow[u] & G \arrow[r] \arrow[u, equal] & 1
\end{tikzcd}
\]
is exact.
We observe that the construction of $\Lambda(G)$ and $\Lambda_v(G)$ applies to every subgroup of $G$ whose action on $K$ preserve the subcomplexes $L_j$. In particular, it applies to the subgroup $H=\Pi_M(\partial \S)$. In this way, we can define groups $\Lambda(H)$ and $\Lambda_v(H)$ which fit in the same commutative diagram above (with $G$ replaced by $H$).\\

Recall that $K'=\A(\partial M)$ and $L_j'=\A(B_j)$ are not necessarily connected. 
Let $\hat{K}'=p^{-1}(K')$ and $\hat{L}'_j=p^{-1}(L'_j)$.
By Lemma \ref{lemma: homotopy tyoe  pe of A_X(A)}, the morphisms $\pi_1(L_j'\hookrightarrow K)$ and $\pi_1(K'\hookrightarrow K)$ are injective for every choice of basepoint. 
Therefore, inside $\hat{K}'$ (resp. $\hat{L}'_j$) we can find several copies of the universal covering of $K'$ (resp. of $L'_j$).
For every $v \in V(T)$, we set
\[
\hat{L}'_v = \hat{L}'_{j(v)}\cap \tilde{L}_v.
\]
It follows from our construction that the actions $\Lambda(H)\acts (\tilde{K}, \hat{K}')$ and $\Lambda_v(H)\acts (\tilde{L}_v,\hat{L}'_v)$ are well-defined simplicial actions on pairs. Moreover, one can deduce from Lemma \ref{lem: G acts K has orbits on L induced by H} that the action $\Lambda(G)\acts \tilde{K}$ has orbits in $\hat{K}'$ induced by $\Lambda(H)$, and similarly, the action $\Lambda_v(G)\acts \tilde{L}_v$ has orbits in $\hat{L}'_v$ induced by $\Lambda_v(H)$.\\

We proceed now with the proof of Proposition \ref{prop: extension of bounded coclasses}. Let $\epsilon > 0$. 
For every $j\in \{1,\dots, k\}$, we are given a coclass $\phi_j \in H^n(C^\bullet_b(L_j,L_j')^{G,H}_\alt)$. Let $f_j \in Z^n_b(L_j,L_j')^{G,H}_\alt$ be a representative of $\phi_j$ such that 
\[
\|f_j\|_\infty \leq \|\phi_j\|_\infty + \epsilon.
\]
For every $v \in V(T)$, we denote by $f_v \in Z^n_b(\tilde{L}_v)$ the pull-back of $f_j$ via the covering $p_v \colon \tilde{L}_v\rightarrow L_{j(v)}$. 
Since $f_v$ is both $G$ and $\Gamma_v$-invariant, it is $\Lambda_v(G)$-invariant too. 
Moreover, since $f_j$ vanishes on simplices supported on $L_j'$, it follows that $f_v$ vanishes on simplices supported on $\hat{L}_v'$, hence
\[
f_v\in Z^n_b(\tilde{L}_v, \hat{L}_v')^{\Lambda_v(G), \Lambda_v(H)}_\alt.
\]
In order to prove Proposition \ref{prop: extension of bounded coclasses}, it is sufficient to find a bounded cocycle \[f \in Z^n_b(\tilde{K}, \hat{K}')^{\Lambda(G), \Lambda(H)}_\alt\] which restricts to $f_v$ on each $C_n(\tilde{L}_v), v \in V(T)$, and is such that \[\|f\|_\infty \leq \max \bigl\{\|f_j\|_\infty \colon \; j\in \{1,\dots,k\} \bigr\}.\]

Let $s = (\sigma, (x_0, \dots, x_n)) \in C_n(\tilde{K})$ be an algebraic $n$-simplex. Since we work with alternating cochains, we can assume without loss of generality that $x_0,\dots,x_n \in \tilde{K}^0$ are the vertices of an $n$-simplex $\sigma$ of $\tilde{K}$ i.e. that $x_i\neq x_j$ for every $i\neq j$.
We already observe that $\tilde{K}$ and $\cup \hat{L}_j$ have the same 0-skeleton.
\begin{defn}
	\label{defn: simplicial path in cup L_j}
	Let $v_0,v_1 \in (\cup \hat{L}_j)^0$. A \emph{simplicial path in $\cup \hat{L}_j$} from $v_0$ to $v_1$ is a 1-dimensional submulticomplex of $\cup \hat{L}_j$ whose topological realization realizes a continuous path in $\bigl\lvert\cup \hat{L}_j\bigr\rvert$ from $v_0$ to $v_1$.
\end{defn}
In the previous definition, it is important that we do not allow generic edges in $\tilde{K}$, where every pair of distinct vertices can be connected by a unique edge (by Lemma \ref{lem: uniqueness of simplices in universal covering}).
Since $x_0,\dots,x_n \in \tilde{K}^0 = (\cup \hat{L}_j)^0$, we can consider simplicial paths in $\cup \hat{L}_j$ between them. 
\begin{defn}
	\label{defn: barycenter}
	A vertex $v \in V(T)$ is a \emph{barycenter} of $\sigma$ if, for every $i,j \in \{0,\dots, n\},\;i \neq j$, every simplicial path in $\cup \hat{L}_j$ from $x_i$ to $x_j$ has an edge $e$ such that $e$ is contained in $\tilde{L}_v$ and $e$ is not contained in $\hat{A}$.
\end{defn}
\begin{lemma}
	\label{lem: uniqueness of barycenter}
	Let $\sigma$ be a $n$-simplex of $\tilde{K}$. Then $\sigma$ has at most one barycenter. If $\sigma$ is supported in $\tilde{L}_v$, for some $v\in V(T)$, then $v$ is the barycenter of $\sigma$ if and only if every component of $\hat{A}$ contains at most one vertex of $\sigma$. If this is not the case, then $\sigma$ does not admit a barycenter. 
\end{lemma}
\begin{proof}
	The proof can be easily adapted from \cite[Lemma 9.5]{Fri17}.
\end{proof}
Let $v\in V(T)$ be fixed. We associate (quite arbitrarily) to $\sigma$ a $n$-simplex $\sigma_v$ with vertices $x_0',\dots, x_n'$ in such a way that the following conditions hold:
\begin{itemize}
	\item $\sigma_v$ is a $n$-simplex of $\tilde{L}_v$;
	\item if $x_i \in \tilde{L}_v$, then $x_i'=x_i$;
	\item if $x_i \notin \tilde{L}_v$, then there exists a unique component $C$ of $\tilde{L}_v\cap \hat{A}$ such that every simplicial path in $\cup \hat{L}_j$ joining $x_i$ with $\tilde{L}_v$ intersects $C$; in this case, we choose $x_i'$ to be any point of $C$.
\end{itemize}
By Lemma \ref{lem: uniqueness of simplices in universal covering}, since $\tilde{L}_v$ is connected, complete, minimal and aspherical, a simplex $\sigma_v$ with these properties exists and is also unique once we fix its vertices $x_0',\dots, x_n'$.
The simplex $\sigma_v$ is \emph{a projection of $\sigma$ on $\tilde{L}_v$}, and plays the role of the \emph{central simplex} of Kuessner's work \cite{Kue15}.
We denote by $s_v= (\sigma_v,( x_0',\dots, x_n'))$ the algebraic simplex associated to $\sigma_v$.
\begin{lemma}
	\label{lem: projections are well defined up to H-translates}
	Let $s_v=(\sigma_v,(x_0,\dots,x_n))$ and $\bar{s}_v=(\bar{\sigma}_v,(\bar{x}_0,\dots,\bar{x}_n))$ be algebraic simplices in $C_n(\tilde{L}_v)$, $v \in V(T)$. 
	Then $\sigma_v$ and $\bar{\sigma}_v$ are two projections of the same simplex $\sigma \in \tilde{K}$ on $\tilde{L}_v$ if and only if there exists $\lambda \in \Lambda_v(G)$ such that $\lambda\cdot \sigma_v = \bar{\sigma}_v$ and $\lambda\cdot x_i = \bar{x}_i$, for every $i\in \{0,\dots,n\}$.
\end{lemma}
\begin{proof}
	Assume that $\sigma_v$ and $\bar{\sigma}_v$ are two projections on $\tilde{L}_v$ of the same $n$-simplex $\sigma$ of $\tilde{K}$. We have already observed that the simplices $\sigma_v$ and $\bar{\sigma}_v$ depend only on their vertices. 
	Therefore, we need to show that, in the construction of $\sigma_v$ and $\bar{\sigma}_v$, the choice of two different points in the same connected component of $\hat{A}$ results in two simplices which are equivalent under the action of $\Lambda_v(G)$.
	By arguing inductively on the cardinality of $\{i \in \{0,\dots, n\} \,|\, x_i\neq \bar{x}_i \}$, we can assume that $x_i = \bar{x}_i$, for every $i\in \{1,\dots, n\}$, and that $x_0$ and $\bar{x}_0$ lie in the same connected component $C$ of $\hat{A}$. 
	By Lemma \ref{lem: universal coverings}, $C$ is the universal covering of some connected component of $A$.
	Hence, by Lemma \ref{lem: uniqueness of simplices in universal covering}, there exists a unique edge $e$ of $C$ with endpoints $x_0$ and $\bar{x}_0$. The edge $\pi(e) \in \A(\S)^1$ identifies an element of $G=\Pi_M(\S)$ in the following way.
	Let $\tau \colon [0,1]\rightarrow \S$ be a path in $\S$ from $\pi(x_0)$ to $\pi(\bar{x}_0)$ representing $\pi(e)$. 
	If $\tau$ is a loop, we set $g=\{\tau\} \in G$. Otherwise, we set $g = \{\tau, \bar{\tau}\} \in G$. By the unique lifting property of covering spaces, there exists a unique lift $\lambda \in \Lambda_v(G)$ of $g$ to the universal covering such that $\lambda \cdot x_i =\bar{x}_i$ for every $i\in \{0,\dots, n\}$.
	%By uniqueness, we have that $\lambda = \{e, \bar{e} \} \in \Pi_{\tilde{L}_v}(C)$ (see Subsection \ref{section: the group Pi_k(L)}).
	Therefore, by Lemma \ref{lem: uniqueness of simplices in universal covering}, we deduce that $\lambda\cdot \sigma_v=\bar{\sigma}_v$ (since $\tilde{L}_v$ is complete, minimal and aspherical, this is sufficient to be true on the 1-skeleton).
	The converse implication is clear, since elements of $\Lambda_v(G)$ preserve the connected components of $\hat{A}$.
\end{proof}
We now set, for every $v \in V(T)$,
\[
\hat{f}_v(s) = f_v(s_v).
\]
Since $f_v$ is $\Lambda_v(G)$-invariant, it follows from Lemma \ref{lem: projections are well defined up to H-translates} that $\hat{f}_v(s)$ is indeed well defined i.e. it does not depend on the choice of $s_v$. The following lemma rephrases in our setting the crucial observation in Lemma \ref{lema: a crucial observation}.
\begin{lemma}
	\label{lem: if v not a barycenter, then f_v = 0}
	If $v$ is not a barycenter of $\sigma$, then $\hat{f}_v(s) =0$.
\end{lemma}
\begin{proof}
	If $v$ is not a barycenter of $\sigma$, then there exists a connected component $C$ of $\hat{A}\cap \tilde{L}_v$ which contains at least two vertices of $\sigma_v$, say $x_i$ and $x_j$, with $i,j \in \{0,\dots, n\}$ (Lemma \ref{lem: uniqueness of barycenter}).
	By Lemma \ref{lem: uniqueness of simplices in universal covering}, there exists a unique edge $e$ of $C$ with endpoints $x_i$ and $x_j$.
	Using the very same argument of Lemma \ref{lem: projections are well defined up to H-translates}, we can prove that there exists an element $\lambda \in \Lambda_v(G)$ such that $\lambda\cdot \sigma_v = \sigma_v$, $\lambda \cdot x_i = x_j$, $\lambda\cdot x_j = x_i$ and $\lambda\cdot x_l = x_l$ for every $l \neq i,j$.
	Since $f_v$ is alternating and $\Lambda_v(G)$-invariant, we obtain that $\hat{f}_v(s)=f_v(s_v)=0$.
\end{proof}
We are ready to define the cochain $f$ as follows: for every algebraic simplex $s=(\sigma,(x_0,\dots, x_n))\in C_n(\tilde{K})$ we set
\[
f(s) = \sum_{v \in V(T)} \hat{f}_v(s).
\]
By Lemma \ref{lem: uniqueness of barycenter} and Lemma \ref{lem: if v not a barycenter, then f_v = 0}, the sum on the right hand is either empty or consists of a single term. From this we already deduce that
\[
\|f\|_\infty \leq \max \bigl\{ \|f_v\|_\infty:\; v \in V(T) \bigr\} = \max \bigl\{ \|f_j\|_\infty :\; j \in \{1,\dots, k\} \bigr\}.
\]
Since $f_v$ is alternating for every $v \in V(T)$, it follows that also $f$ is alternating. Moreover, it is clear from the construction that $f$ is $\Lambda(G)$-invariant, so
\[
f \in C^n_b(\tilde{K})^{\Lambda(G)}_\alt.
\]
Assume that the simplex $\sigma$ is supported on $\tilde{L}_v$, for some $v \in V(T)$. 
Then we can set $s_v=s$, and from Lemma \ref{lem: uniqueness of barycenter} and Lemma \ref{lem: if v not a barycenter, then f_v = 0}, we deduce that $f(s)=f_v(s)$. 
Therefore, we have that $f$ indeed restricts to $f_v$ on simplices supported on $\tilde{L}_v$.
In order to show that $f$ defines a cocycle, let $s=(\sigma, (x_0,\dots, x_{n+1}))$ be an algebraic $(n+1)$-simplex, where $\sigma$ is a $(n+1)$-simplex of $\tilde{K}$. 
Let $\sigma_v$ denote the projection of $\sigma$ on $\tilde{L}_v$, according to the same procedure that was described for $n$-simplices. It readily follows from the construction that $\partial_i(\sigma_v)$ is a projection of $\partial_i\sigma$ on $\tilde{L}_v$, hence
\begin{align*}
	\hat{f}_v(\partial s) &  = \sum_{i=0}^{n+1} (-1)^i \hat{f}_v((\partial_i\sigma),(x_0,\dots,\hat{x}_i,\dots,x_{n+1}))\\
	& =  \sum_{i=0}^{n+1} (-1)^i f_v((\partial_i\sigma)_v,(x_0',\dots,\hat{x}_i',\dots,x_{n+1}'))\\
	& =  \sum_{i=0}^{n+1} (-1)^i f_v(\partial_i(\sigma_v),(x_0',\dots,\hat{x}_i',\dots,x_{n+1}'))\\
	& = f_v(\partial s_v) = 0,
\end{align*}
where the last equality follows from the fact that $f_v$ is a cocycle.
Therefore we have that
\[
f(\partial s)=\sum_{v \in V(T)}\hat{f}_v(\partial s) = 0,
\]
so $f$ is a cocycle.

In order to conclude, we need to show that $f$ vanishes on simplices supported on $\hat{K}'$. Let $s= (\sigma,(x_0,\dots,x_n)) \in C_n(\hat{K}')$. 
By Lemma \ref{lem: if v not a barycenter, then f_v = 0}, it is sufficient to prove that $f_v(s_v)=0$, where $v\in V(T)$ is the barycenter of $\sigma$. 
We have that $\sigma$ is contained in some connected component $C$ of $\hat{K}'$. Observe now that $C$ either is contained $\tilde{L}_v$, or it does intersect some connected components of $\hat{A}$. In the first case, since $v$ is the barycenter of $\sigma$, we can take $\sigma=\sigma_v$. 
It follows that $\sigma_v \in \hat{K}'\cap \tilde{L}_v\subseteq \hat{L}_v'$.
In the second case, we can clearly construct a projection $\sigma_v$ of $\sigma$ on $\hat{L}_v'$. 
%Notice that projections are well defined only up to $\Lambda_v(G)$-translates (Lemma \ref{lem: projections are well defined up to H-translates}), but since the action $\Lambda_v(G)\acts \tilde{L}_v$ has orbits in $\hat{L}'_v$ induced by $\Lambda_v(H)$, we can of course assume that $\sigma_v \in \hat{L}'_v$.
In both cases, since $f_v$ vanishes on simplices supported on $\hat{L}_v'$, we get $f_v(s_v)=0$.

\section{Gromov equivalence theorem}
\label{sec: gromov equivalence theorem}
Let $(X,Y)$ be a pair of topological spaces. Let $A\subseteq Y$ be the union of some connected components of $Y$.
We start by introducing the following one-parameter family of seminorms on $H_n(X,Y)$, depending on $A$ (see \cite[Section 6.5]{Thu79}).
Let $\theta \in [0,\infty)$. We identify $C_n(X,Y)$ with the quotient of $C_n(X)$ by $C_n(Y)\subseteq C_n(X)$. We want to define a norm depending on $\theta$ and $A$ on the space of cycles $Z_n(X,Y)\subseteq C_n(X,Y)$. Let $c\in Z_n(X,Y)$. Since $\partial c=0$, it follows that every cycle $b\in C_n(X)$ representing $c$ satisfies $\partial b \in C_{n-1}(Y)$. 
There is a canonical projection $C_{n-1}(Y)\rightarrow C_{n-1}(A)$, and we denote by $\partial b|_{A}\in C_{n-1}(A)$ the image of $\partial b$ under this projection. We set
\[
\|c\|_1^A(\theta)= \inf \bigl\{ \|b\|_1 + \theta \|\partial b|_{A}\|_1 \;|\; b \in C_n(X) \text{ is a representative of } c \bigr \}.
\]
This indeed defines a norm on $Z_n(X,Y)$, that is equivalent to the usual norm $\|\cdot\|_1 = \|\cdot\|_1^A(0)$, and induces a quotient seminorm on $H_n(X,Y)$, still denoted by $\|\cdot\|_1^A(\theta)$.

The goal of this section is to give a proof of the following version of Gromov Equivalence Theorem \cite[Section 4.1]{Gro82}\cite[Theorem 5]{BBF+14}. This was originally stated in \cite{Kue15} without assumptions on higher homotopy.
\begin{thm}[Gromov Equivalence Theorem]
	\label{thm: equivalence theorem}
	Let $(X,Y)$ be a good pair and let $A\subseteq Y$ be the union of some connected components of $Y$. 
	If every connected component of $A$ has an amenable fundamental group, then the seminorms $\|\cdot\|_1^A(\theta)$ on $H_n(X,Y)$ coincide for every $\theta \in [0,\infty)$.
\end{thm}
The previous result admits the following equivalent formulation, which can be deduced as in \cite[Corollary 6]{BBF+14}.
\begin{cor}
	\label{cor: equivalence theorem, equivalent reformulation}
	Let $(X,Y)$ be a good pair of topological spaces. Let $A\subseteq Y$ be the union of some connected components of $Y$ and assume that every connected component of $A$ has an amenable fundamental group. Let $\alpha \in H_n(X,Y)$, $n\in \N_{\geq 2}$. 
	Then, for every $\epsilon \ge 0$, there exists an element $c\in C_n(X)$ with $\partial c \in C_{n-1}(Y)$ such that $[c]=\alpha \in H_n(X,Y)$, $\|c\|_1\leq \|\alpha\|_1 + \epsilon$ and $\|\partial c|_{A}\|_1< \epsilon$.
\end{cor}
Our proof of Theorem \ref{thm: equivalence theorem} uses the same construction of the mapping cone complex in \cite{BBF+14}.
We denote by $i_n \colon C_n(Y)\rightarrow C_n(X)$ the map induced by the inclusion $i\colon Y\hookrightarrow X$. We consider the complex $(C_\bullet(Y\hookrightarrow X), \bar{d}_\bullet)$, where
\[
C_n(Y\hookrightarrow X)= C_n(X)\oplus C_{n-1}(Y), \qquad \bar{d}_n(u,v)=(\partial_n u + i_{n-1}(v), -\partial_{n-1}(v)).
\]
We denote by $Z_\bullet(Y\hookrightarrow X)$ the corresponding cocycles and by $H_\bullet(Y\hookrightarrow X)$ its homology. For every $n\in \N$ and $\theta \in [0,\infty)$, we endow $C_n(Y\hookrightarrow X)$ with the norm
\[
\|(u,v)\|_1(\theta)= \|u\|_1 + \theta \|v\|_1.
\]
In a similar but not equivalent way, we endow $Z_n(Y\hookrightarrow X)$ with the norm
\[
\|(u,v)\|^A_1(\theta)= \|u\|_1 + \theta \|v|_{A}\|_1,
\]
which induces a seminorm, still denoted by $\|\cdot\|^A_1(\theta)$, on $H_\bullet(Y\hookrightarrow X)$.
Consider the chain map
\[
\beta_n \colon C_n(Y\hookrightarrow X) \rightarrow C_n(X,Y), \qquad \beta_n(u,v)=[u].
\]
The very same argument of \cite[Lemma 5.1]{BBF+14} can be used to prove the following.
\begin{lemma}
	\label{lem: mapping cone, isomorphism in homology}
	The map
	\[
	H_n(\beta_n) \colon (H_n(Y\hookrightarrow X), \|\cdot\|_1^A(\theta))\rightarrow (H_n(X,Y), \|\cdot\|_1^A(\theta))
	\]
	is an isometric isomorphism for every $\theta \in [0,\infty)$.
\end{lemma}
The dual chain complex of $(C_\bullet(Y\hookrightarrow X), \|\cdot\|_1(\theta))$ is the mapping cone complex associated to the chain map $i^\bullet \colon C^\bullet_b(X)\rightarrow C^\bullet_b(Y)$ (see Section \ref{sec: mapping cones}), i.e. the complex $(C^\bullet_b(Y\hookrightarrow X), \bar{\delta}^\bullet)$, where 
\[
C^n_b(Y\hookrightarrow X) = C^n_b(X)\oplus C^{n-1}_b(Y), \qquad \bar{\delta}^n(f,g)=(\delta^nf, -i^n(f)- \delta^{n-1}g).
\]
In our setting, $C^n_b(Y\hookrightarrow X)$ is endowed with the norm
\[
\|(f,g)\|_\infty(\theta) = \max\bigl\{\|f\|_\infty, \theta^{-1} \|g\|_\infty \bigr\}.
\]
We denote by $Z^\bullet_b(Y\hookrightarrow X)$ the corresponding cocycles and by $H^\bullet_b(Y\hookrightarrow X)$ the corresponding cohomology.
Similarly, the dual chain complex of $(Z_\bullet(Y\hookrightarrow X), \|\cdot\|^A_1(\theta))$ is $(Z^\bullet_b(Y\hookrightarrow X), \bar{\delta}^\bullet)$ equipped with the norm
\[
\|(f,g)\|^A_\infty(\theta) = \max\{\|f\|_\infty, \theta^{-1} \|g|_{A}\|_\infty \}.
\]
As above, this norm induces a seminorm, still denoted by $\|\cdot\|^A_\infty(\theta)$, in $H^\bullet_b(Y\hookrightarrow X)$.
Consider the chain map
\[
\beta^n \colon C^n_b(X,Y)\rightarrow C^n_b(Y\hookrightarrow X), \qquad \beta^n(f)=(f,0).
\]
In presence of good pairs, we can prove the following.
\begin{lemma}
	\label{lem: lem: mapping cone, isomorphism in cohomology}
	Assume that $(X,Y)$ is a good pair and that every component of $A$ has an amenable fundamental group.
	Then the map
	\[
	H^n(\beta^n) \colon (H^n_b(X,Y), \|\cdot\|_\infty) \rightarrow (H_n(Y\hookrightarrow X), \|\cdot\|_\infty^A(\theta))
	\]
	is an isometric isomorphism for every $\theta \in [0,\infty)$.
\end{lemma}
\begin{proof}
	By the very same argument of \cite[Proposition 5.3]{BBF+14} one can deduce that $H^n(\beta^n)$ is an isomorphism. In order to show that it is isometric, we need to exploit the fact that every connected component of $A$ has an amenable fundamental group. We denote by $B = Y\setminus A$ be the union of the remaining connected components of $Y$. Consider the chain map 
	\[
	\gamma^\bullet \colon C^\bullet_b(Y\hookrightarrow X)\rightarrow C^\bullet_b(X), \qquad (f,g)\rightarrow f,
	\]
	and the chain map 
	\[
	h^\bullet \colon C^\bullet_b(Y\hookrightarrow X)\rightarrow C^\bullet_b(B\hookrightarrow X),
	\]
	induced by the inclusion $j\colon (X,B)\rightarrow (X,Y)$. It is easy to verify that the composition $\gamma^n\circ h^n\circ \beta^n\colon Z^n_b(X,Y)\rightarrow Z^n_b(X)$ actually takes values in $Z^n_b(X,B)\subseteq Z^n_b(X)$ and in fact it coincides with the chain map 
	\[
	j^n \colon Z^n_b(X,Y)\rightarrow Z^n_b(X,B),
	\]
	induced by the inclusion $j$.
	Since the map $j$ induces an isometric isomorphism in bounded cohomology (Theorem \ref{thm: relative bounded cohomology and amenable connected components}) and since $H^n(\beta^n)$, $H^n(h^n)$ and $H^n(\gamma^n)$ are norm non-increasing, we can conclude that the isomorphism $H^n(\beta^n)$ is isometric for every $n\in \N_{\geq 2}$.
\end{proof}

\begin{proof}[Proof of Theorem \ref{thm: equivalence theorem}]
	We know that $(Z^\bullet_b(X,Y), \|\cdot\|_\infty)$ is the topological dual complex of \[(Z_\bullet(X,Y),\|\cdot\|_1)=(Z_\bullet(X,Y),\|\cdot\|^A_1(0))\] and $(Z^\bullet_b(Y\hookrightarrow X),\|\cdot\|^A_\infty(\theta))$ is the topological dual complex of $(Z_\bullet(Y\hookrightarrow X), \|\cdot\|^A_1(\theta))$. Therefore, from Lemma \ref{lem: lem: mapping cone, isomorphism in cohomology} and L{\"o}h Translation Principle \cite[Theorem 1.1]{Loh08}, we deduce that the map
	\[
	H_n(\beta_n)\colon (H_n(Y\hookrightarrow X), \|\cdot\|_1^A(\theta)) \rightarrow (H_n(X,Y), \|\cdot\|_1)
	\]
	is an isometric isomorphism. In conclusion, we know from Lemma \ref{lem: mapping cone, isomorphism in homology} that also the map
	\[
	H_n(\beta_n)\colon (H_n(Y\hookrightarrow X), \|\cdot\|_1^A(\theta)) \rightarrow (H_n(X,Y), \|\cdot\|^A_1(\theta))
	\]
	is an isometric isomorphism, and therefore we get the statement.
\end{proof}

\section{Proof of Theorem \ref{thm: additivity full boundary components}}
\label{sec: additivity}

Let $n \in \N_{\geq 2}$.
Let $M_1,\dots, M_k$ be triangulated $n$-manifolds.
Consider a pairing $(S_1^+,S_1^-),\dots,(S_h^+,S_h^-)$ of some of their boundary components and let $M$ be the manifold obtained by gluing $M_1,\dots, M_k$ along the pairing.
Assume that $M$ is path-connected.
Let $q_j\colon M_j\rightarrow M$ denote the quotient map.
We still denote by $M_j$ the image of $M_j$ in $M$ via $q_j$.
We denote by $\S$ the union of the manifolds $S_i^\pm$, $i\in \{1,\dots, h\}$, understood as a subset of $M$. 

The goal of this section is to deduce Theorem \ref{thm: additivity full boundary components} from a more general result stated in terms of good pairs. Recall that the union of two aspherical CW-complexes $X$ and $Y$ along a common subcomplex $Z$ is again aspherical, provided that $Z$ is itself aspherical and $\pi_1$-injective in $X$ and $Y$ \cite[Theorem 3.1]{Edm20}. Of course, in this case both the pairs $(X,Z)$ and $(Y,Z)$ are good. 
Building on this fact, it is immediate to deduce Theorem \ref{thm: additivity full boundary components} from the following.
\begin{prop}
	\label{prop: additivity full boundary components}
	If the following conditions hold: for every $j\in \{1,\dots, k\}$
	\begin{enumerate}
		\item[$(1)$] $(M,\partial M)$, $(M,M_j)$ and $(M_j, \partial M_j)$ are good pairs;
		\item[$(2)$] Every connected component of $\S$ has amenable fundamental group;
	\end{enumerate}
	then $\|M,\partial M\| = \|M_1,\partial M_1\| + \dots + \|M_k,\partial M_k\|$.
\end{prop}
We subdivide the proof of Proposition \ref{prop: additivity full boundary components} in two steps. In the first, the inequality
\begin{equation}
	\label{eq: subadditivity}
	\|M,\partial M\| \leq \|M_1,\partial M_1\| + \dots + \|M_k,\partial M_k\|
\end{equation}
is deduced as an application of Gromov Equivalence Theorem for good pairs (Corollary \ref{cor: equivalence theorem, equivalent reformulation}) and the Uniform Boundary Condition of Matsumoto and Morita \cite{Mat85}. Our proof follows closely the argument in \cite[Remark 6.2]{BBF+14}.
In the second, we show how the inequality
\begin{equation}
	\label{eq: superadditivity}
	\|M,\partial M\| \geq \|M_1,\partial M_1\| + \dots + \|M_k,\partial M_k\|,
\end{equation}
can be deduced by adapting the argument presented in Section \ref{sec: superadditivity}.

\subsection{Proof of subadditivity}
Since the pairs $(M_j,\partial M_j)$, $j\in \{1,\dots ,k\}$, are good, by Corollary \ref{cor: equivalence theorem, equivalent reformulation}, we can choose a real fundamental cycle $c_j\in C_n(M_j,\partial M_j)$ of $M_j$ such that
\[
\|c_j\|_1\leq \|M_j,\partial M_j\| + \epsilon, \qquad \|{\partial_n c_j}|_{\S_j}\|_1 < \epsilon.
\]
Let $c = c_1 + \dots + c_k\in C_n(M)$, where we identify every chain in $M_j$ with its image in $M$. 
Since $\partial c_j$ is the sum of real fundamental cycles of the boundary components of $M_j$, and since the gluing maps are orientation-reversing, it follows that $\partial c = \partial b + z$ for some $b\in C_n(\S)$ and $z \in C_{n-1}(\partial M)$.
Moreover, every connected component of $\S$ has amenable fundamental group, hence the singular chain complex of $\S$ satisfies the Uniform Boundary Condition in every positive degree \cite[Theorem 2.8]{Mat85}. Therefore, we can assume that the following inequalities
\[
\|b\|_1\leq \alpha \|\partial c - z\|_1 \leq \alpha k \epsilon
\]
hold for some universal constant $\alpha$.
The chain $c'=c-b$ is a relative fundamental cycle for $(M,\partial M)$. It follows that
\[
\|M,\partial M\| \leq \|c'\|_1\leq \|c\|_1 + \|b\|_1 \leq \sum_{j=1}^k \|M_j,\partial M_j\| + k\epsilon + \alpha k \epsilon,
\]
which implies (\ref{eq: subadditivity}), since $\epsilon$ is arbitrary.
\subsection{Proof of superadditivity}
We show how to adapt the argument presented in Section \ref{sec: superadditivity} to get the desired inequality.
In fact, under our assumptions, we know that the pairs $(M,\partial M)$ and $(M_j,B_j)$ are good, for every $j \in \{1,\dots, k\}$.
Therefore, in the proof of Proposition \ref{prop: superadditivity}, we can replace the bi-Lipschitz isomorphisms $\Psi^n$ and $\Psi_j^n$ in the diagram in Figure \ref{eq: diagram superadditivity} with the \emph{isometric} isomorphisms
\[
\Phi^n\colon H^n_b(\A(M),\A(\partial M)) \rightarrow H^n_b(M,\partial M),
\]
\[
\Phi_j^n\colon H^n_b(\A(M_j),\A(B_j)) \rightarrow H^n_b(M_j,B_j),
\]
from Theorem \ref{thm: relative mapping theorem, isometric version}.
Moreover, since the action of the group $G=\Pi_M(\S)$ on $\A(M)$ (resp. $\A(M_j)$) preserves the subcomplex $\A(\partial M)$ (resp. $\A(B_j)$), we can replace the bi-Lipschitz isomorphisms $I^n$ and $I^n_j$ in the diagram in Figure \ref{eq: diagram superadditivity} with the \emph{isometric} isomorphisms
\[
H^n(C^\bullet_b(\A(M),\A(\partial M))^G_\alt)\rightarrow H^n_b(\A(M),\A(\partial M)),
\]
\[
H^n(C^\bullet_b(\A(M_j),\A(B_j))^G_\alt)\rightarrow H^n_b(\A(M_j),\A(B_j)),
\]
from Proposition \ref{proposition: amenable - invariant resolutions}.
The same construction then applies and the inequality (\ref{eq: superadditivity}) easily follows.

%\begin{rem}
	%In \cite{BBF+14} the authors show that the inequality (\ref{eq: superadditivity}) is true under the additional assumption that \emph{every} boundary component of $M_1,\dots, M_k$ has amenable fundamental group.
	%Their proof is based on the following result on bounded cohomology: given a CW-pair $(X,A)$ such that every connected component of $A$ has amenable fundamental group, then the map $H^n_b(X,A)\rightarrow H^n_b(X)$ is an isometric isomorphism for every $n \in \N_{\geq 2}$ \cite[Theorem 2]{BBF+14}. If we restrict our attention to \emph{good} pairs, our Proposition \ref{prop: relative bounded cohomology and amenable connected components} is of course stronger, and their argument can be adapted \emph{verbatim} to prove inequality (\ref{eq: superadditivity}) in our setting.
%\end{rem}

\bibliographystyle{alpha}
\bibliography{multicomplexes}

\end{document}